\documentclass[12pt]{amsart}
\usepackage{amssymb}
\usepackage{amsmath}
\usepackage{eurosym}
\usepackage{paralist}
\usepackage{graphics}
\usepackage{epsfig}
\usepackage{tabu}
\usepackage{graphicx}
\usepackage{epstopdf}
\usepackage[colorlinks=true]{hyperref}
\usepackage{hyperref}
\usepackage{amsmath}
\usepackage{amsfonts}
\usepackage{fancyhdr}
\usepackage[margin=0.85in]{geometry}

\newcommand{\R}{\mathbb R}
\newcommand{\dha}{\dot H^1_a}
\newcommand{\bd}{\textbf{d}}
\newcommand{\eps}{\varepsilon}

\newcommand{\la}{\mathcal L_a}
\newcommand{\tq}{\tilde q}
\newcommand{\lsm}{\lesssim}
\newcommand{\nha}{H^1_a}
\newcommand{\Rd}{\mathbb R^d}
\newcommand{\ha}{\dot H^1_a}

\newcommand{\vertiii}[1]{{\left\vert\kern-0.25ex\left\vert\kern-0.25ex\left\vert #1 
    \right\vert\kern-0.25ex\right\vert\kern-0.25ex\right\vert}}
\newtheorem{theorem}{Theorem}[section]

\newtheorem{claim}[theorem]{Claim}

\newtheorem{corollary}[theorem]{Corollary}

\newtheorem{lemma}[theorem]{Lemma}

\newtheorem{proposition}[theorem]{Proposition}
\newtheorem{remark}[theorem]{Remark}

\DeclareMathOperator{\spa}{span}

\DeclareMathOperator{\area}{area}

\begin{document}
\title{The uniqueness of the ground state and the dynamics of nonlinear Sch\"odinger equation with inverse square potential}
\author[K. Yang]{Kai Yang}
\address[Kai Yang]{School of Mathematics, Southeast University, Nanjing, P.R. China 211189}
\email{yangkai99sk@gmail.com,kaiyang@seu.edu.cn}
\thanks{KY is supported in part by the Jiangsu Shuang Chuang Doctoral Plan and the Jiangsu Provincial Scientific Research Center of Applied Mathematics under Grant No. BK20233002}

\author[C. Zeng]{Chongchun Zeng}
\address[Chongchun Zeng]{School of Mathematics, Georgia Institute of Technology, Atlanta, GA 30332}
\email{zengch@math.gatech.edu}
\thanks{CZ is supported in part by the National  Science Foundation grant  DMS-2350115.} 

\author[X. Zhang]{Xiaoyi Zhang}
\address[Xiaoyi Zhang]{Department of Mathematics, University of Iowa, Iowa City, IA 52242}
\email{xiaoyi-zhang@uiowa.edu}
%\thanks{XZ was supported by Simons collaboration grant. } 
\maketitle

\begin{abstract}
In this paper, we first provide an alternative proof of the uniqueness of the ground state solution for NLS with inverse square potential and power nonlinearity $|u|^pu$ for all $0<p<\frac 4{d-2}$ in dimensions $d\ge 3$. While the uniqueness result was previously obtained by Mukherjee-Nam-Nguyen using a functional analytic approach, our method successfully adapts the classical ``shooting method'' to the case with the singular potential, accompanied by a more detailed analysis on the ground state equation. Based upon this result and a comprehensive spectral analysis, we construct the stable/unstable manifolds of the ground state standing wave solutions and classify solutions on the mass-energy level surface of the ground state in dimensions $d=3, 4, 5$. 

\end{abstract}

%\begin{center}
%{\scriptsize Keywords: dynamics, inverse square potential, NLS, energy
%critical, ground state solution.}
%\end{center}

%\address{Department of Mathematics, University of Iowa, Iowa City, IA 52242}

%\section{Preliminaries}

\section{Introduction} Let $a\in \big(-(\tfrac{d-2}2)^2,0\big)$ and $\la=-\Delta+\frac a{|x|^2}$, $0<p<\frac 4{d-2}$, and we consider the initial value problem 
\begin{align*}
\begin{cases}
(i\partial_t-\la)u+|u|^p u=0, \; (t,x)\in \R\times \R^d, \\
u(0,x)=u_0\in H^1(\R^d), 
\end{cases}\tag{NLS$_a$}
\end{align*}
for the complex-valued function $u:\R\times\R^d\to \mathbb C$. 
%Here $H^1(\R^d)$ is the non-homogeneous Sobolev space $\dot H^1\cap L^2$. 
For $a>-(\tfrac{d-2}2)^2$, the bilinear form $\langle \la f, f\rangle$ is positive definite due to the sharp Hardy's inequality, and the norm defined by
\[
\sqrt{\langle\la f,f\rangle}=\|(\la)^{\frac 12} f\|_{L_x^2}
\]
is equivalent to the usual Sobolev norm $\|f\|_{\dot H^1}=\|\nabla f\|_{L_x^2}$. We thus use $\dha(\R^d)$ to denote the Hilbert space $\dot H^1(\R^d)$ equipped with this equivalent norm and use $\nha(\R^d)$ to denote the space $H^1(\R^d)$ with the equivalent norm $\sqrt{\langle(\la+1)f, f\rangle}$. 

Throughout this paper, a solution  to (NLS$_a$) is referred to as a ``strong solution'', by which we mean a function $u(t,x)\in C_t(H^1)$ satisfying the integral equation 
\begin{align*}
u(t)=e^{-it\la} u_0+i\int_0^t e^{-i(t-s)\la}(|u|^p u)(s) ds.
\end{align*}
%with certain finite spacetime norm. 
Constructing a local strong solution using Strichartz methodology imposes further restrictions on $a$, one can see  \cite{KMVZ17,  KMVZZ17, LMM18} for the discussion of the local well-posedness and the explicit condition on $a$.

On the time of existence, the solution of (NLS$_a$) satisfies the conservation of mass and energy: 
\begin{align*}
M(u(t))&=\int_{\R^d} |u(t,x)|^2 dx, \\
E(u(t))&=\frac 12 \int_{\R^d} |\nabla u(t,x)|^2 +\frac a{2|x|^2}|u(t,x)|^2 -\frac 1{p+2} |u(t,x)|^{p+2} dx. 
\end{align*}
The cases $p=\frac 4d$ and $p=\frac 4{d-2}$ are called mass-critical and energy-critical, respectively as the rescaling of the equation 
\begin{align*}
u(t,x)\to u_{\lambda}(t,x)=\lambda^{-\frac 2p}u(\tfrac t{\lambda^2}, \tfrac x\lambda)
\end{align*}
also leave the mass and energy invariant. 

 In the energy-critical case, results concerning scattering below the ground state threshold and the dynamics of solutions on the energy surface of the ground state can be found in previous works \cite{Y20, Y21, YZ23, YZZ22}. The ground state in this setting is identified with the optimizer of the Sobolev embedding $\dha (\R^d)\subset L^{\frac{2d}{d-2}}(\R^d)$ and it also serves as a static solution of (NLS$_a$). 

 In the inter-critical regime $\frac 4d<p<\frac 4{d-2}$, which is the focus of this paper, the standing wave solution $e^{it}Q(x)$, where $Q(x)$  denotes the ground state, serves as the ``minimal'' threshold for non-scattering solutions. Here, the ground state $Q(x)$ is characterized as the maximizer of the functional 
\begin{align}\label{4}
J(f)=\frac{\|f\|_{L_x^{p+2}}^{p+2}}{\|f\|_{L_x^2}^{2-\frac p2(d-2)}\|f\|_{\dha}^{\frac{pd}2}}, \quad f \in H^1 (\R^d).
\end{align}
The minimality of the threshold can be further illustrated by the following result obtained in \cite{KMVZ17,  LMM18}:
on the mass level set of the ground state, i.e., $M(u)=M(Q)$, if the energy and kinetic energy are below those of the ground state, or more specifically $E(u)<E(Q)$, $\|u_0\|_{\dot H^1_a}<\|Q\|_{\dot H^1_a}$, then the solution scatters to $0$ when $t\to \pm \infty$.

    The purpose of this paper is twofold. First, we revisit the problem of the uniqueness of the ground state. For the classical nonlinear Schr\"odinger equation (NLS) without potential, the uniqueness was first established in the seminal work of Coffman \cite{C72} for the cubic NLS in three dimensions. This result was later extended to general dimensions and more general nonlinearities by McLeod and Serrin \cite{MS81, MS87} with certain restrictions. The widely cited work of Kwong \cite{Kwong} removed the restriction and further extended the analysis to general domains. Further discussions on this topic can be found in \cite{CJ93,KL92,KZ91,Tang03} and the references therein. While the underlying strategy in these works—the shooting method framework, which establishes a topological structure of the initial data of the reduced ODE by carefully analyzing trajectories initiated from different starting points—may appear conceptually straightforward, the detailed analysis is highly nontrivial. In particular, extending these techniques to cases involving singular potentials at the origin introduces significant complications.

A recent work by Mukherjee, Nam, and Nguyen \cite{MNN21} provides an elegant proof of the uniqueness for the NLS with inverse-square potential, using a generalized Pohozaev-type identity. Their approach follows the spirit of earlier works \cite{SW13} by Shioji and Watanabe and \cite{Y91} by Yanagida,  and includes several key innovations. Nevertheless, as noted in \cite{MNN21}, it remains a compelling question whether the classical shooting method can be adapted to overcome the challenges posed by the singular potential.

In the first part of this paper, we present an alternative proof of the uniqueness of the ground state following primarily the shooting method. Based on some blow-up analysis and the local invariant manifold theory, we first derive the exact asymptotic behavior of solutions to the ground state ODE near both the origin and the infinity. Subsequently we establish a smooth parametrization of trajectories corresponding to functions which are locally $H^1$  near the origin. The uniqueness result then follows from a detailed study of the dynamics of these trajectories, which reveals insights into the topological structure among the associated parameters.
 %Our approach involves deriving the exact asymptotic behavior of solutions to the ground state ODE and using this to smoothly parametrize trajectories which correspond to functions locally $H^1$  near the origin. 
    
  To state this result, we first note that a standard argument allows us to reduce the problem to finding positive, radial solutions of the associated Euler–Lagrange equation (see Section \ref{S:groundS} for further discussion).
  \begin{align}\label{2}
\la q+q-|q|^p q=0. 
\end{align}
  So our first result on the uniqueness of the ground state reads: 
  \begin{theorem}\label{thm: ground state}
Let $d\ge 3$, $a\in \big(-(\tfrac{d-2}2)^2,0\big)$, $0<p<\frac 4{d-2}$, $\beta=\sqrt{(d-2)^2+4a}$. Then there exists a unique radial positive solution $Q\in H^1(\R^d)$ to \eqref{2}. Moreover, $Q(x)$ satisfies 

1) $Q(x)\in C^{\infty}(\R^d \setminus \{0\})$ and the radial derivative $Q_r(r)<0$. 

2) There exists $b_0>0$ such that 
$$
\lim_{r\to 0^+} r^{\frac{d-2-\beta}2}Q(r)=b_0, \ \lim_{r\to 0^+} r^{\frac{d-\beta}2}Q_r(r)=-\tfrac{d-2-\beta}2 b_0. 
$$

3) There exists $c_0>0$ such that 
$$
\lim_{r\to \infty} r^{\frac{d-1}2}e^{r} Q(r)=c_0, \ \lim_{r\to \infty} r^{\frac{d-1}2}e^{r}Q_r(r)=-c_0. 
$$

4) Let $C_{GN}$ be the sharp constant in the Gagliardo-Nirenberg inequality 
\begin{align}\label{3}
\|f\|_{L_x^{p+2}}^{p+2}\le C_{GN} \|f\|_{L_x^2}^{\frac{4-p(d-2)}2}\|f\|_{\dha}^{\frac{pd}2}. 
\end{align}
The ``=" holds if and only if there exist $c\in \mathbb C$ and $\lambda>0$ such that 
\begin{align}
f(x)=cQ(\lambda x). 
\end{align}

%5) More properties of $Q$ and the linearized operator can be found in {\color{red} Section 2 and 3}. 
\end{theorem}

    The second main result of this paper is the construction of the stable and unstable manifolds of the ground state, as well as a characterization of the dynamics of solutions on the energy-mass level surface containing the ground state. In the energy-critical case, a similar analysis was carried out by the authors in \cite{YZZ22}. When $a>0$, the inverse square potential is repulsive, Miao, Murphy and Zheng obtain the threshold scattering for cubic NLS$_a$ in \cite{MMZ23}. One is referred to  \cite{Bourgain, CFR22, DHR08, DM08, DM08a, DR10, KM06, LZ09, LZ11, SZ23, MMMZ25} for similar studies in the classical case without potential. 

 Due to regularity constraints, our approach requires the nonlinearity exponent in the inter-critical regime to satisfy $p\ge 1$, which imposes a restriction on the spatial dimension: specifically, $d=3,4,5$. The additional requirement  $p\ge 1$ is mainly used in the construction of the stable manifold in Subsection~\ref{SS:stableM}.
    
   \begin{theorem} Let $d=3,4,5$ and $p\in (\frac 4d, \frac 4{d-2})\cap[1,\infty)$. Let $0>a>\tfrac{(d-2)^2}4(-1+\tfrac{p^2}{(p+1)^2})$. 
   
   $\bullet$ There exist two solutions $e^{it}Q_{\pm}(t, x)$ of NLS$_a$ satisfying 
   \begin{gather*}
   \lim_{t \to +\infty} \|Q_{\pm}-Q\|_{H^1}=0, \    M(Q_{\pm})=M(Q), \ E(Q_{\pm})=E(Q),\\
 \|Q_+(t)\|_{\dot H^1_a}>\|Q\|_{\dot H^1_a}, \ \ \|Q_-(t)\|_{\dot H^1_a}<\|Q\|_{\dot H^1_a}.
   \end{gather*}
Such solutions are unique up to time translations. Moreover, there exist $C, c_0>0$ such that    
\[
\|Q_{\pm}-Q\|_{H^1} \le C e^{- c_0t}, \ \forall t\ge 0.
\]
Finally,   $e^{it} Q_-$ is a global solution and scatters as $ t\to -\infty$. 
%   $Q_+$ blows up at finite time under certain conditions. ( {\color{red} to be checked})
   
  $\bullet$ Conversely, if a solution $u(t,x)$ of NLS$_a$ satisfies 
   \begin{align*}
   M(u)=M(Q), \ E(u)=E(Q), \ \|u_0\|_{\dot H^1_a}<\|Q\|_{\dot H^1_a}, \ \|u\|_{S^1([0,\infty))}=\infty,
   \end{align*}
   then there exists a unique pair $(\theta, T)\in \mathbb S^1\times \mathbb R$ such that 
\[
   u(t,x)=e^{i(\theta+t)}Q_-(t+T,x).
   \]
% In the opposite time direction, $u$ exists globally and scatters.  
If instead $\|u\|_{S^1((-\infty,0])}=\infty$, then $u$ must  scatter on $[0, \infty)$ and for some $\theta$ and $T$, 
   \[u(t,x)=e^{i(\theta+t)}\bar Q_-(-t+T,x).\]
      Here $S^1([0,\infty))$ is the Strichartz space defined in \eqref{snorm} in Section \ref{S:manifold}. 
   
   $\bullet$ \  If a solution $u(t,x)$ of NLS$_a$ on $ [0, \infty)$ satisfies 
   \begin{align*}
   M(u)=M(Q), \ E(u)=E(Q), \ \|u_0\|_{\dot H^1_a}>\|Q\|_{\dot H^1_a}, \ xu_0\in L^2(\R^d) \mbox{ or } u_0\in H^1_{rad}(\R^d),
   \end{align*}
   then there exists a unique pair $(\theta, T)\in \mathbb S^1\times \mathbb R$ such that 
   \begin{align*}
   u(t,x)=e^{i(\theta+t)}Q_+(t+T,x).
   \end{align*}
If instead $u$ is a solution on $(-\infty,0]$ obeying the same norm conditions, then $u(t,x)=e^{i(\theta+t)}\bar Q_+(-t+T,x)$. 
     \end{theorem}
     
   \begin{remark} \label{R:1.3}
   In the above, we constructed the 1-dimensional stable/unstable manifold of $Q$ 
    \[
    \{Q_{\pm}(T)\, , \ \forall\ T \in \mathbb{R}\} \ \text{ or } \ \{\bar Q_{\pm}(T)\, , \ \forall\ T \in \mathbb{R}\}.
    \] 
   
   In fact, due to the rotational and scaling symmetry of the equation NLS$_a$, the ground state manifold forms a cylindrical two-dimensional manifold 
   %with the two directions representing  symmetry: 
     \[
    \mathcal{M} = \{ \mathcal G_{\theta, \lambda} Q=e^{i\theta}\lambda^{-\frac 2p}Q(\tfrac x\lambda)\}.
     \]  
   Each point in the manifold generates a standing wave solution $e^{i\frac t{\lambda^2}}\mathcal G_{\theta,\lambda}Q \in \mathcal{M}$ to NLS$_a$. 
   The stable/unstable manifold of $\mathcal G_{\theta,\lambda}Q$ and the corresponding nonlinear solutions on the mass-energy level surface of $\mathcal G_{\theta,\lambda}Q$ are obtained by moduloing these symmetries. More specifically, the nonlinear solutions on the stable manifold $\{e^{i\theta}\lambda^{-\frac 2p}Q_{\pm}(T, \frac x\lambda), \ \forall T\}$ of $\mathcal G_{\theta,\lambda}Q$ are $e^{i(\theta+\frac t{\lambda^2} )}\lambda^{-\frac 2p}Q_{\pm}(\frac t{\lambda^2}+T,\frac x\lambda)$. The unstable manifold is obtained via time reversibility. 
%    or $e^{i(\theta+\frac t{\lambda^2} )}\lambda^{-\frac 2p}\bar Q_{\pm}(-\frac t{\lambda^2}+T,\frac x\lambda)$. 
\end{remark}

  The remainder of the paper is organized as follows. In Section \ref{S:groundS}, we prove the existence of the ground state as a solution to an ODE. Moreover, for this ODE, we derive detailed asymptotic behavior of solutions locally in $H^1(\R^d)$ near the origin and at infinity. We then provide a smooth parametrization of these local $H^1$ solutions based on their asymptotics and classify them into three categories. Section \ref{S:uniqueness} is dedicated to analyzing the solutions of the linearized ODE and establishing the uniqueness of the ground state by examining the topological structure of the set of parameters. The subsequent sections focus on constructing the stable and unstable manifolds and investigating the dynamics of solutions on the mass-energy surface of the ground state $Q$. Additional properties of the linearized operators are provided in the Appendix.

\section{Existence of Ground States and the parametrization of local $H^1$ solutions} \label{S:groundS}

In this section, under the assumptions of Theorem \ref{thm: ground state}, we establish the existence of the ground state and analyze the asymptotics of local $H^1$ solutions. We then provide a smooth parametrization of the local radial $H^1$ solutions of \eqref{2} based on their asymptotics near the origin, and classify them into three categories. Although some results in this section such as the existence have appeared elsewhere in \cite{MNN21}, our approach is somewhat different and we include full details for the sake of completeness.

\subsection{The existence of the ground state solution}

The result about the existence of the ground state is stated as follows: 

\begin{lemma}\label{lm: existence}
(1) There exists a maximizer $R(x)$ of the functional $J$  defined in \eqref{4}.  

(2) Any maximizer $R(x)$ of $J$ is sign definite (either $R>0$ or $R<0$),  $R\in H^1_{rad}(\R^d)\cap C^{\infty}(\R^d\setminus \{0\})$ and  $R(x)$ realizes the best constant in the sharp Gagliardo-Nirenberg inequality \eqref{3}. 

(3) There exist $c_1, c_2>0$ such that 
\[
Q(x)=sgn(R)\bigl(\frac{c_2}{c_1}\bigr)^{\frac 1p}R\bigl(\frac x{\sqrt{c_1}}) >0
\]
satisfies $Q_r<0$ and solves the Euler-Lagrange equation:
\begin{align}\label{3.1}
\la Q+Q-Q^{p+1}=0. 
\end{align}
Moreover, we have the non-negativity of the operator $L_1:=\la+1-(p+1) Q^p$ on a codim 1 subspace in the following sense: 
\begin{align}\label{5.10}
\langle L_1 g, g\rangle\ge 0,\ \forall g\perp_{\la+1}  Q, \textit{ i.e. }\langle (\la+1) Q,g\rangle=0. 
\end{align}

\end{lemma}

\begin{proof}
We first obtain the existence of the maximizer of $J$. From the Gagliardo-Nirenberg inequality, there exists a constant $C>0$ such that $J(f)\le C$ for any $f\in H^1(\R^d)$, $f\neq 0$. Denote the sharp constant
$$
C_{GN}=\sup_{f\neq 0, f\in H^1} J(f).
$$
We can take a maximizing sequence $f_n\in H^1(\R^d)$ such that $J(f_n)\to C_{GN}$. Note for $a<0$, 
the radial rearrangement decreases the $\dha$ norm. We can apply the radial rearrangement and assume $f_n(r)$ is spherically symmetric and monotone decreasing in the radial variable $r$. Now let 
\begin{align*}
c_n=\frac 1{\|f_n\|_{\dha}}, \; \lambda_n=c_n \|f_n\|_{L_x^2}  \textit{ and } g_n(x)=c_n\lambda_n^{\frac {d-2}2}f_n(\lambda_n x),
\end{align*}
it is easy to see that $\{g_n\}$, while radial and non-increasing in $r$, remains a maximizing sequence. Moreover 
\begin{align*}
\|g_n\|_{L_x^2}=\|g_n\|_{\dha}=1, \quad \|g_n\|_{L_x^{p+2}}^{p+2} \to C_{GN} \; \text{ as } n \to \infty. 
\end{align*}
From the compact embedding $H_{rad}^1(\R^d)\hookrightarrow L^{p+2}(\R^d)$, we can extract a subsequence (which we still denote as $g_n$) and an $H^1$-function $R(x)$ such that 
\begin{align}\label{5}
\begin{cases}
g_n(x)\to R(x), \textit{ strongly in } L^{p+2}(\R^d),\\
g_n(x) \rightharpoonup R(x), \textit{ weakly in } H^1(\R^d). 
\end{cases}
\end{align}
Hence, 
\begin{align}\label{5.1}
\|R\|_{L_x^{p+2}}^{p+2}= \lim_{n\to \infty} \|g_n\|_{L_x^{p+2}}^{p+2}=C_{GN}, \; \|R\|_{L_x^2}\le 1,\; \|R\|_{\dha}\le 1. 
\end{align}
As $J(R)\le C_{GN}$, applying the bound from \eqref{5.1} we immediately have
\begin{align}\label{6}
 \|R\|_{L_x^2}= 1,\; \|R\|_{\dha}= 1,  
 \end{align}
% This together with \eqref{5.1} implies 
 hence 
 \begin{align}
 %g_n(x)\to Q(x) \textit{ in }H^1(\R^d) \textit{ and } 
 J(R)=C_{GN}. 
 \end{align}
 The existence of the maximizer is proved. 
 
Let $R(x)$ be any maximizer. Due to the spherical rearrangement explained above, $R(x)$ must be radially symmetric and decreasing with respect to the radial variable $r$. It is also standard in the calculus of variations that such maximizer is sign definite. The equation for $R$ can be derived directly from $\langle J'(R), g\rangle=0$ for any test function $g \in H^1 (\R^d)$. By direct computation, we find 
\begin{align}\label{5.0}
\langle J'(f), g\rangle =c_0(f)\int_{\R^d}(\la f+c_1(f)f-c_2(f)|f|^pf)(x) g(x) dx, 
\end{align}
where 
\begin{align}\label{5.2}
\begin{cases}
c_0(f)=-\tfrac{pd}2\tfrac{\|f\|_{L_x^{p+2}}^{p+2}}{\|f\|_{L_x^2}^{2-\frac p2(d-2)}}, \; 
c_1(f)=\frac{4-p(d-2)}{pd}\tfrac{\|f\|_{\dha}^2}{\|f\|_{L_x^2}^2},\\
c_2(f)=\frac{2(p+2)}{pd}\frac{\|f\|_{\dha}^2}{\|f\|_{L_x^{p+2}}^{p+2}}. 
\end{cases}
\end{align}
Moreover, 
\begin{align}\label{5.3}
&\langle J''(f)h, g\rangle \\
=&\langle c_0'(f), h\rangle \int_{\R^d}\langle \la f+c_1(f)f-c_2(f)|f|^pf,  g\rangle dx \notag \\
&+ c_0(f)\int_{\R^d}\langle \la+c_1(f)-(p+1)c_2(f)|f|^p)h, g\rangle dx\notag\\
&+c_0(f)\langle c_1'(f), h\rangle \int_{\R^d} fg dx-c_0(f)\langle c_2'(f), h\rangle \int_{\R^d}|f|^pf gdx. \notag
\end{align}
From \eqref{5.0}, $R$ satisfies the Euler-Lagrange equation
\begin{align}\label{6.1}
 \la R+c_1(R)R-c_2(R) |R|^{p}R=0.
 \end{align}
To eliminate the constants, we define
 \[
 Q(x)=sgn R(x)\Big(\frac{c_2(R)}{c_1(R)}\Big)^{\frac 1p}R\Big(\frac x{\sqrt{c_1(R)}} \Big).
 \]
 It is easy to check $Q(x)$ satisfies \eqref{3.1} with
 \begin{align}\label{5.6}
 c_1( Q)=c_2(Q)=1. 
 \end{align}
To prove the non-negativity property \eqref{5.10}, we take $f= Q$ and $h=g$ in \eqref{5.3}, using \eqref{3.1} and \eqref{5.6} we obtain 
 \begin{equation}\label{5.7} \begin{split}
 \langle J''( Q)g, g\rangle=&c_0( Q)\langle L_1 g, g\rangle +c_0( Q)\langle c_1'( Q), g\rangle \langle Q, g\rangle\\
 &-c_0( Q) \langle c_2'( Q), g\rangle \langle (\la+1) Q,g\rangle. 
 \end{split} \end{equation}
Substituting 
 \begin{align}\label{5.8}
\langle c_1'( Q), g \rangle=\tfrac{4-p(d-2)}{pd}\biggl[\tfrac 2{\| Q\|_{L_x^2}^2}\langle (\la+1) Q,g\rangle -\tfrac{2\| Q\|_{H^1_a}^2}{\| Q\|_{L_x^2}^4}\langle  Q,g\rangle\biggr],
 \end{align}
 we finally obtain that, if $g\perp_{\la+1} Q$, then 
 \begin{align}
 \langle J''( Q)g, g\rangle =c_0( Q)\bigl[ \langle L_1 g, g\rangle -c_3\langle  Q, g\rangle^2\bigr], \ c_3=\tfrac{4-p(d-2)}{pd}\tfrac{2\| Q\|_{H^1_a}^2}{\| Q\|_{L_x^2}^4}.
 \end{align}
 Recalling $c_0( Q)<0$ and $J''(Q)\le 0$, we immediately have \eqref{5.10}. 
  Lemma \ref{lm: existence} is proved. 
 \end{proof}

 \subsection{Asymptotics of local $H^1_{rad}$ solutions of equation \eqref{2} as $r\to +\infty$}\label{S:22}

In this subsection, we derive the precise asymptotic behavior of all local radial $H^1$ solutions of \eqref{2} as $r\to \infty$ via the method of invariant manifolds. In radial coordinates, equation \eqref{2} takes the form 
\begin{align}\label{12}
-q_{rr}-\tfrac{d-1}r q_r+\tfrac a{r^2} q+q-|q|^{p} q=0. 
\end{align}

In the following subsections, we will also require the asymptotics of  the linearized equation 
\begin{align}\label{r1}
-q_{rr}-\tfrac{d-1}r q_r+\tfrac A{r^2}q+q-(p+1)Q^p q=0,
\end{align}
where $A\in \R$, and $Q(r)$ denotes a ground state solution of \eqref{12}. As we will show, all radial local $H^1$ solutions of \eqref{12} exhibit some exponential decay at infinity. To treat both cases in a unified framework, we consider the more general equation
\begin{align}\label{r2}
-q_{rr}-\tfrac{d-1}r q_r+q+f(\tfrac 1r, q)=0,
\end{align}
 where the function $f(\tau, q)$ corresponds to either $f_1$ or $f_2$, representing equations \eqref{12}, \eqref{r1}, respectively:
\begin{align} \label{f12}
f_1(\tau, q)=a\tau^2 q-|q|^{p} q, \quad f_2(\tau, q)=A\tau^2 q-(p+1)Q(\tfrac 1\tau)^p q.
\end{align}
Both $f_1$ and $f_2$ (after $Q$ is proved to decay exponentially) share the following property:
\begin{align}\label{r3}
f(\tau, q)\in C^1, \quad f(\tau, 0)=0, \quad Df(0, 0)=0.
\end{align}

We now obtain the precise exponential decay of solutions that are initially known only to decay in Lemma \ref{lm: existence}. As will be shown later in Corollary \ref{cor: infinity}, this decay property is readily ensured by the assumption that $q\in H^1_{rad}(\R^d)$ near infinity.  

\begin{lemma}\label{lm: infinity}
Let $f(\tau, q)$ be defined through \eqref{f12}-\eqref{r3} and $q(r)$ be a solution of \eqref{r2} satisfying
\begin{align*}
\lim_{r\to \infty}q(r)=\lim_{r\to \infty}q_r(r)=0. 
\end{align*}
Then there exists a constant $c\neq 0$ such that 
\begin{align}\label{18}
\lim_{r\to \infty}  r^{\frac{d-1}2}e^{r}q(r)=c, \quad \lim_{r\to \infty}  r^{\frac{d-1}2}e^{r}q_r(r)=-c.
\end{align}
Moreover for
\begin{align}\label{19}
q_1(r)=\tfrac 2p q(r)+rq_r(r),
\end{align}
we have  
\begin{align}\label{20}
\lim_{r\to \infty}  r^{\frac{d-3}2}e^{r} q_1(r)=-c, \ \lim_{r\to \infty}  r^{\frac{d-3}2}e^{r} (q_1)_r(r)=c. 
\end{align}
\end{lemma}

\begin{proof}

We define
\begin{align}\label{12.1}
v_1(r)=q(r)+q_r(r),\; v_2(r)=q(r)-q_r(r),\textit{ and } \; \tau=\frac 1r,
\end{align}
and use this to transform \eqref{r2} into the system

\begin{align}\label{13}
\begin{cases}
\tfrac d{dr}v_1=v_1-\tfrac{d-1}2\tau(v_1-v_2)+f(\tau, \tfrac{v_1+v_2}2) \\
\tfrac d{dr}v_2=-v_2+\tfrac{d-1}2\tau(v_1-v_2)-f(\tau, \tfrac{v_1+v_2}2)\\
\tfrac d{dr}\tau=-\tau^2. 
\end{cases}
\end{align}
Clearly, $(\tau, v_1, v_2)=(0,0,0)$ is an unstable equilibrium state with center, unstable, and stable directions corresponding to the  $\tau$, $v_1$, and $v_2 $ components, respectively. In view of \eqref{r3}, there exists a $C^1$ center stable manifold $W_{cs}^\infty$  expressed as the graph of a function $\phi^\infty (\tau, v_2)$ \footnote{In the case where $f=f_2$, $\phi^\infty$ is linear in $v_2$.}
\begin{align}\label{21.1}
v_1=\phi^\infty (\tau, v_2), \; \phi^\infty \in C^{1}(\R^2),\; \phi^\infty(0,0)=0, \; D\phi^\infty(0,0)=0.
\end{align}
While the center-stable manifold $W_{cs}^\infty$ is not uniquely defined in general, its intersection with the set $\{\tau\ge0\}$ is unique due to the positive invariance under the ODE flow \eqref{13} of neighborhoods of $0$ in $W_{cs}^\infty$. An orbit of \eqref{13} belongs to $W_{cs}^\infty$ iff 
%contains all orbits of \eqref{13}
it converges to the origin as $r\to +\infty$, including the trivial solution $(\tau=\frac 1r, 0, 0)$. Thus,
\begin{align}\label{21.2}
\phi^{\infty}(\tau, 0)=0.
\end{align}
Applying the mean value theorem and using \eqref{21.1}, we estimate
\begin{align}
v_1&=\phi^{\infty}(\tau, v_2)=\phi^{\infty}(\tau, v_2)-\phi^{\infty}(\tau, 0)=D_{v_2}\phi^{\infty}(\tau, \tilde v_2) v_2 =o(1) v_2,\label{22.1}
\end{align}
for $|v_2|, |\tau|\ll 1$. 

With \eqref{f12}, this observation yields a preliminary estimate for $v_2$. 
%Let $\big( \tau=\frac 1r, v_1 =\phi^\infty (\tau, v_2(r)), v_2(r)\big)$ be a solution to \eqref{13} 
On $W_{cs}^\infty$ the $v_2$-equation implies
\begin{align}\label{23}
\tfrac d{dr}v_2=-v_2+v_2O(\tfrac 1r+|v_2|^p), \textit{ for } |v_2|, |\tau|\ll 1. 
\end{align}
In the case $f=f_1$, this estimate follows directly from the form of $f_1$, in the case of $f=f_2$, it can be deduced using \eqref{18}.  
Multiplying both sides of the equation by $e^rv_2(r)$, we obtain:
\[\tfrac 12\tfrac d{dr}(e^r v_2^2(r))=e^r v_2^2(-\tfrac 12+O(\tfrac 1r+|v_2|^p))<0,\]
which shows that $e^rv_2^2$ is strictly decreasing for $r\gg1 $ and $|v_2|\ll1$. Therefore,
\begin{align}\label{24}
v_2(r)\neq 0, \textit{ and } |v_2(r)|\le C e^{-\frac r2},
\end{align}
for sufficiently large $r$. 

Next, we refine  \eqref{22.1} to show
\begin{align}\label{24.1}
\tfrac{v_1(r)}{rv_2(r)}=O(\tfrac 1{r^2}). 
\end{align}
Indeed, let $u(r)=\tfrac{v_1(r)}{rv_2(r)}$, and differentiate it to obtain
\begin{align}\label{25}
u_r=2u- \tfrac 1{r^2}[\tfrac{d-1}2\tfrac{v_1^2-v_2^2}{v_2^2}+\tfrac{v_1}{v_2}]+\tfrac 1{r}\tfrac{v_1+v_2}{v_2^2}f\bigl(\tfrac 1r, \tfrac{v_1+v_2}2\bigr). 
\end{align}
Using the estimate in \eqref{22.1} and \eqref{24}, we simplify it to  
\begin{align}\label{26}
u_r=2u+O(\tfrac 1{r^2}).
\end{align}
 Since  $u(r)\to 0$ as $r\to \infty$,  
we solve \eqref{26} 
 \[u(r)=-\int_r^\infty e^{-2(\rho-r)}O(\tfrac 1{\rho^2}) d\rho,\]
which yields $|u(r)|\le \tfrac C{r^2}$, proving \eqref{24.1}. 

Using \eqref{24.1} and \eqref{24}, we rewrite the $v_2$ equation in \eqref{13}:  
\begin{align}\label{27}
\tfrac d{dr} v_2=(-1-\tfrac{d-1}2 \tau)v_2+O(\tfrac 1{r^2})v_2. 
\end{align}
Thus, for sufficiently large $r$ and $r_0$, we have 
\begin{align*}
v_2(r)&=v_2(r_0)e^{-\int_{r_0}^r 1+\frac{d-1}{2\rho} d\rho} e^{\int_{r_0}^r O(\frac 1{\rho^2}) d\rho}\\
&=v_2(r_0)e^{r_0}r_0^{\frac{d-1}2} e^{-r}r^{-\frac{d-1}2}\psi(r_0,r),
\end{align*}
where $\psi(r_0, r)\to \psi_0> 0$ as $r\to \infty$ due to the convergence of $\int_{r_0}^\infty \frac 1{r^2}dr$. We can summarize:
\begin{gather*}
v_2(r)=c(r) r^{-\frac{d-1}2}e^{-r}, \; c(r)\to 2c, \;  v_1(r)=O(r^{-\frac{d+1}2}e^{-r}) \ \textit{ as } r\to \infty.
\end{gather*}
The asymptotic behaviors of $q,\ q_r$ and $q_1$ follow from this and equations  \eqref{12.1} and \eqref{19}. This completes the proof of Lemma \ref{lm: infinity}. 
\end{proof}

As an immediate consequence of Lemma \ref{lm: infinity}  and the decay of local radial $H^1$ solutions established below, we obtain a precise exponential decay estimate for local $H^1_{rad}$ solutions of \eqref{r2}-\eqref{f12}. 
\begin{corollary}\label{cor: infinity}
Let $q(r)$ be a local $H^1_{rad}$ solution of \eqref{r2}-\eqref{f12} as $r\to +\infty$. Then $q(r)$ satisfies the exponential decay estimate \eqref{18}-\eqref{20}. 
\end{corollary}

\begin{proof}
It suffices to show that $q(r)\to 0$ and $q_r(r)\to 0$ as $r\to \infty$. From the radial Sobolev embedding, we have
\begin{align}\label{d1}
|q(r)|\le C_d r^{-\frac{d-1}2}\|q\|_{H^1(|x|\ge r)}, 
\end{align}
which implies  $r^{\frac{d-1}2}q(r)\to 0$ as $r\to \infty$. 

To estimate the decay of $q_r$, we multiply both sides of \eqref{r2} by $r^{d-1} q_r$ and integrate over $[r_1, r_2]$. Integration by parts yields
 \begin{align*}
 \tfrac 12 q_r^2 r^{d-1}|_{r_1}^{r_2}=-\tfrac {d-1}2 &\int_{r_1}^{r_2} r^{d-2} q_r^2 dr+\tfrac 12 q^2 r^{d-1}|_{r_1}^{r_2}\\
 &-
 \tfrac{d-1}2 \int_{r_1}^{r_2} r^{d-2} q^2 dr +\int_{r_1}^{r_2}f(\tfrac 1r, q) q_r r^{d-1}dr . 
 \end{align*}
 Using the Cauchy-Schwarz inequality and estimate \eqref{d1},  we can bound the right-hand side by (in the case of $f_1$)
 \begin{align*}
 \tfrac 12 \big| q^2 r^{d-1}|_{r_1}^{r_2} \big| +C\big({r_1}^{-1}+(r_1^{-\frac{d-1}2} \|q\|_{H^1} )^p\big)\|q\|_{H^1}^2,
 \end{align*}
which converges to 0 as $r_1, r_2\to\infty$. In the case of $f_2$, a similar estimate can be obtained using the exponential decay estimate for $Q$. Therefore, $ r^{\frac{d-1}2}q_r\to 0$ as $r\to \infty$. 
\end{proof}

%\begin{remark} It is worth pointing out this Corollary applies not only to the ground state, but also to the radial sign changing bound states. 
%\end{remark}

Before concluding this subsection, we provide a remark on reparametrizing the center stable manifold $W_{cs}^\infty$ using the variables $(\tau, q, q_r)$, a formulation that will be used in subsequent analysis. Substituting  \eqref{12.1} into \eqref{21.1}, we obtain
\begin{align}\label{20.1}
q+q_r=\phi^{\infty}(\tau, q-q_r) \Longleftrightarrow 
%\end{align}
%Defining $w(r)=q(r)+q_r(r)$, this equation is equivalent to 
%\begin{align}\label{20.2}
F(\tau, q, v_1):=v_1-\phi^{\infty}(\tau, 2q-v_1)=0.
\end{align}
Using \eqref{21.1}, we find that there exists $\delta>0$ such that 
\begin{align*}
F(\tau,0,0)=0,\ F_{v_1}(\tau, 0, 0)\ge \tfrac 12, \; \forall |\tau| <\delta. 
\end{align*}
By the implicit function theorem, there exist $\eta>0$ and a $C^1$ function $\tilde \phi^\infty(\tau,q)$ such that for $|\tau|, |q|, |v_1|<\eta$, \eqref{20.1} holds if and only if $
v_1 =\tilde\phi ^\infty(\tau,q)$. In this new coordinate system on $W^{\infty}_{cs}$, solutions of \eqref{12} restricted to $W^{\infty}_{cs} $ satisfy  
\begin{align}\label{20.3}
q_r=-q+\tilde \phi^\infty(\tau, q).
\end{align}
 Moreover, it follows directly from \eqref{21.1} that
  \begin{align}\label{20.4}
 \tilde \phi^\infty \in C^1(\R^2), \ \tilde \phi^\infty(\tau,0)=0, \; D \tilde \phi^\infty(0,0)=0. 
 \end{align}
It again indicates that the two-dimensional surface $W^{\infty}_{cs}$ is tangent to $q+q_r=0$ as given by the invariant manifold theorem.

\subsection{Asymptotics of local $H_{rad}^1$ solutions near $r=0^+$}

We establish the following asymptotics of the local $H_{rad}^1$ solutions and come up with a smooth parametrization of 
%use this to smoothly parametrize 
all such solutions. Derivative estimates are also provided for later use in the subsequent sections. 

\begin{lemma}\label{lm: q}
Let $q(x)$ be a local $H^1_{rad}$ solution of \eqref{12} near $r=0^+$, then 
%there must exist $b\in \R$ such that 
\begin{align}\label{20.5}
b= \lim_{r\to 0^+} q(r) r^{\frac{d-2-\beta}2} \text{ exists}.
\end{align}
Moreover, for each $b\in \R$, there exists a unique local $H_{rad}^1$ solution $q(b, r)$ to \eqref{12} near $r=0^+$ satisfying \eqref{20.5} and $q(-b, r) = - q(b, r)$. Additionally, $q \in C^\infty \big( 
%the functions $q_{rrr}$ and $q_{brr}$ are continuous 
U\backslash (\R \times \{0\}) \big)$, where $U$ is a neighborhood of $\R \times \{0\}$ in  $\R \times [0, \infty)$. Furthermore, we have the following asymptotics: 
%We then are able to parametrize all the locally $H^1_{rad}$ solutions near $r=0^+$ by the parameter $b$ and write $q(r)=q(b,r)$. We have the following asymptotics: 
\begin{align}\label{64}
\begin{cases}
\lim_{r\to 0^+} q_r(b,r) r^{\frac{d-\beta}2}=-\tfrac{d-2-\beta}2 b,\\
\lim_{r\to 0^+} q_b(b,r) r^{\frac{d-2-\beta}2}=1,\\
\lim_{r\to 0^+} q_{br}(b,r) r^{\frac{d-\beta}2}=-\tfrac{d-2-\beta}2<0, \\
\lim_{r\to 0^+} q_1(b,r)r^{\frac{d-2-\beta}2}=b(\tfrac 2p-\tfrac{d-2-\beta}2),
\end{cases}
\end{align}
where $q_1(r)=q_1(b,r)$ is defined in \eqref{19} from the previous subsection. 
% Moreover, suppose $I$ is any compact interval of $r$ on which $q(b,r)$ exists, then we have
%\begin{align*}
%q(r)\in C^3(I), \; q_b(r)\in C^2(I). 
%\end{align*}
\end{lemma}

The rest of this subsection is dedicated to proving this lemma. 
We begin by introducing 
%the variables
\begin{align}\label{8}
u_1(r)=r^{\frac{d-2} 2} q(r), \ \; u_2(r)=r^{\frac d2}q_r(r). 
\end{align}

\begin{lemma} \label{L:H1-0a}
If $q(r)$ solves \eqref{12} on $(0, r_0)$ and $q(|\cdot|) \in H^1(\R^d)$ on $ \{x \in \R^d \mid |x| < r_0\}$, then $\lim_{r\to 0^+} u_1(r) =\lim_{r\to 0^+} u_2(r)  =0$.   
\end{lemma} 

%One relevant implication of this change is the decay of $u_1, u_2$ as $r\to 0^+$. 

\begin{proof} 
Indeed, to see this for $u_1$, we apply the Fundamental Theorem of Calculus, the Cauchy-Schwarz inequality, and Hardy inequality:
\begin{align*}
\big|u_1(r)^2|_{r_1}^{r_2} \big|&= \Big| 2\int_{r_1}^{r_2} r^{d-2} qq_r dr+(d-2)\int_{r_1}^{r_2}r^{d-3} q^2 dr \Big|\\
&=  \Big| 2\int_{r_1}^{r_2}q_r\tfrac q r r^{d-1} dr+(d-2)\int_{r_1}^{r_2} \tfrac{q^2}{r^2} r^{d-1} dr  \Big|\\
& \le C(\|q\|_{H^1(r_1 \le |x| \le r_2)} + \|r^{-1} q \|_{L_x^2(r_1\le |x|\le r_2)} \big)  \|r^{-1} q \|_{L_x^2(r_1\le |x|\le r_2)},
\end{align*}
showing that the sequence $\{u_1(r)\}$ is Cauchy. Since $\frac qr \in L_x^2 (|x| < r_2)$ again due to  the Hardy's inequality, we obtain $\lim_{r\to 0^+} u_1(r) =0$. 
For $u_2$, we compute
\begin{align*}
(u_2^2)_r=(2-d) r^{d-1}q_r^2+2a r^{d-2}qq_r+2r^d qq_r-2r^d |q|^{p}qq_r,
\end{align*}
and integrate over $[r_1,r_2]$. Proceeding similarly, we obtain
\begin{align*}
\left|u_2(r)^2|_{r_1}^{r_2} \right|\le & C(\|q\|_{H^1(r_1\le |x|\le r_2)} + \|\tfrac qr \|_{L_x^2(r_1\le |x|\le r_2)} \big)\|q\|_{H^1(r_1 \le |x| \le r_2)} \\
& +2\int_{r_1}^{r_2} r^{2-p\frac{d-2}2}|u_1|^p r^{d-2}|qq_r |dr\\
\le & C(\|q\|_{H^1(r_1\le |x|\le r_2)} + \|\tfrac qr \|_{L_x^2(r_1\le |x|\le r_2)} \big) \|q\|_{H^1(r_1 \le |x| \le r_2)} \\
& +C r_2^{2-p\tfrac{d-2}2} \big(\sup_{r\in[r_1, r_2]}|u_1(r)|^p \big) \|q\|_{H^1(r_1 \le |x| \le r_2)}^2.
\end{align*}
Since $p<\frac 4{d-2}$, it follows that $\lim_{r\to 0^+} u_2(r) = 0$.
\end{proof}

We now rewrite \eqref{12} as the following ODE system: 
\begin{align}\label{9}
\begin{cases}
\tfrac d{dr} u_1=\tfrac 1r(\tfrac{d-2}2 u_1+u_2),\\
\tfrac d{dr} u_2=\tfrac 1r \big(au_1-\tfrac{d-2}2 u_2+r^2 u_1-r^{2-\tfrac{d-2}2p}|u_1|^{p}u_1\big). 
\end{cases}
\end{align}
The singular factor $\frac 1r$ can be eliminated via the substitution $s=-\ln r$, yielding
\begin{align}\label{10old}
\begin{cases}
\tfrac d{ds}u_1&=-\tfrac {d-2}2 u_1-u_2, \\
\tfrac d{ds}u_2&=-au_1+\frac{d-2}2 u_2-r^2 u_1+r^{2-\frac{d-2}2p}|u_1|^{p}u_1,\\
r_s&=-r. 
\end{cases}
\end{align}
To manage the potentially non-smooth factor $r^{2-\frac{d-2}2p}$, we define 
\[\sigma=1-\tfrac{d-2}4p \in (0, 1),\quad \tau=r^{\sigma}, 
\] 
and rewrite equation \eqref{10old} in terms of $(\tau, u_1, u_2)$ as a $C^1$ system of ODEs:
\begin{align}\label{35}
\begin{cases}
\tfrac d{ds}u_1&=-\tfrac {d-2}2 u_1-u_2, \\
\tfrac d{ds}u_2&=-au_1+\frac{d-2}2 u_2-\tau^{\frac 2\sigma} u_1+\tau^2|u_1|^{p}u_1,\\
\tau_s&=-\sigma \tau. 
\end{cases}
\end{align}

To diagonalize the linear part, we introduce new variables: 
\begin{align}\label{36}
u_1=v_1+v_2, \; u_2=(-\tfrac{d-2}2-\tfrac 12\beta)v_1+(-\tfrac{d-2}2+\tfrac 12\beta)v_2,
\end{align}
which equivalently gives
\begin{align}\label{37}
\begin{pmatrix}v_1\\ v_2\end{pmatrix}=\frac 1\beta\begin{pmatrix}-\tfrac{d-2}2+\tfrac 12\beta&-1\\
\tfrac{d-2}2+\tfrac 12 \beta&1\end{pmatrix}
\begin{pmatrix}u_1\\u_2\end{pmatrix}.
\end{align}
Substituting into \eqref{37}, we obtain:
\begin{align}\label{38}
\begin{cases}
\tfrac d{ds}v_1&=\tfrac 12\beta v_1+\tfrac 1\beta \tau^{\frac 2\sigma}(v_1+v_2)-\tfrac 1\beta \tau^2|v_1+v_2|^{p}(v_1+v_2)\\
\tfrac d{ds}v_2&=-\tfrac 12\beta v_2-\tfrac 1\beta \tau^{\frac 2\sigma}(v_1+v_2)+\tfrac 1\beta \tau^2|v_1+v_2|^{p}(v_1+v_2)\\
\tau_s&=-\sigma \tau. 
\end{cases}
\end{align}
By invariant manifold theory, there exists a unique 2-dim $C^1$ local stable manifold $W_s^0$ near $(\tau, v_1,v_2)=(0,0,0)$, described as the graph of a function:
\begin{align}\label{39}
v_1=\phi^0(\tau, v_2), 
\end{align}
with
\begin{align}\label{40}
\phi^0\in C^1, \quad \phi^0(\tau, 0)=0, \; D \phi^0(0, 0)=0. 
\end{align}
According to the stable manifold theorem, an orbit $(\tau (s), v_1(s), v_2(s))$ of \eqref{38} belongs to $W_s^0$ iff $\limsup_{s \to +\infty} e^{-\frac \beta4 s} |v_j(s)| < \infty$,  $j=1, 2$. 
Moreover, all orbits of \eqref{38} on $W_s^0$ satisfy $|v_1(s)| + |v_2(s)| \le O(e^{-\frac \beta4 s})$ as $ s\to +\infty$. Therefore, Lemma \ref{L:H1-0a} and direct computations yield 

\begin{lemma} \label{L:H1-0b}
Suppose $q(r)$ solves \eqref{12} on $(0, r_0)$, then $q(|\cdot|) \in H^1(\R^d)$ on $ \{x \in \R^d \mid |x| < r_0\}$ iff its corresponding orbit belongs to $W_s^0$.
%$\lim_{r\to 0+} u_1(r) =\lim_{r\to 0+} u_1(r)  =0$. 
\end{lemma}

Since our objective is both deriving the asymptotics and smoothly parametrizing all local $H_{rad}^1$ solutions near $r=0^+$, we proceed by solving the equivalent ODE on the stable manifold. This approach enables us to establish continuous dependence on the parameter $b$ in Lemma \ref{lm: q}. Some technical complications arise due to the limited regularity of the nonlinearity in \eqref{38} when $p$ is small. 
%A shorter way of giving rise to the asymptotics is sketched in Corollary \ref{cor: g}. 
We first prove a derivative estimate for $\phi^0$, which would be straightforward if the nonlinearity in \eqref{38} were smoother. 

\begin{lemma}\label{lm: dv2} 
There exists $C, \eta>0$ such that for all $|\tau|, |v_2|\le \eta$, 
\begin{align}\label{40.1}
|D_{v_2} \phi^0(\tau, v_2)|\le C\tau^2. 
\end{align}
\end{lemma}

\begin{proof}
Due to the invariance of $W_s^0$ and the form of \eqref{38}, there exists a sufficiently small $\eta>0$ such that for any $|\tau_0|+|v_{20}|\le \eta$, the solution to the equation \eqref{38} given by
\begin{align}\label{40.2}
(\tau(s), v_1(s), v_2(s))=\bigl(e^{-\sigma s}, \phi^0(e^{-\sigma s}, v_2(s)), v_2(s)\bigr)
\end{align}
with initial data at $s_0=-\frac 1\sigma \ln \tau_0$:
\[
(\tau(s_0), v_1(s_0), v_2(s_0))=\bigl(\tau_0, \phi^0(\tau_0, v_{20}), v_{20}\bigr)
\]
remains uniformly bounded for all later time:
\begin{align}\label{40.3}
|\tau(s)|+|v_1(s)|+|v_2(s)|\le 2\eta, \ s\ge s_0. 
\end{align}

To derive \eqref{40.1}, we consider the linearized solution along the trajectory $(\tau(s), v_1(s), v_2(s))$ described in \eqref{40.2}. Let 
\[\bigl(e^{-\sigma s}, v_1(s,\eps), v_2(s, \eps)\bigr)=\bigl(e^{-\sigma s}, \phi^0(e^{-\sigma s}, v_2(s,\eps)), v_2(s, \eps)\bigr)\]
be a perturbed solution on $W^0_s$ with the initial value 
\[\bigl(e^{-\sigma s_0}, v_1(s_0,\eps), v_2(s_0, \eps)\bigr)=\bigl(e^{-\sigma s_0}, \phi^0(e^{-\sigma s_0}, v_2(s_0)+\eps), v_2(s_0)+ \eps\bigr).\]
They yield a linearized solution $(\delta v_1(s), \delta v_2(s))$:
\begin{align}\label{43}
\delta v_1(s)=\tfrac{d}{d\eps}v_1(s,\eps)|_{\eps=0}, \;\; \delta v_2(s)=\tfrac{d}{d\eps}v_2(s,\eps)|_{\eps=0}. 
\end{align}
It follows that
\begin{align}\label{44}
%\begin{cases}
\delta v_1(s)=\phi_{v_2}^0(e^{-\sigma s}, v_2(s))\delta v_2(s), \; \forall s\ge s_0, \qquad 
\delta v_2(s_0)=1,
%\end{cases}
\end{align}
and from \eqref{38}, the linearized solution satisfies
\begin{align}\label{45}
\begin{cases}
\tfrac d{ds}\delta v_1=\tfrac 12\beta \delta v_1+\tfrac 1\beta e^{-2s}(\delta v_1+\delta v_2)-\tfrac {p+1}\beta e^{-2\sigma s}|v_1+v_2|^p (\delta v_1+\delta v_2),\\
\tfrac d{ds}\delta v_2=-\tfrac 12\beta \delta v_2-\tfrac 1\beta e^{-2s}(\delta v_1+\delta v_2)+\tfrac {p+1}\beta e^{-2\sigma s}|v_1+v_2|^p (\delta v_1+\delta v_2).
\end{cases}
\end{align}
From \eqref{44} and \eqref{40}, we deduce that 
\begin{align}\label{40.5}
|\delta v_1(s)|=o(1)|\delta v_2(s)|\le |\delta v_2(s)|.
\end{align}
Moreover, using the $\delta v_2$-equation and applying the bounds \eqref{40.3} and \eqref{40.5}, it follows that $(\delta v_2)^2(s)$ is strictly decreasing. Therefore,
 \[
|\delta v_2(s)|, \, |\delta v_1(s)|\le 1, \quad \forall s\ge s_0. 
\]
Next, we multiply both sides of the $\delta v_1$-equation in \eqref{45} by $e^{-\frac 12 \beta(s-s_0)}$ and integrate over $[s_0, \infty)$. This yields 
\[
\delta v_1(s_0)= \int_{s_0}^\infty e^{-\frac 12\beta(s'-s_0)}(\delta v_1+\delta v_2)O(e^{-2s'}+e^{-2\sigma s'} )ds',
\]
and we deduce
\begin{align}\label{47}
|\delta v_1(s_0)|\le Ce^{-2\sigma s_0}.
\end{align}
Combined with \eqref{44}, this implies 
\begin{align*}
\bigl|\phi^0_{v_2}(e^{-\sigma s_0}, v_{20})\bigr|=|\delta v_1(s_0)|\le C e^{-2\sigma s_0} = C \tau_0^2. 
\end{align*}
Since $\tau_0$ and $v_{20}$ were arbitrary, this proves \eqref{40.1}. Hence, the lemma is established. 
\end{proof}

We now proceed to rewrite the $q$-equation on the stable manifold $W_s^0$. To this end, we substitute $v_1, v_2$ in \eqref{39} using the relation given in \eqref{37}, expressing them in terms of $(u_1, u_2)$. This yields:
\begin{align}\label{49}
u_2+\tfrac{d-2-\beta}2 u_1+\beta \phi^0(\tau, \tfrac{d-2+\beta}{2\beta}u_1+\tfrac 1\beta u_2)=0. 
\end{align}
Let us define
\begin{align}\label{31}
w:=-\beta v_1= u_2+\tfrac{d-2-\beta}2 u_1.
\end{align}
Then equation \eqref{49} can be written as 
\begin{align}\label{32}
F(\tau, u_1, w):=w+\beta \phi^0(\tau, u_1+\tfrac 1\beta w)=0.
\end{align}
From \eqref{40}, it is straightforward to verify that 
\begin{align*}
F(0, 0, 0)=F_\tau(0,0,0)=F_{u_1}(0,0,0)=0, \; F_w(0,0,0)=1. 
\end{align*}
Therefore, by the implicit function theorem, there exists a $C^1(\R^2)$ function $\tilde \phi^0(\tau, u_1)$ defined in a neighborhood of $(0,0)$ such that 
%$F=0$ if and only if
\begin{align}\label{32.1}
F=0 \Longleftrightarrow w=\tilde \phi^0(\tau, u_1).
\end{align}
Moreover, we have
\begin{align}\label{33}
\tilde \phi^0(\tau, 0)=0, \; D \tilde \phi^0(0,0)=0.
\end{align}

We also claim that
\begin{align}\label{50}
|\tilde \phi^0_{u_1}(\tau, u_1)|\le C\tau ^2,
\end{align}
for sufficiently small $\tau$ and $u_1$. Indeed, inserting \eqref{32.1} into \eqref{32}, we obtain
\begin{align*}
\tilde \phi^0(\tau, u_1)=-\beta \phi^0(\tau, u_1+\tfrac 1\beta \tilde \phi^0(\tau, u_1)). 
\end{align*}
Differentiating both sides with respect to $u_1$ yields 
\begin{align*}
\partial_{u_1}\tilde \phi^0(\tau, u_1)=-\phi_{v_2}^0\bigl (\tau, u_1+\tfrac 1\beta\tilde \phi^0(\tau, u_1)\bigr)(\tilde \phi_{u_1}^0(\tau, u_1)+\beta), 
\end{align*}
which, together with \eqref{40.1}, implies the estimate in \eqref{50}. 

In summary, in a small neighborhood of $0$, 
%$(\tau=r^\sigma, u_1, u_2)=(0,0,0)$, 
equation \eqref{49} is equivalent to 
\begin{align*}
u_2+\tfrac{d-2-\beta}2 u_1=\tilde \phi^0(r^\sigma, u_1), 
\end{align*}
which, using \eqref{8}, can be rewritten as
\begin{align}\label{34}
q_r+\tfrac{d-2-\beta}{2r}q=r^{-\frac d2}\tilde \phi^0(r^\sigma, r^{\frac{d-2}2}q).
\end{align}
 Multiplying both sides of \eqref{34} by $r^{\frac{d-2-\beta}2}$, and defining the new variable 
%%we write \eqref{34} as
%%\begin{align*}
%%(r^{\frac{d-2-\beta}2}q)_r=r^{-1-\frac 12\beta}\tilde \phi^0(r^\sigma, r^{\frac{d-2}2}q). 
%%\end{align*}
%Define the new variable 
\begin{align}\label{34.1}
z(r)=r^{\frac{d-2-\beta}2}q(r),
\end{align}
we reduce \eqref{34} to 
\begin{align}\label{34.2}
z_r=r^{-1-\frac 12\beta}\tilde \phi^0(r^\sigma, r^{\frac \beta 2}z). 
\end{align}
If $q(r)$ is a local $H_{rad}^1$ solution near $r=0^+$, we have shown that $\lim_{r\to 0+} u_1(r) =0$. Therefore, using \eqref{33} and \eqref{50}, we obtain, for $r\ll1$, 
\begin{align}
|z_r|&=\bigl| r^{-1-\frac 12 \beta}\bigl[\tilde \phi^0(r^\sigma, r^{\frac \beta 2} z)-\tilde \phi^0(r^\sigma, 0)
\bigr]\bigr| =\bigl |r^{-1-\frac 12 \beta}D_{u_1}\tilde \phi^0(r^\sigma, \tilde r^{\frac\beta 2}z) r^{\frac \beta 2}z(r)\bigl| \le Cr^{-1+2\sigma}|z(r)|.\label{1210}
\end{align}
This estimate implies the existence of $\lim_{r\to 0+} z(r) \ne 0$ unless $z \equiv 0$, thereby establishing \eqref{20.5}.

From the above analysis, we conclude that finding the local $H_{rad}^1$ solution $q(b, r)$ to the equation \eqref{12} near $r=0^+$ is equivalent to solving the initial value problem 
\begin{align}\label{52}
\begin{cases}
z_r=g(r,z):=r^{-1-\frac 12\beta}\tilde \phi^0(r^\sigma, r^{\frac \beta 2} z), \\
z(0)=b, 
\end{cases}
\end{align}
where $\tilde \phi^0\in C^1$ satisfies properties in \eqref{33} and \eqref{50}. 
%We now turn to.
We have the following result of the well-posedness of \eqref{52}.

\begin{lemma}\label{lm: z}
For any $b_0 \in \R$ and $R>0$, there exists a sufficiently small $\eta>0$ such that \eqref{52} admits a unique solution 
\begin{align*}
z(b,r)\in C^1(S) \cap C^0 (\bar S), 
\textit{ where } S = [b_0-R, b_0+R] \times (0, \eta].
%I_{b,\eta}=[0, \min\bigl (\eta, (\tfrac \eta {|b|})^{\frac 2\beta}\bigr)]. 
\end{align*}
Moreover, for any $(b, r)\in \bar S$, 
%I_{b,\eta}$,
\begin{align}\label{53}
\begin{cases}
sgn (b) z(b,r)>0,\quad \text{ if } \ b\ne 0, \\
z_b(b,r)>0, \; |z(b,r)|+z_b(b,r)\le C, \\
|z_r(b,r)|+|z_{br}(b,r)|\le C r^{-1+2\sigma}, \quad r>0. 
\end{cases}
\end{align}
\end{lemma}

\begin{proof} 
Although $g(r, z)$ in \eqref{52} is not continuous at $r=0^+$, it follows from \eqref{33} and \eqref{50} that 
 $g$ is $L^1$ in $r \ll 1$ for any fixed $z$ and, for almost every $r$, $g$ is $C^1$ in $z$.
%$g\in L^1(dr)\cap C^1(dz)$ and $g_z\in L^1(dr)$ as well. 
Therefore,
the existence, uniqueness, and $C^1$ dependence on $b$ of solutions to \eqref{52} follow from a standard application of the contraction mapping principle. Below, we sketch the argument, through which we also derive the detailed estimates in \eqref{53}. 

For $b \in [b_0-R, b_0+R]$, consider 
% and define the solution map
%the solution map 
\begin{align}\label{55}
\Phi(b, z)(r)=b+\int_0^ rg(\rho, z(\rho))d\rho. 
\end{align}
We show that $\Phi (b, \cdot)$ is a contraction on the ball 
\begin{align}\label{56}
\mathcal B=\bigl \{z(r)\in C^0([0, \eta]) \mid |z|_{C^0([0, \eta])}\le |b_0| + R+1 \bigr\},
\end{align}
for sufficiently small $\eta$ to be determined. 
This requires suitable estimates on $g(r, z)$ and on the difference $g(r, z_1)-g(r,z_2)$ for $z, z_1, z_2\in \mathcal B$. 

By \eqref{1210}, we have 
\begin{align}
|g(r,z)|\le Cr^{-1+2\sigma}|z(r)|_{C^0}\le C (|b_0| + R+1)  r^{-1+2\sigma}. \label{57}
\end{align}
Thus, for any $r\in (0,\eta]$,
\begin{align}\label{57.1}
\bigl|\int_0^r g(\rho, z)d\rho\bigr|\le C\eta^{2\sigma} \le \frac{1}{2}, \; \textit{ and }|\Phi(b, z)(r)|\le |b_0| +R +1. 
\end{align}
Similarly, 
\begin{align*}
\bigl |g(r,z_1)-g(r,z_2)\bigr|\le Cr^{-1+2\sigma}|z_1-z_2|_{C^0}, 
\end{align*}
and hence
\begin{align*}
|\Phi(b, z_1)-\Phi(b, z_2)|\le C \eta^{2\sigma}|z_1-z_2|_{C^0}\le \frac 12|z_1-z_2|_{C^0}. 
\end{align*}
This proves existence and uniqueness, as well as the $C^0$ dependence on $b$, of the solution $z$ with initial value $b$.  %of the solution $z\in C^0([0, \eta])$ solution with initial value $b$.

Now let $z(b, r)$ denote the unique solution with the data $z(b,0)=b$. 
%We write $z(r)=z(b,r)$ to show the  dependence on $b$. 
The sign of $z(b, r)$ matches that of $b$ due to the continuity in $b$. From \eqref{57} and \eqref{57.1}, we also obtain 
\begin{align*}
%\frac 12 |b|\le sgn(b)z(b,r)\le 2|b|, \textit{ and }
|z_r(b,r)|=|g(r,z)|\le Cr^{-1+2\sigma}, \textit{ for all }r\in (0, \eta].
\end{align*}
To estimate $z_b$, differentiate \eqref{52} with respect to $b$ to obtain
\begin{align}\label{62}
\begin{cases}
(z_b)_r=r^{-1}D_{u_1}\tilde \phi^0(r^\sigma, r^{\frac{\beta}2} z)z_b, \\
z_b(b,0)=1. 
\end{cases}
\end{align}
It follows that 
\begin{align*}
z_b(b,r)=e^{\int_0^r \rho^{-1}D_{u_1}\tilde \phi^0(\rho^\sigma, \rho^{\frac \beta 2}z)d\rho} \in C^0([0, \eta]) \cap C^1 ((0, \eta]), 
\end{align*}
and hence
\begin{align*}
z_b(b,r)>0, \; z_b(b,r)\le \exp({Cr^{2\sigma}})\le C. 
\end{align*}
Substituting into \eqref{62}, we deduce
\begin{align*}
|z_{br} (b,r)|\le C r^{-1+2\sigma}.  
\end{align*}
Lemma \ref{lm: z} is proved. 
\end{proof}

%Finally, we remark that by transforming the bound for $z$ to $q$ via relation \eqref{34.1}, we are able to prove all the asymptotic estimates in Lemma \ref{lm: q}. Suppose $q(b,r)$ is the constructed local solution defined on the interval $(0, \eta]$. We can extend the solution by solving a non-singular initial value problem with data $(q(b,r_{b,\eta}), q_r(b,r_{b,\eta})$. Repeatedly using the equation \eqref{12}, we get the desired smoothness in $r$. 

From Lemma \ref{lm: z}, one obtains the limits in \eqref{64} through direct computations. The smoothness of $q$ in $(b, r)$ for $1 \gg r > 0$ follows from the fact $q \ne 0$ and the form of equation \eqref{12}. The proof of Lemma \ref{lm: q} is complete.
%At this point, we have completed the proof of Lemma \ref{lm: q}. The limited regularity arises primarily from the nonlinearity associated with small values of the exponent of $p$.  

Before concluding this subsection, we establish the asymptotic behavior of the linearized solution defined in \eqref{r1} for $A=a+d-1$, which will be used in later sections. 

\begin{corollary}\label{cor: g}
Let $Q(x)$ be the ground state solution of \eqref{12}. 
%that satisfies the asympotics in Lemma \ref{lm: q}. 
Let $G(x)$ be a radial, nontrivial local $H^1$ solution of the equation 
\begin{align}\label{r10}
\mathcal L_{a+d-1}G+G-(p+1)Q^p G=0
\end{align}
in a neighborhood of $r=0^+$. Then there exists $c\neq 0$ such that 
\begin{align}\label{r11}
\lim_{r\to 0^+}\frac {G(r)}{r^{-\frac{d-2}2+\frac 12\sqrt{d^2+4a}}}=c. 
\end{align}
\end{corollary}
\begin{proof} Due to similarity with the previous argument, we provide only a sketch of the main steps. Define
\begin{align}\label{r11.1}
u_1(r)=r^{\frac{d-2}2}G(r), \ u_2(r)=r^{\frac d 2} G_r(r),
\end{align}
and use the change of variable $s=-\ln r$. Then, equation \eqref{r10} transforms into the following ODE system: 
\begin{align}\label{r12}
\begin{cases}
\tfrac d{ds} u_1=-\tfrac{d-2}2 u_1-u_2,\\
\tfrac d{ds} u_2=-(a+d-1)u_1+\tfrac{d-2}2 u_2-r^2(1-(p+1)Q(r)^p)u_1,\\
r_s=-r. 
\end{cases}
\end{align}
From the asymptotic behavior of $Q$ near $r=0^+$, as established in Lemma \ref{lm: q}, we have
\begin{align*}
r^2 Q(r)^p=O(r^{2+p(-\frac{d-2}2+\frac 12\beta)}), 
\end{align*}
which might fail to be $C^1$ when $p$ is close to $\tfrac 4{d-2}$. To remove this nonessential singularity, we introduce the auxiliary variable
\begin{align}\label{r13}
\sigma=\tfrac 12\bigl(2+p(-\tfrac{d-2}2+\tfrac 12\beta)\bigr) \in (\tfrac \beta4 p, 1), \quad \tau=r^\sigma, 
\end{align}
as in \eqref{35}. Applying a diagonalization procedure, we rewrite \eqref{r12} in terms of new variables $v_1, v_2$ as 
\begin{align}\label{r15}
\begin{cases}
\tfrac d{ds} v_1=\tfrac 12 \tilde \beta v_1+\tfrac 1{\tilde \beta}\tau^{\frac 2{\sigma}}\bigl(1-(p+1)Q(\tau^{\tfrac 1\sigma})^p\bigr)(v_1+v_2)\\
\tfrac d{ds} v_2=-\tfrac 12 \tilde \beta v_2-\tfrac 1{\tilde \beta}\tau^{\frac 2\sigma}\bigl(1-(p+1)Q(\tau^{\tfrac 1\sigma})^p\bigr)(v_1+v_2),\\
\tau_s =-\sigma \tau. 
\end{cases}
\end{align}
where $\tilde \beta=\sqrt{d^2+4a}$. The transformation between variables is given by
\begin{align}\label{r16}
u_1=v_1+v_2, \; u_2=(-\tfrac{d-2}2-\tfrac 12 \tilde \beta)v_1+(-\tfrac{d-2}2+\tfrac 12 \tilde \beta)v_2. 
\end{align}
It is straightforward to verify that the right-hand sides of \eqref{r15} are $C^2$ in $v_1, v_2,\tau$. Due to the linearity of the system in $v_1, v_2$, there exists a unique $C^2$ stable manifold $W^0_s$ near $(\tau, v_1, v_2)=(0,0,0)$, which can be expressed as the graph of a linear function:
\begin{align*}
v_1=\omega(\tau)v_2, \textit{ for } |\tau|\ll 1, 
\end{align*}
with $\omega(\tau)\in C^1$ and $\omega(0)=0$. This yields the estimate
\begin{align}\label{r17}
v_1=O(\tau) v_2. 
\end{align}
Substituting \eqref{r17} into the second equation of \eqref{r15}, and using the estimate
\[
\tau^{\frac 2\sigma}Q(\tau^{\frac 1\sigma})^p =O(\tau^2),
\]
we see that the equation for $v_2$ becomes
\[
\tfrac d{ds}v_2=-\frac 12 \tilde \beta v_2+O(\tau^2)v_2=v_2(-\tfrac 12 \tilde \beta+O(e^{-2\sigma s})). 
\]
Thus, we obtain
\[
v_2(s)=C(s)e^{-\tfrac 12 \tilde \beta s},\ \textit{ for some } C(s)\to c\neq 0, \textit{ as } s\to \infty,
\]
and $v_1$ is a high-order term due to \eqref{r17}. Combining the identities \eqref{r16}, \eqref{r13} and \eqref{r11.1}, we arrive at the asymptotic formula \eqref{r11}. 
%This completes the proof of Corollary \ref{cor: g} and concludes the discussion of this subsection. 
\end{proof}

\subsection{Classification of positive  local $H_{rad}^1$ solutions}
Let $b>0$ and let $q(b,r)$ denote the positive radial local $H^1$ solution of \eqref{12} near $r=0^+$, as constructed in the last subsection. From Lemma \ref{lm: q}, for every $b>0$, there exists $r_0>0$ such that the pair $(q(b, r), q_r(b,r))$ lies in the fourth quadrant of the $(q,q_r)$-plane for all $r \in (0, r_0)$. 
%These include all possible $H_{rad}^1$ solutions \eqref{12}. 
The goal of this subsection is to classify the parameter space $b>0$ into three categories based on the eventual behavior of the curve 
\begin{align*}
\Gamma=\{(q(b,r), q_r(b,r)) \mid r >0\}. 
\end{align*}
As in the classical literature on the uniqueness of positive radial solutions without the inverse square potential, we will prove that only three qualitative behaviors are possible:
\begin{align*}
S^-&=\bigl\{b>0\mid \exists \, r_1>0 \textit{ such that } q(b,r_1)=0,\, q_r(b, r_1)<0,    \\
&\qquad\qquad\qquad q(b,r)>0, \ q_r(b,r)<0, \ \forall r\in (0, r_1) \bigr\},
\end{align*}
\begin{align*}
S^+&=\bigl\{b>0 \mid \exists \, r_2> 0 \textit{ such that } q_r(b,r_2)=0,\, q(b, r_2) >0,  \\
&\qquad\qquad\qquad q(b,r)>0, \ q_r(b,r)<0, \ \forall  r\in(0, r_2) \bigr\},
\end{align*}
\begin{align*}
S^0=\bigl\{b>0\mid  &q(b, r)>0,\, q_r(b, r)<0, \, \forall r>0, \\
&\lim_{r\to \infty} q(b,r)=\lim_{r\to \infty} q_r(b,r)=0 \bigr\}.
\end{align*}
Clearly $S^0$ corresponds to decaying solutions, which by Lemma \ref{lm: infinity} are ground states of \eqref{12}. 

\begin{lemma}\label{lm: b} 
Let $S^{\pm}$, $S^0$ be the sets of $b$ defined above. Then
%All solutions $q(b, r)$ exist for all $r>0$ and 

a) $\R^+=S^+\cup S^0\cup S^-$,

b) $S^+$ and $S^-$ are open subsets of $\R^+$. 

c) For any $b\in S^+$, let $r_2$ be given in the definition of $S^+$, then we have $q(b,r)> q(b,r_2)$ for any $r > r_2$. 
%we have the lower bound  
%\begin{align}\label{65}
%\inf_{r\ge r_2} q(b,r)\ge q(b,r_2). 
%\end{align}
\end{lemma}

\begin{proof} We first claim that the trajectory $\Gamma$ is bounded. Namely for any $r_0>0$, there exists $C>0$ such that
\begin{align}\label{65.1}
|q(b, r)|+|q_r(b, r)|<C, \quad \forall r\ge r_0. 
\end{align}
This boundedness ensures the global existence of the trajectory $\Gamma$. 
To establish this, define the energy function 
\begin{equation}\label{Hdef}
H(r)=\frac 12 q_r^2-\frac a{2r^2}q^2-\frac 12 q^2+\frac 1{p+2}|q|^{p+2}.
\end{equation}
Multiplying both sides of \eqref{12} by $q_r$, we obtain 
\begin{align}\label{66}
\frac d{dr}H(r)=-\frac{d-1} rq_r^2+\frac a{r^3}q^2<0,
\end{align}
so $H(r)<H(r_0), \ \forall r>r_0.$

On intervals where $q(r)$ is sufficiently large such that $\frac 1{p+2}|q|^{p+2}-\frac 12 q^2\gtrsim |q|^{p+2}$, we deduce: 
\[
q_r^2+|q|^{p+2}\lesssim H(r)< H(r_0).
\]
On the complementary regions where $|q(r)|\lesssim 1$, it follows immediately that 
\[
\frac 12 q_r^2< H(r_0)+\frac 12 q^2 \lesssim H(r_0)+ 1.
\]
Thus, the bound \eqref{65.1} is established. 

Now, consider the square in the $(q, q_r)$-plane bounded by $q=C$, $q_r=-C$ and the axes $q=0, \ q_r=0$. The trajectory $\Gamma$ must either remain inside this square for all $r$, or touch the boundary at some finite $r$. Due to \eqref{65.1} and uniqueness, any such boundary contact must occur on one of the axes, excluding the origin $(0,0)$. 

Suppose $\Gamma$ intersects the $q_r$-axis first at $r= r_1$,  
with $q(b,r_1)=0$ and $q_r(b,r)<0$ for all $r< r_1$. Then $b\in S^-$. Moreover, since $q_r(b,r_1)<0$, 
$q$ becomes negative for $r=r_1+$, implying a transverse crossing. The smoothness of $q(b,r)$ ensures that $S^-$ is open. 
%we must have $q(b, r_1+)<0$ and $\Gamma$ crosses over the vertical axis due to the monotonicity of $q$. The non-degeneracy of the zero point $r_1$ and the continuity of $q(b,r)$ also immediately imply that $S^-$ is open. 

If instead $\Gamma$ touches the $q$-axis  first at $r= r_2$, 
%$r_2>r_0$ is the first point at which 
then $q_r(b,r_2)=0$ with $q(b,r_2)>0$ and $q_r(b,r)<0$ for $r<r_2$, hence $b\in S^+$. We claim that $\Gamma $ crosses the $q$ axis transversally. Indeed, if $q_{rr}(b,r_2)=0$, then from \eqref{12}
\begin{align*}
q_{rrr}(b,r_2)=\bigl (-\tfrac{d-1}r q_r+\tfrac a{r^2}q+q-q^{p+1}\bigr)_r\bigl|_{r=r_2}=-\tfrac{2a}{r^3} q(b,r_2)>0, 
\end{align*}
implying that $q_r$ has a local minimum at $r_2$, which contradicts $q_r (b, r_2-)<0$. Thus
 \begin{equation} \label{69}
q_{rr}(b,r_2)> 0, 
\end{equation} 
confirming transversality and implying that $S^+$ is open. 

If $\Gamma$ remains strictly inside the square for all $r> 0$, then $q_r(b,r)<0$ and $q(b,r)>0$ is strictly decreasing in $r$, thus,
\begin{align}\label{70}
q(b,r)\to c\ge 0 \textit{ as } r\to \infty
%\end{align}
\; \text{ and } \; 
%Suppose a sequence $r_n\to \infty$ exists with $q_r(b,r_n)\to \tilde c<0$. Then the boundedness of $q_{rr}$ \eqref{65.1} and the consequent boundedness of $q_{rr}$ would immediately lead to a contradiction to \eqref{70}. Hence, the only possibility for $q_r$ is 
% \begin{align}\label{71}
0> q_r(b,r)\to 0, \textit{ as } r\to \infty. 
\end{align}

If $c=0$, then $b\in S^0$. If $c>0$ and $c\neq 1$, then from \eqref{12}, we obtain
\begin{align}\label{72}
q_{rr}(b,r)\to c(1-c^p)\neq 0, \textit{ as } r\to \infty, 
\end{align}
which contradicts $q_r(b,r)\to 0$.

The final case $c=1$ is more subtle. Set 
\begin{align*}
u=q-1, \; v=q_r,
\end{align*}
so that equation \eqref{12} can be rewritten as the system 
\begin{align}\label{74}
\begin{cases}
u_r=v\\
v_r=-\tfrac{d-1}r v+\tfrac a{r^2}(u+1)-pu+O(|u|^{\min\{ p+1, 2\}}). 
\end{cases}
\end{align}
Heuristically, from the linear analysis, it is not difficult to observe that the only scenario in which $(q,q_r)\to (1,0)$ (i.e. $(u,v)\to 0$) is when the trajectory $\Gamma$ spirals around and approaches $(1,0)$. Therefore, $\Gamma$ cannot approach $(1,0)$ without eventually entering into the upper plane. 

To make this argument rigorous, we switch to polar coordinates by letting 
\[u=\tfrac \rho{\sqrt p} \cos \theta, \ v=\rho \sin\theta,
\]
 where the angular variable $\theta$ satisfies the equation
\begin{align}\label{75}
\theta_r=-\sqrt p +\frac a{\rho r^2}\cos\theta-\tfrac{d-1}r \sin\theta\cos\theta +\tfrac a{\sqrt p r^2}\cos^2\theta
+O(\rho^{\min\{ p, 1\}}). 
\end{align}
From \eqref{70} we know for sufficiently large $r$, the trajectory $\Gamma$ approaches $(1,0)$ from the lower right, implying that $-\tfrac\pi 2<\theta<0$ and thus $\cos \theta>0$. Substituting this and the fact $\rho\to 0$ to \eqref{75} gives 
\begin{align}\label{76}
\theta_r<-\sqrt p+O(\tfrac 1r+\rho^{\min\{p, 1\}})<-\tfrac{\sqrt p}2
\end{align}
for sufficiently large $r$. This shows that $\theta$ is decreasing at a uniform rate, leading to a contradiction with $\theta> -\frac \pi2$. 

Thus, the first two statements in Lemma \ref{lm: b} are established. 

%The openness of the set $S^{\pm}$ follows simply from the fact that $ \frac{\partial^k}{\partial r^k} q\in C(\R^+\times I)$ for any compact interval of $r$ on which $q$ exists and $k=0,1, 2,3$ as shown in Lemma \ref{lm: q}. 

Finally, we prove the uniform lower bound in statement c) by contradiction. For convenience, we omit the dependence on the parameter $b$. 
%Suppose, for contradiction, that \eqref{65} fails, namely  
%\begin{align}\label{77}
%\inf_{r\ge r_2}q(r)<q(r_2).
%\end{align}
 From \eqref{69}, there must exist $r_3>r_2$ such that 
\begin{align}\label{78}
q(r_3)=q(r_2), \; q(r)\ge q(r_2), \ \forall r\in [r_2, r_3]. 
\end{align}
 Integrating \eqref{66} over the interval $[r_2, r_3]$, and noting that $q_r(r_2)=0$, we obtain 
 \begin{align}
 \tfrac 12 q_r( r_3)^2+\tfrac a2 q(r_2)^2[\tfrac 1{r_2^2}-\tfrac 1{r_3^2}]&=\int_{r_2}^{r_3}(-\tfrac{d-1}r q_r^2+\tfrac a{r^3} q^2) dr\notag\\
 &<\int_{r_2}^{r_3}\tfrac a{r^3} q(r)^2 dr\le \tfrac a2 q(r_2)^2(\tfrac 1{r_2^2}-\tfrac 1{r_3^2}).\label{79}
 \end{align}
 This leads to the contradiction $q_r(r_3)^2<0$. Therefore, statement c) holds, completing the proof of Lemma \ref{lm: b}.
\end{proof}

 \section{Uniqueness of the ground state \label{S:uniqueness}}

In this section, we prove the uniqueness of the ground state solution, which is equivalent to $|S^0|=1$. The strategy is to establish the isolation of the set $S^0$ (Lemma \ref{lm: i}) and the rigidity of the local structure around points in $S^0$ (Lemma \ref{lm: s}). One may compare with the classical treatment by McLeod \cite{McLeod} in the case without the singular inverse square potential.

\subsection{Zeros of solutions to the linearization of \eqref{12}}  

In this  subsection, we consider $b\in S^0\cup S^-$, and for simplicity, we write 
   \[
   q(r):=q(b,r).
   \]
  From the definition, $q(r)$ satisfies the asymptotics in Lemma \ref{lm: q} near $r=0^+$. Moreover, if $b \in S^-$, 
%  $q(r)$ corresponds to an $S^-$ solution, 
then there exists some $0<r_0<\infty$ such that $q(r_0)=0, \ q_r(r_0)<0$ and $q(r)>0,\  q_r(r)<0$ for all $r\in (0,r_0)$. Throughout this section, we define
   \begin{align*}
   r_0=\begin{cases} \mbox{first zero as described above}; & \mbox{ if } b\in S^-,\\
   \infty;&\mbox{ if } b\in S^0. 
   \end{cases}
   \end{align*}

Define    
    \[
    q_1(r)=\tfrac 2p q(r)+rq_r(r), \ \tilde q(r)=q_b(r,b)
    \]
  Recall the differential equations satisfied by $q$, $q_1$, $\tilde q$: 
    \begin{align*}
    \begin{cases}
    L_1 q=(\la+1-(p+1)q^p) q=-pq^{p+1},\\
    L_1 \tilde q
    %=(\la+1-(p+1)q^p)\tilde q
    =0,\\
    L_1 q_1=-2q.
    %L_2 q=(\la+1-q^p)q=0. 
    \end{cases}
    \end{align*}
 In particular, these equations imply that the zeros of $q$ and $\tilde q$ are always non-degenerate.   
    
    Our goal in this subsection is to prove the uniqueness of zeros for $\tilde q$. We first establish the existence of at least one zero. 
    
    \begin{lemma} \label{lm: qt} 
    Let $b \in S^0\cup S^-$ and $q_1(r), \tilde q(r)$, $r_0$ be defined as above. Then $q_1(r)$ has at least one zero in the interval $(0, r_0). $ Let $r_1$ denote the first zero of $q_1(r)$. Then $\tilde q(r)$ also has at least one zero in $(0, r_1)$. 
    \end{lemma}
    \begin{proof} 
    The existence of a zero for $q_1(r)$ follows from its asymptotics. Due to  the assumptions $b>0$ and $p<\frac  4{d-2}$, Lemma \ref{lm: q} implies $q_1(0+)>0$. If $b \in S^0$, 
    %If $q(r)$ is an $S^0$ solution, 
    by Lemma \ref{lm: infinity} with $c>0$, 
% since $(q, q_r)$ remains in the 4-th quadrant, 
we find $q_1(\infty)<0$. Hence $q_1(r)$ must vanish somewhere on $(0, \infty)$. For $S^-$ solutions, we have  $q_1(r_0)=\frac 2p q(r_0)+r_0q_r(r_0)<0$, so the same conclusion follows. 
    
    Let $r_1$ be the first zero for $q_1(r)$, then it holds 
    \begin{align}\label{10}
    q_1(r_1)=0, \ (q_1)_r(r_1)\le0. 
    \end{align}
    To show that $\tilde q(r)$ changes sign at least once in $(0,r_1)$, we use the following integration-by-parts identity. For any $f(r)$, $g(r)$, and interval $(s,t)$, we have 
    \begin{align}\label{ibp}
    \int_s^t g(r) \big(L_1f(r) \big) r^{d-1} dr=(-g f_r+g_r f)r^{d-1}|_s^t+\int _s^t \big(L_1 g(r) \big) f(r) r^{d-1} dr. 
    \end{align}
 Applying \eqref{ibp} to $f=q_1$, $g=\tilde q$, and interval $(0, r_1)$, noting 
 \begin{align*}
 L_1 \tilde q=0, \quad r^{d-1}\tilde q(q_1)_r |_{r=0+}=r^{d-1}\tilde q_rq_1|_{r=0+}=q_1(r_1)=0,
 \end{align*}
   we have:
   \begin{align}\label{11}
   \tilde q(r_1)(q_1)_r(r_1)=2\int_0^{r_1} \tilde q(\rho)q(\rho)\rho^{d-1} d\rho. 
   \end{align}  
    If $\tilde q(r)\ge 0$ on $(0,r_1)$, then the right-hand side of \eqref{11} is positive, which contradicts the left-hand side being nonpositive by \eqref{10}. Therefore, $\tilde q(r)$ must change sign in $(0, r_1)$, proving the lemma. 
    \end{proof}

     \begin{proposition}\label{prop: uniq}   Let $b \in S^0\cup S^-$, then the function $\tilde q(r)$ has exactly one zero on $(0,r_0)$. 
     \end{proposition}
     \begin{proof}
     We first reduce the uniqueness of zeros to the monotonicity of $\frac {q_1}{q}(r)$:   
     \begin{equation}\label{1003}
     \bigl(\frac{q_1}q\bigr)'(r)<0, \ r\in (0, r_0). 
     \end{equation}
   
  Assume \eqref{1003} holds. Suppose for contradiction that $\tilde q(r)$ has at least two zeros in $(0,r_0)$, say, at $\tilde r<r^*$. Define 
     \[
     v(r)=q_1(r)-sq(r),
     \]
     and compute 
     \begin{align*}
     L_1v(r)=L_1 q_1(r)-sL_1q(r)=-2q(r)+spq(r)^{p+1}=q(r)(sp q(r)^p-2). 
     \end{align*}
     We choose 
     \[
     s:=\frac 2{pq(\tilde r)^p}.
     \]
     Since $q(r)$ is decreasing in $(0,r_0)$, we obtain: 
     \begin{align*}
     L_1v(\tilde r)=0,\quad L_1v(r)>0 \mbox{ on } (0, \tilde r) \mbox{ and }L_1 v(r)<0 \mbox{ on } (\tilde r, r^*). 
     \end{align*}
 Applying \eqref{ibp} to $f=v$, $g=\tilde q$, and $(0, \tilde r)$ and $(0, r^*)$, respectively, we obtain: 
     \begin{align}
    0<\int_0^{\tilde r} \tilde q(r) L_1 v(r) r^{d-1} dr=(\tilde q_r v-\tilde q v_r)r^{d-1}|_0^{\tilde r}=\tilde q_r(\tilde r) v(\tilde r)\tilde r^{d-1},\label{1117}\\
    0<\int_0^{r^*} \tilde q(r)L_1 v(r) r^{d-1}dr=\tilde q_r(r^*)v(r^*) (r^*)^{d-1}. \label{1118}
     \end{align}
  Here in the last inequality we used the fact that $\tilde q$ always changes sign at its zeros.
   Along with $\tilde q_r(\tilde r)<0$ and $\tilde q_r(r^*)>0$, this implies
        \[
     v(\tilde r)<0, \ \mbox{ and } v(r^*)>0, 
     \]
or
     \begin{align}\label{1227}
     \frac{q_1}q(\tilde r)<s<\frac{q_1}q(r^*),
     \end{align}
  contradicting the monotonicity \eqref{1003}. 
  
    It remains to prove \eqref{1003}. Define the following Pohozaev-type quantities inspired by \cite{P65, Tang03}
          \begin{align}\label{3h}
      \begin{cases}
      H(r)=\frac 12 q_r^2-\frac a{2r^2} q^2-\frac 12 q^2+\frac 1{p+2} |q|^{p+2},\\
      H_1(r)=2r^d H(r)+(d-2)r^{d-1}qq_r,\\
      H_2(r)=H_1(r)+\frac p{p+2} r^d |q|^{p+2}.
      \end{cases}
      \end{align}
 Direct computation using the equation for $q$ yields:
     \begin{align*}
     %H'(r)&=-\frac{d-1} rq_r^2+\frac a{r^3}q^2;\\
     H_1'(r)&=r^{d-1}(\tfrac{4-p(d-2)}{p+2}|q|^{p+2}-2q^2);\\
     \biggl(\frac{q_1}q\biggr)'&=-\frac{H_2(r)}{r^{d-1}q^2}.
     \end{align*}
 From Lemma \ref{lm: q} %, Lemma \ref{lm: infinity}, 
    and the assumptions $\beta>0$ and $p<\frac 4{d-2}$, we deduce that:
    \begin{align}\label{hsign}
    \lim_{r\to r_0-} H(r) \ge 0, \quad \lim_{r\to r_0-}H_1(r) \ge 0, \quad \lim_{r\to 0+}H_1(r)=0. 
    %{\color{blue} H(r_0) >0, \quad H_1 (r_0) >0, \quad H_1(0) =0.}
    \end{align}
    
      %Consequently 
      %%the positivity of $H(r)$ follows quickly from 
      %$H(r_0)>0$ and  $H'(r)<0$ throughout $(0,r_0)$ yield $H(r)>0$  on $(0,r_0)$. As for $H_1$, 
  Since $q>0$ is strictly decreasing on $(0,r_0)$, we observe from the expression for $H_1'(r)$ that it can change sign exactly 
      %at most 
      once 
on this interval. Therefore, $H_1$ transitions from increasing to decreasing exactly once. Combined with the boundary behavior in \eqref{hsign}, this implies $H_1(r)>0$ for all $r\in (0, r_0)$. Consequently, $H_2(r)>0$ and thus 
      \[
       (\frac{q_1}q)'<0, \ \mbox{ on }(0,r_0). 
       \] 
              as claimed. 
  \end{proof}

 \subsection{Uniqueness of the ground state solution}

In this subsection, we complete the proof of the uniqueness of the ground state.
% by proving Lemma \ref{lm: i}: the isolation of $S^0$-point and Lemma \ref{lm: s}: the rigidity of the local structure around $S^0$ point. {\color{red} Without the singular inverse square potential, one can also check the classical treatment by McLeod \cite{McLeod}.}

We start by proving some asymptotic behavior of the linearized solution near the ground state. 

\begin{lemma}\label{lm: qb} 
Let $b_0\in S^0$ and $\tilde q(r)=q_b(b_0,r)$, then  
for some $r_0$ and $c_*>0$, 
\[
\tilde q(r), \; \tilde q_r(r) \le - c_* r^{-\frac {d+1}2} e^r, \quad \forall r > r_0.
\]

\end{lemma}
\begin{proof} 
Denote $Q(r)=q(b_0,r)$, $Q_1=\frac 2p Q+rQ_r$ and let $r_*$ be the unique zero of $\tilde q$ on $(0,\infty)$ so that $\tilde q(r)<0$ on $(r_*, \infty)$. Let $s=\frac 2{pQ^p(r_*)}$, then from the monotonicity of $Q$ we have 
\begin{align*}
L_1(Q_1-sQ)=Q(spQ^p-2)\begin{cases} >0, & r<r_*\\=0, &r=r_*,\\<0, & r>r_*.\end{cases}
\end{align*}
On the one hand, applying \eqref{ibp} with 
\[
g(r)=\tilde q(r), \ f(r)=Q_1(r)-sQ(r), 
\]
and interval $(0,r)$, we have 
\begin{align}\label{212}
0<\int_0^r \tilde q(r') L_1(Q_1-sQ) (r')^{d-1} dr'=\bigl[(r')^{d-1} (\tilde q_r f -\tilde q f_r) \bigr]\bigr |_0^r.
\end{align}
Applying Lemma \ref{lm: q}, we have the vanishing of the limit as $r\to 0^+$, and thus \eqref{212} turns into 
\begin{align}\label{213}
\tilde q_r f  - \tilde q f_r >0, \; \big( r^{d-1} (\tilde q_r f  - \tilde q f_r) \big)_r \ge 0, \quad \forall r>0.
%\lim_{r\to \infty}\bigl[\tilde q_r f r^{d-1}-\tilde q f_r r^{d-1}\bigr]>0. 
\end{align}
Using the second inequality in the above and the asymptotics of $Q_1$ from Lemma \ref{lm: infinity} with $c>0$,  there exists $r_0\gg 1$ such that for $ r \ge r_0$, we have $f(r) <0$  and, for some $c_0>0$, 
\[
(\tilde q/f)_r \ge c_0 r^{1-d} f^{-2} \ge  \frac{c_0}{2c^2} r^{-2} e^{2r}. 
\]
Integrating this inequality on $[r_0, r]$ and using the asymptotics of $Q_1$ again yield 
\[
\tilde q/f \ge c_1 +  \frac{c_0}{4 c^2}  r^{-2} e^{2r} \implies \tilde q(r) \le - c \big( \frac{c_1}{2} + \frac{c_0}{8 c^2} r^{-2} e^{2r}\big) r^{\frac {3-d}2} e^{-r}.  
\] 
The estimate on $\tilde q_r$ follows from integrating the equation 
% On the other hand, for sufficiently large $r$, from 
\[
\partial_r(r^{d-1}\tilde q_r)=r^{d-1}(\frac a{r^2}+1-(p+1)Q^p)\tilde q
\]
and the proof of the lemma is complete. 
% we know $r^{d-1}\tilde q_r$, hence $\tilde q_r$ is monotone decreasing, consequently 
% \[
% \lim_{r\to \infty}\tilde q_r(r)=c
% \]
% for some $c\in[-\infty, \infty]$. 
% 
%
% 
% The case $c>0$ can be easily excluded due to the fact that $\tilde q(r)<0$ for $r>r_*$. In the case $c=0$, the only possible scenario is $\tilde q(r)$ monotone increases to some $\tilde c\le 0$, but this immediately contradicts   \eqref{213} together with the exponential decay property of $Q$ and $Q_1$ in Lemma \ref{lm: infinity}. The remaining cases $-\infty\le c< 0$ yield the negativity of $\tilde q_r(r)$ for sufficiently large $r$ and 
% \[
% \lim_{r\to \infty} \tilde q(r)=-\infty
% \] 
% as desired. The lemma is proved. 
 %
%{\color{red}  On the other hand, for sufficiently large $r$, from 
 %\[
 %\partial_r(r^{d-1}\tilde q_r)=r^{d-1}(\frac a{r^2}+1-(p+1)Q^p)\tilde q<0,
 %\]
 %we know $r^{d-1}\tilde q_r$ is monotone decreasing, consequently $\tilde{q}$ is monotone and Proposition \ref{prop: uniq} yields
 %\[
 %\lim_{r\to \infty}\tilde q(r)=c
 %\]
 %for some $c\in[-\infty, 0]$. If $c$ is finite, then this contradicts \eqref{213} together with the exponential decay property of $Q$ and $Q_1$ in Lemma \ref{lm: infinity}. Hence, $\lim_{r\to \infty}\tilde q(r)=-\infty$. } 
\end{proof}

%In the following, we continue to take $b_0\in S^0$ such that $q(b_0, r)$ is a ground state solution satisfying the asymptotics in \eqref{18} and \eqref{64}. We will show that $b_0$ is isolated in the following sense

Next, we show that any point in $S^0$ must be isolated:

\begin{lemma}[Isolation of $S^0$ points]\label{lm: i}
For any $b_0\in S^0$, there exists $\eps>0$ such that  $(b_0-\eps, b_0+\eps) \cap S^0=\{b_0\} $. 
\end{lemma}
\begin{proof} Assume $b_0\in S^0$, so $q(b_0, r)$ is a ground state solution. From the discussion in Subsection \ref{S:22}, there exists a  sufficiently large $r_g$ such that for all $r\ge r_g-1$,
\[
(\tfrac 1r, q(b_0,r), q_r(b_0,r))\in W_{cs}^\infty,
\]
and $q(b_0,r)$ satisfies (c.f \eqref{20.3}):
\begin{align}\label{81}
q_r(b_0,r)=-q(b_0,r)+\tilde\phi^\infty(\tfrac 1r, q(b_0,r)), \; \forall r\ge r_g-1. 
\end{align}

%From the $C^1$ dependence of $q(b, r)$ on $b$ and $r$, for any $\gamma>0$, there exists a $\delta$ neighborhood $I_{b_0}$ of $b_0$ such that for all $b\in I_{b_0}$, the solution $q(b,r)$ of \eqref{12} 
%exists on $r\in [1, r_g+1]$ and 
%satisfies
%\begin{align}\label{82}
%\sup_{1\le r\le r_g+1}|q(b,r)-q(b_0,r)|+|q_r(b,r)-q_r(b_0,r)|+|q_b(b,r)-q_b(b_0,r)|\le \gamma. 
%\end{align}

Suppose, for contradiction, that the lemma fails. Then there must exist a sequence $b_n\to b_0$ with $b_n\in S^0$. By the invariance of $W_{cs}^\infty$ and the continuity of $q(b,r)$ in $b$ and $r$, it follows that
\begin{align*}
 (\frac 1{r}, q(b_n,r), q_r(b_n, r))\in W_{cs}^\infty, \mbox{ for all sufficiently large } n \mbox{ and }r\ge r_g,
 \end{align*}
 and, 
\begin{align}\label{83}
q_r(b_n, r)=-q(b_n,r)+\tilde \phi^\infty(\tfrac 1r, q(b_n,r)), \; \forall r\ge r_g. 
\end{align}
Since
%from \eqref{82}, $q_b(b_0,r)$ exists for all $b\in I_{b_0}$ and $r\in [1, r_g+1]$, 
$q(b, r)$ is $C^1$ in $b$ and $r$, we may compute the derivative $q_b(b_0,r)$ via the difference quotient:  
\begin{align*}
q_b(b_0, r)=\lim_{n\to \infty}\tfrac{q(b_n, r)-q(b_0,r)}{b_n-b_0}, \ \forall r\ge r_g.
\end{align*}
From this and applying equations \eqref{81}, \eqref{83}, we obtain: 
\begin{align}\label{85}
(q_b)_r(b_0,r)=-q_b(b_0,r)+D_q \tilde \phi^\infty(\tfrac 1r, q(b_0,r))q_b(b_0,r), \ \forall r\ge r_g.
\end{align}
Solving this linear ODE yields
\begin{align}\label{s954}
q_b(b_0,r)=q_b(b_0,r_g) e^{-\int_{r_g}^r(1- D_q\tilde\phi^\infty(\tfrac 1\rho,  q(b_0,\rho)))d\rho}. 
\end{align}
From \eqref{20.4} and the decay of $q(b_0,r)$(c.f \eqref{18}), we have 
\begin{align*}
|D_q\tilde \phi^{\infty}(\tfrac 1r, q(b_0,r))|\le \tfrac 12, \ \forall r\ge r_g.
\end{align*} 
Therefore, \eqref{s954} implies the exponential decay: 
\begin{align*}
|q_b(b_0,r)|\le C e^{-\frac r2}, \ \forall r\ge r_g.
\end{align*}
This obviously contradicts Lemma \ref{lm: qb}, hence Lemma \ref{lm: i} is proved. 
\end{proof}

Now we are ready to prove: 

\begin{lemma}[Local structure around $S^0$ point]\label{lm: s} Let $b_0\in S^0$. Then there exists $\eps>0$ such that $(b_0-\eps, b_0)\subset S^+$, $(b_0,b_0+\eps)\subset S^-$. 
\end{lemma}

\begin{proof}  Denote $Q(r)=q(b_0,r)$, $\tilde q(r)=q_b(b_0,r)$ and let $r_0$ be the unique zero of $\tilde q(r)$. Let $\delta>0$ be a small number to be specified later. By Lemma ~\ref{lm: infinity} and Lemma~\ref{lm: qb} we can choose $r^*>r_0+1$ sufficiently large such that 
\begin{align}\label{103}
\begin{cases}
(\tfrac 1{r^*}, Q(r^*), Q_r(r^*))\in W^{\infty}_{cs}, \; Q(r^*)+|Q_r(r^*)|\le \frac 12\delta,\\
 \tilde q(r^*)<0,\; \tilde q_{r}(r^*)<0
 \end{cases}
\end{align}
holds. 

We now analyze three intervals of $r$ to describe the behavior of trajectories near the ground state solution $Q(r)$. 

From Lemma~\ref{lm: z}, there exist $\varepsilon, \eta>0$ such that $z(b,r)=r^{\frac{d-2-\beta}2}q(b,r)$ satisfies \eqref{53} on $[b_0-\eps,b_0+\eps]\times[0,\eta]$, which also implies, for all $r\in[0,\eta]$, $b^+\in (b_0,b_0+\eps)$ and $b^-\in(b_0-\eps,b_0)$,
\begin{align*}
z(b^+,r)-z(b_0,r)>0, \qquad  z(b^-,r)-z(b_0,r)<0.
 \end{align*}
Consequently, 
\[
q(b^+,r)-Q(r)>0, \qquad q(b^-,r)-Q(r)<0.
\]
Next, consider the compact interval $[\eta,r^*]$, on which $q(b,r)$ is smooth in both $b$ and $r$. Using the linear approximation of the smooth function $q(b,r)$ at $b_0$, for $r\in [\eta,r^*]$ we have
\begin{align}\label{104}
\begin{cases}
q(b,r)=Q(r)+\tilde q(r)(b-b_0)+o(b-b_0), \\
q_r(b,r)=Q_r(r)+\tilde q_r(r)(b-b_0)+o(b-b_0),
\end{cases}
\end{align}
where $\frac {o(b-b_0)}{|b-b_0|} \to 0$ as $b\to b_0$.  Since $\tilde q$ changes sign exactly once due to Proposition~\ref{prop: uniq}, shrinking $\eps$ if necessary, we conclude that $q(b,r)-Q(r)$ must change to the opposite sign when $r$ varies from $r=\eta$ to $r=r^*$ for all $b\in [b_0-\eps, b_0+\eps]$.  Hence for all the above defined $b^{\pm}$, the following holds:
\begin{align}\label{105}
\begin{cases}
b^{\pm}\in S^+\cup S^-, \\
q(b^+, r^*)<Q(r^*), q_r(b^+, r^*)<Q_r(r^*),\\
q(b^-, r^*)>Q(r^*), q_r(b^-, r^*)>Q_r(r^*),\\
|q(b^\pm, r^*)-Q(r^*)|+|q_r(b^\pm, r^*)-Q_r(r^*)|\le \delta^2. 
\end{cases}
\end{align}

%\begin{align}\label{104}
%\begin{cases}
%q(b,r^*)=Q(r^*)+\tilde q(r^*)(b-b_0)+o(b-b_0), \\
%q_r(b,r^*)=Q_r(r^*)+\tilde q_r(r^*)(b-b_0)+o(b-b_0),
%\end{cases}
%\end{align}
%where $\frac {o(b-b_0)}{|b-b_0|} \to 0$ as $b\to b_0$. Using \eqref{103}-\eqref{104} together with Lemmas~\ref{lm: i} and~\ref{lm: qb}, we find that there exists $\eps>0$ such that for all 
%$b^+\in (b_0, b_0+\eps)$ and $b^-\in (b_0-\eps, b_0)$, the following holds:
%\begin{align}\label{105}
%\begin{cases}
%b^{\pm}\in S^+\cup S^-, \\
%q(b^+, r^*)<Q(r^*), q_r(b^+, r^*)<Q_r(r^*),\\
%q(b^-, r^*)>Q(r^*), q_r(b^-, r^*)>Q_r(r^*),\\
%|q(b^\pm, r^*)-Q(r^*)|+|q_r(b^\pm, r^*)-Q_r(r^*)|\le \delta^2. 
%\end{cases}
%\end{align}
This implies that at $r^*$, the point 
\[
(\tfrac 1{r^*},q(b^+,r^*), q_r(b^+,r^*))
\] is inside the corner bounded by $W^\infty_{cs}$ and the planes $q=0$, $q_r=-\delta$, $\tau=0, \tau=\delta$; while 
\[
(\tfrac 1{r^*},q(b^-,r^*), q_r(b^-,r^*))
\] 
lies in the opposite corner bounded by $W^\infty_{cs}$ and the plane $q=\delta$, $q_r=0$, $\tau=0$, $\tau=\delta$. 

For all later $r\ge r^*$, as long as the trajectory stays within
\begin{align}\label{106}
0<q(b^{\pm}, r)<\delta, \; -\delta<q_r(b^{\pm},r)<0,
\end{align}
we use equation \eqref{12} and Lemma \ref{lm: q} to estimate
\begin{align*}
q_{rr}(b^\pm,r)&=-\tfrac{d-1}r q_r(b^{\pm}, r)+\tfrac a{r^2}q(b^\pm, r)+q(b^\pm, r)-q(b^\pm, r)^{p+1} \ge \frac 12 q(b^\pm, r)>0, 
\end{align*}
provided $\delta$ and $1/r^*$ are reasonably small. 
Hence, under \eqref{106}, $q_{rr}(b^{\pm},r)>0$ so the trajectory $(\tfrac 1r, q(b^+, r), q_r(b^+, r))$ approaches the plane $q=0$, while the trajectory $(\tfrac 1r, q(b^-, r), q_r(b^-, r))$ approaches the plane $q_r=0$. %In view of the first line in \eqref{105}, the $b^\pm$ trajectories do not belong to $W_{cs}^\infty$. 

Since any trajectory is either contained in $W_{cs}^\infty$ or never intersects it, both orbits must remain within their respective corners until they hit the boundary planes:
\begin{itemize}
\item The $b^+$ trajectory must intersect $q=0$, hence $b^+\in S^-$;
\item The $b^-$ trajectory must intersect $q_r=0$, hence $b^-\in S^+$.
\end{itemize}
Thus, we conclude that: 
\[
(b_0-\eps, b_0)\subset S^+, \ (b_0, b_0+\eps)\subset S^-,
\]
proving the lemma. 
\end{proof}

\section{Construction of Local Stable Manifold}\label{S:manifold}

In this section, we employ the Lyapunov–Perron framework to construct local stable and unstable manifolds of the two-dimensional manifold of ground states. These ground states are generated by applying rotational and scaling symmetries to the ground state solution $Q$ (see Remark \ref{R:1.3}). The construction proceeds by first establishing the stable and unstable manifolds of the single ground state $Q$, and then extending them via the action of these symmetries.

\subsection{Equation for perturbations}
We begin by expressing solutions $u(t,x)$ of NLS$_a$ in a rotating frame:
\[u(t,x)=e^{it}(Q(x)+v(t,x)).\]
Then $v=v_1+iv_2$ satisfies the equation 
\begin{align}\label{1205}
iv_t-\la v-v+(p+1)Q^p v_1+iQ^p v_2-iR(v)=0, 
\end{align}
where the remainder $R(v)$ is defined through
\begin{align}\label{r}
-iR(v)=|Q+v|^p(Q+v)-Q^{p+1}-(p+1)Q^pv_1-iQ^p v_2.
\end{align}
Putting the linear part in vector form, the equation for $v$ can be rewritten as 
\begin{align}\label{1102}
v_t=JL v+R(v),
\end{align}
where 
\[-i\sim J=\begin{pmatrix}0&1\\-1&0\end{pmatrix}, \ L=\begin{pmatrix}L_1&0\\0& L_2\end{pmatrix},\]
and 
\begin{align}\label{l1l2}
L_1=\la+1-(p+1)Q^p, \quad L_2=\la+1-Q^p. 
\end{align}

In view of Lemma \ref{lm: l2} and \eqref{503} in the Appendix, along with the fact that $\langle LQ,Q\rangle<0$, we are able to apply Theorem 2.1 in \cite{LZ22} for general linear Hamiltonian systems to obtain the following linear exponential dichotomy:

\begin{proposition}\label{prop:linzeng}
The flow $e^{tJL}$ is a well-defined group of bounded linear operators on $(H^1(\R^d))^2$ and there exist closed subspaces $E^u,\ E^s$ and $E^c$ such that 

a) 
%$H^1 = E^u \oplus E^s \oplus E^c$ and 
$\dim E^u=\dim E^s=1$.

b) $e^{tJL}(E^{u,s,c})=E^{u,s,c}$.

c) $ \langle Lu, u\rangle =0, \; \forall u\in E^{u,s},$ and $$E^c=\{u\in (H^1(\R^d))^2 \mid \langle Lu, v\rangle=0, \ \forall v\in E^u\oplus E^s\}.$$

%d) There exist $c,\lambda>0$ such that 
%$$
%JLu=\pm \lambda u, \forall u\in E^{u,s}, \mbox{ hence }\biggl |e^{tJL}|_{E^{u,s}}\biggr|\le Ce^{\mp \lambda t}, \ \forall t\ge 0. 
%$$

d) $\bigl |e^{tJL}|_{E^c}\big|\le C(1+|t|)$, $\forall t\in \mathbb R $.

e) There exist $\epsilon >0$, $e_0>0$, and a closed subspace $E^e$ such that $E^c=\ker L\oplus E^e$ and $(H^1(\R^d))^2=E^u\oplus E^s\oplus \ker L \oplus E^e$. Moreover,  

$$L\sim 
\begin{pmatrix} 0 & 1 & 0 & 0\\
1&0&0&0\\
0&0&0&0\\
0&0&0&L_e
\end{pmatrix}, \quad 
JL\sim \begin{pmatrix}
e_0 &0&0&0\\
0&-e_0 &0&0\\
0&0&0&A_{0e}\\
0&0&0&A_e
\end{pmatrix}
$$
and 
$$
L_e\ge \epsilon>0, \quad \langle L_e e^{tA_e}u, e^{tA_e}v\rangle=\langle L_e u, v\rangle.
$$   
Here $\sim$ indicates that the operators $L$ and $JL$ take those block forms in the direct sum decomposition. $L_e: E^e\to (E^e)^*$ is defined as $\langle L_e u, v\rangle=\langle Lu, v\rangle$, for any $u, v\in E^e$. Operators $A_e$ and $A_{0e}$ are defined as the projections of $JL|_{E^c}$ onto subspaces $E^e$ and   $\ker L$ based on the above direct sum decomposition. In particular, $A_{0e} \in L(E^e, \ker L)$.      
\end{proposition}

In the rest of this section, we will assume $V^{\pm}\in H^1(\R^d)$ are the eigenfunction taken from $E^u$ and $E^s$ satisfying
\begin{align*}
JL V^{\pm}=\pm e_0 V^{\pm}, \quad \langle LV^+, V^-\rangle =1. 
\end{align*}

\subsection{Preliminary estimates}

   We begin by giving pointwise estimates for the remainder term $R(v)$. For clarity and future reference, we note that these estimates hold for any exponent $p>0$ and any function $Q\in C^1 (\R^d \setminus \{0\})$ and $Q(x)\neq 0, \ \forall x$.

\begin{lemma}\label{lm:rv} Let $R(v)$ be the remainder term defined in \eqref{r}. Then in the case $0<p\le 1$, we have
\begin{align}\label{R1}
\begin{cases}
|R(v)|\lsm |v|^{p+1}, \qquad |R(v_1)-R(v_2)|\lsm |v_1-v_2|(|v_1|^p+|v_2|^p), \\
|\nabla R(v)|\lsm |v|^p \bigl(|\nabla v|+|Q^{-1}\nabla Q v|\bigr),\\
|\nabla R(v_1)-\nabla R(v_2)|\lsm  \bigl(|Q^{-1}\nabla Q (v_1-v_2)|+|\nabla (v_1-v_2)|\bigr)\sum_{i=1}^2 |v_j|^p\\
\qquad \qquad\qquad+|v_1-v_2|^p\sum_{j=1}^2\bigl(|\nabla v_j|+|Q^{-1}\nabla Q v_j|\bigr). 
\end{cases}
\end{align}
In the case $p>1$, we have
\begin{align}\label{R2}
\begin{cases}
|R(v)|\lsm Q^{p-1}|v|^2+ |v|^{p+1}, \qquad |R(v_1)-R(v_2)|\lsm |v_1-v_2|\sum_{j=1}^2(Q^{p-1}|v_j|+|v_j|^p), \\
|\nabla R(v)|\lsm (Q^{p-1}|v|+|v|^p) \bigl(|\nabla v|+|Q^{-1}\nabla Q v|\bigr),\\
|\nabla R(v_1)-\nabla R(v_2)|\lsm  \bigl(|Q^{-1}\nabla Q (v_1-v_2)|+|\nabla (v_1-v_2)\bigr)\sum_{j=1}^2(Q^{p-1}|v_j|+ |v_j|^p)\\
\qquad \qquad\qquad+|v_1-v_2|\big(\sum_{j=1}^2 (|\nabla v_j|  +|Q^{-1}\nabla Q v_j|)\big)
%+|Q^{-1}\nabla Q v_i|\bigr)
\bigl(Q^{p-1}+\sum_{j=1}^2|v_j|^{p-1}\bigr). 
\end{cases}
\end{align}
In all the above estimates, the implicit constants depend only on $p$. 
\end{lemma}
\begin{proof}
It suffices to prove the difference estimates, as the rest can be obtained by taking  $v_2=0$.

We introduce 
\[
J(z)=|1+z|^p(1+z)-1-\tfrac{p+2}2 z-\tfrac p2 \bar z.
\]
This function appears naturally in the following equations: 
\begin{align}\label{r4}
-iR(v)=Q^{p+1}J(z),\ \mbox{ for } z=\tfrac vQ, 
\end{align}
and 
\[
-i(R(v_1)-R(v_2))=Q^{p+1}(J(z_1)-J(z_2)) \ \mbox{for } z_i=\tfrac{v_i}Q.
\]
Differentiating $J(z)$ in $z$ and $\bar z$ yields
\begin{align*}
J_z(z)=\tfrac{p+2}2(|1+z|^p-1), \ \ J_{\bar z}(z)=\tfrac p2[ |1+z|^{p-2}(1+z)^2-1], 
\end{align*}
which are special cases for a broader class of functions:
\[
J_1(a, b)=\tfrac{p+2}2[|b+a|^p-|b|^p], \ \ J_2(a,b)=\frac p2[|b+a|^{p-2}(b+a)^2-|b|^{p-2}b^2],
\]
for $a=z$ and $b=1$. 

We now establish bounds for $J_1$, $J_2$ by considering two cases. When $|a|<\frac 12|b|$, $J_1$, $J_2$ are $C^1$ with derivatives bounded by $C|b|^{p-1}$. The Mean Value Theorem gives 
\[
|J_1|+|J_2|\lsm |a| |b|^{p-1}. 
\]
For $|a|\ge \frac 12 |b|$, we have  
\[
|J_1|+|J_2|\lsm |a|^p. 
\]
These combine to give the bound
\begin{align}\label{r6}
|J_1|+|J_2|\lsm \begin{cases}
|a|^p, & 0 < p\le 1;\\
|a|(|a|^{p-1}+|b|^{p-1}), & p>1. 
\end{cases}
\end{align}
Specializing to $a=z$ and $b=1$ yields:
\begin{align}\label{r41}
|J_z(z)|+|J_{\bar z}(z)|\lsm \begin{cases}
|z|^p, & 0<p\le 1;\\
|z|(|z|^{p-1}+1), & p>1,
\end{cases}
\end{align}
which then implies 
\begin{align}\label{r42}
|J(z_1)-J(z_2)|\lsm \begin{cases}
|z_1-z_2|(|z_1|^p+|z_2|^p), &\ 0<p\le 1;\\
|z_1-z_2|(|z_1|+|z_2|)(|z_1|^{p-1}+|z_2|^{p-1}+1), & p>1.
\end{cases}
\end{align}
The estimates of $|R(v_1)-R(v_2)|$ then follow from \eqref{r42} with $z_i=\tfrac{v_i}Q$, $i=1,2$.

The estimates of the derivative differences rely on both \eqref{r42} and the following: 
\begin{align}\label{r5}
|J_z(z_1)-J_z(z_2)|+|J_{\bar z}(z_1)-J_{\bar z} (z_2)|\lsm \begin{cases}
|z_1-z_2|^p, & p\le 1, \\
|z_1-z_2|(|z_1|^{p-1}+|z_2|^{p-1}+1), & p>1.
\end{cases}
\end{align}
which can be derived from \eqref{r6} by taking $a=z_1-z_2$ and $b=1+z_2$. 

We now proceed to write the derivative differences as
\begin{align}
-i[\nabla R(v_1)-\nabla R(v_2)]=&(p+1) Q^p\nabla Q[J(\tfrac{v_1}Q)-J(\tfrac{v_2}Q)]\label{R6}\\
&+Q^{p+1}[J_z(\tfrac{v_1}Q)\nabla(\tfrac{v_1}Q)-J_z(\tfrac{v_2}Q)\nabla(\tfrac{v_2}Q)]\label{r7}\\
&+Q^{p+1}[J_{\bar z}(\tfrac{v_1}Q)\nabla(\tfrac{\bar v_1}Q)-J_{\bar z}(\tfrac{v_2}Q)\nabla(\tfrac{\bar v_2}Q)].\label{r8}
\end{align} 
The estimate of \eqref{R6} follows from \eqref{r42} for $z_i=\tfrac {v_i}Q$ and \eqref{r7} can be rewritten as
\begin{align*}
\eqref{r7}=Q^{p+1}&[J_z(\tfrac{v_1}Q)-J_z(\tfrac{v_2}Q)](\tfrac{\nabla v_1}Q-\tfrac{\nabla Q v_1}{Q^2})+Q^{p+1}J_z(\tfrac{v_2}Q)\bigl(\tfrac{\nabla(v_1-v_2)}Q-\tfrac{\nabla Q(v_1-v_2)}{Q^2}\bigr). 
\end{align*}
The above two terms can be controlled using \eqref{r5} and \eqref{r41} for $z_i=\tfrac {v_i}Q$ respectively. Since \eqref{r8} is conjugate to \eqref{r7}, we complete the estimate of the derivative difference, hence complete the proof of Lemma \ref{lm:rv}. 
\end{proof}

\subsection{Strichartz estimate of the linear flow}

Next, we prove a linear Strichartz estimate for the flow $e^{tJL}$. To this end, we define 
\begin{align}\label{qqstar}
r=\tfrac{2d(p+1)}{d-2},\ r^*=\tfrac{2d(p+1)}{2p+d}, \ q^*=\tfrac{4{(p+1)}}{p(d-2)}. 
\end{align}
It is easy to check 
\begin{align}\label{jj5}
\tfrac 1r = \tfrac 1{r^*} - \tfrac 1d, \quad  \tfrac 12=\tfrac 1{r^*}+\tfrac pr,\quad \tfrac 2{q^*}+\tfrac d{r^*}=\tfrac d2, 
\end{align}
hence, $H^{1, r^*}(\R^d)\subset L^r (\R^d)$ and $(q^*,r^*)$ is an admissible pair. We thus define the Strichartz space over a time interval $I$ by: 
\begin{equation}\label{snorm}
S^1(I)=L_t^\infty H^1\cap L_t^{q^*}  H^{1,r^*}(I\times\R^d). 
\end{equation}
The homogeneous Sobolev norm $\|\cdot \|_{\dot H^{1, r^*}}$ will be estimated via the operator $(\la)^{\frac12}$  
using the equivalence of Sobolev norms developed in an earlier work \cite{KMVZZ18}. In the following lemma, we record a specific version applied to our setting and state the integrability of $Q$:

\begin{lemma}\label{lm: normequal}
Assume $\tfrac 4d<p<\tfrac 4{d-2}$ and $a> \tfrac{(d-2)^2}4\bigl(-1+\frac{p^2}{(p+1)^2}\bigr)$, then we have 
\[
\|\nabla f\|_{L_x^m}\sim \|(\la)^{\frac 12} f\|_{L_x^m}, \quad  \forall \ 2\le m\le r^*, \; f\in C_c^\infty(\R^d). 
\]
Moreover, the asymptotic behavior of $Q$ guarantees
\begin{align}\label{qspace}
Q\in L^1(\R^d)\cap H_a^{1,r^*}(\R^d)\cap L^r (\R^d). 
\end{align}
\end{lemma}

We first prove the linear Strichartz estimate in the central space $E^c$ with algebraic growth in time. 
 An alternative approach is also possible, see \cite{YZ22}.

\begin{lemma}\label{Lemma3.1}
	Let $u_0\in E^c$, $f(t)\in E^c$ and $T>0$. Let 
\[u(t,x)=e^{tJL} u_0, \quad v(t,x)=\int_0^t e^{(t-s)JL}f(s)ds,\quad w(t,x)=\int_t^T e^{(t-s)JL}f(s)ds.\]
Then,  
	\begin{align}
		\|u(\pm t)\|_{ H^1} &\lesssim \left\langle t \right\rangle \|u_0\|_{ H^1}, \label{H1bd}\\
		% \|u(\pm t)\|_2&\lsm \|u_0\|_2+\left\langle t\right\rangle^2\|u_0\|_{\dot H^1},\label{607}\\
		\|u\|_{ S^1([0,T])} &\lesssim \left\langle T \right \rangle^2 \|u_0\|_{ H^1}. \label{local u}\\
		\|v\|_{ S^1([0,T])}+\|w\|_{ S^1([0,T])}&\lsm \langle T\rangle ^2 \|f\|_{L_t^1 H^1([0,T])}. \label{918}
	\end{align}
\end{lemma}

\begin{proof}  We begin by estimating the $H^1$ norm of $u$. According to the block decomposition of the operator $JL$ in Proposition \ref{prop:linzeng}, for any $u_0\in E^c$, we decompose $u_0=u_0^k+u_0^e$ where $u_0^k\in \ker L$, $u_0^e\in E^e$. The solution evolves as:
	\begin{align*}
		u(t)=e^{tJL}u_0 &= 
		\begin{pmatrix}
			I &  \int_0^t A_{0e}e^{\tau A_e}d\tau \\
			0 & e^{tA_e}
		\end{pmatrix}
		\begin{pmatrix}
			u_0^k \\
			u_0^e
		\end{pmatrix}
		=
		\begin{pmatrix}
			u_0^k +\int_0^t A_{0e} e^{\tau A_e}u_0^e d\tau \\
			e^{tA_e}u_0^e
		\end{pmatrix}.
	\end{align*}
As the second row is under control due to the ellipticity and the invariance of $L_e$: 
		$$
	\|e^{tA_e}u_0^e\|_{ H^1} \sim \|u_0^e\|_{H^1} \lesssim \|u_0\|_{ H^1},
	$$
we immediately obtain:
	\begin{align*}
		\|u(t)\|_{ H^1} \le \|e^{tA_e}u_0^e\|_{H^1} + \|u_0^k\|_{ H^1} 
                  +\int_0^t\|A_{0e} e^{\tau A_e}u_0^e\|_{ H^1} d\tau 
		 \lesssim \langle t \rangle \|u_0\|_{ H^1}.
	\end{align*}
Here we used the boundedness of $A_{0e}$ in $H^1(\R^d)$ and thus \eqref{H1bd} is proved. 
	
	Next, we prove \eqref{local u} using the Strichartz estimate for $e^{it\la}$, simply treating other terms as perturbations. Fix $\eta>0$ and partition$[0,T]$ as:
	$$
	[0, T] = \bigcup_{j=0}^{N} I_j, \; \mbox{ with } \ I_j=[j\eta, (j+1)\eta], \; j\le N-1; \; I_{N}=[N\eta,T].
	$$
	On each interval $I_j$, applying Strichartz estimate for $e^{it\la}$ in \cite{BPSTZ03, KT98, ZZ20} and using H\"older inequality and \eqref{jj5} yields: 
	\begin{align*}
		\|u\|_{ S^1(I_j)} &\lesssim \|u (j\eta)\|_{H^1} +
		    \|(\la+1)^{\frac 12} \bigl(u-Q^p((p+1)u_1+iu_2)\bigr)\|_{L^1_tL_x^2(I_j)} \\
		 &\lesssim \langle j\eta \rangle \|u_0\|_{H^1} +
		   \sum_{i=1}^2 \bigl(\eta \|u_i\|_{L_t^\infty H^1(I_j)}+\|Q^p\langle \nabla \rangle u_i\|_{L_t^1L_x^2(I_j)} +\|\langle  \nabla \rangle QQ^{p-1} u_i\|_{L_t^1L_x^2(I_j)} \bigr)\\
		   &\lsm \langle j\eta \rangle \|u_0\|_{H^1} +
		   \eta \|u\|_{L_t^\infty H^1(I_j)}\\
		   &\qquad+\eta^{1-\frac 1{q^*}}\bigl(\|Q\|_{L_x^r}^p \|\langle \nabla \rangle u\|_{L_t^{q^*}L_x^{r^*}(I_j)}+\|\langle \nabla \rangle Q\|_{L_x^{r^*}}\|Q\|_{L_x^r}^{p-1}\|u\|_{L_t^{q^*}L_x^r(I_j)}\bigr)\bigr]\\
		   &\lsm \langle j\eta\rangle \|u_0\|_{ H^1}+\eta \|u\|_{L_t^\infty H^1(I_j)}+\eta^{1-\frac 1{q^*}}\|u\|_{S^1(I_j)}. 	   \end{align*}	      
	For sufficiently small $\eta$, we obtain
	$$
	\|u\|_{ S^1(I_j)} \lsm\langle j\eta \rangle \|u_0\|_{H^1} 
	$$
	and \eqref{local u} follows by summing over $j$. 	
				
   For $v-$estimates (with $w$ being similar), from its equation:
	$$
	iv_t=(\la+1) v-((p+1)Q^pv_1+iQ^pv_2)+if, 
	$$ 
we apply Strichartz estimates on each subinterval $I_j$ to get:
	\begin{align}
		\|v\|_{ S^1(I_j)} &\lesssim \|v(j\eta)\|_{H^1} + \|v-\bigl((p+1)Q^p v_1 +iQ^pv_2\bigr)\|_{L_t^1 H^1(I_j)}
                  +\|f\|_{L_t^1H^1(I_j)}\label{JJ6} \\
		&\lesssim \|v(j\eta)\|_{H^1} +\eta^{1-\frac 1{q^*}}\|v\|_{ S^1(I_j)} +\|f\|_{L_t^1H^1(I_j)}. \notag
	\end{align}
Note from \eqref{H1bd} and $v(0)=0$, we have
	\begin{align*}
		\|v(t)\|_{ H^1} &\le \int_0^t\|e^{(t-s)JL}f(s)\|_{H^1} ds 
		 \lesssim \int_0^t\langle t \rangle \|f(s)\|_{H^1} ds 
		\le \langle t \rangle \|f\|_{L_t^1H^1([0,t])},
	\end{align*}
Taking $\eta$ sufficiently small in \eqref{JJ6} yields:
	$$
	\|v\|_{S^1(I_j)} \lesssim  \langle j\eta \rangle  \|f\|_{L_t^1  H^1([0,(j+1)\eta))},
	$$
	and summing over $j$ gives \eqref{918}.
	\end{proof}

\subsection{Solving equation \eqref{1102}}\label{SS:stableM}

We now establish the existence and uniqueness of an exponentially decaying solution to equation \eqref{1102}. To avoid technical complications arising from rough nonlinearities, we assume $p\ge 1$. Thus, the admissible range for $p$ becomes 
\[ 
p\in (\tfrac 4d, \tfrac 4{d-2})\cap [1,\infty),
\]
which is nonempty in dimensions $3\le d\le 5$. 

\begin{remark}
The case $p<1$ may also be tractable via a partial contraction argument provided the Strichartz estimate for $e^{tJL}$ holds in $(L^2(\R^d))^2$, even with weak exponential temporal growth. This extension is nontrivial and will be addressed in future work. 
\end{remark}

Combining these restrictions, we state our result: 

\begin{theorem}\label{exiuni} 
Let $d=3,4,5$ and $p\in (\tfrac 4d, \tfrac 4{d-2})\cap [1,\infty)$. Let $0>a>\tfrac{(d-2)^2}4(-1+\frac{p^2}{(p+1)^2}).$ There exists a constant $C>0$ such that, for any $\lambda\in(0, e_0]$ and $y_0^- \in (-\delta, \delta)$ where 
% $\delta\in(0,\delta_\lambda)$ with 
$\delta=\frac 1{C} \min(\lambda,\lambda^4)$, equation \eqref{1102} admits a unique solution satisfying  
\begin{equation} \label{E:class}
v(0)=y_0^-V^-+y_0^+V^++v^c(0), \textit{ and } \|v(t) \|_{H^1}
%\dot S^1([t,\infty)\times\R^3)}
\le C \delta e^{-\lambda t},  \quad \forall t\ge 0.
\end{equation}
Moreover, the following estimates hold:
\begin{align}\label{438}
& |y_0^+|+\|v^c(0)\|_{H^1_x}\le C |y_0^-|^2, \\
& \|v\|_{ S^1([t,\infty))}\le C^2 \delta e^{-e_0 t}, \quad \forall t\ge 0. 
\end{align}
For any $y_0$, $\tilde y_0$ with $y_0^-\tilde y_0^->0$ and $|y_0^-|, |\tilde y_0^-|< \delta$, the corresponding solutions satisfy $v(t,x)=\tilde v(t+T,x)$ for some $T=T(y_0, \tilde y_0)$. 
\end{theorem}

\begin{proof} 
Using the decomposition $(H^1(\R^d))^2=E^s\oplus E^u\oplus E^c$ from Proposition \ref{prop:linzeng}, we express: 
\begin{align}\label{decu}
v=y^+V^++y^-V^-+v^c
\end{align}
where $y^{\pm}=\langle LV^{\mp}, v\rangle$ and $v^c=v-y^-V^--y^+V^+$. Equation \eqref{1102}  decouples into:
\begin{align*}
\begin{cases}
\dot y^-&=-e_0 y^-+R^-(v)\\
\dot y^+&=e_0 y^++R^+(v)\\
\frac{\partial}{\partial t}v^c&=JL v^c+R^c(v), 
\end{cases}
\end{align*}
whose Duhamel formulation (incorporating exponential decay) becomes:
\begin{align}\label{inteqn}
\begin{cases}
y^-(t)&=e^{-e_0 t} y_0^-+\int_0^t e^{-e_0(t-s) }R^-(v(s))ds\\
y^+(t)&=-\int_t^\infty e^{e_0(t-s)}R^+(v(s))ds\\
v^c(t)&=-\int_t^\infty e^{JL(t-s)}R^c(v(s))ds. 
\end{cases}
\end{align}
Here, $R^{\pm}(v)$ and $R^c(v)$ are defined via
\begin{align*}
R^{\pm}(v)=\langle LV^{\mp}, R(v)\rangle, \ R^c(v)=R(v)-R^+(v)V^+-R^-(v)V^-. 
\end{align*}
For functions $v(t, x)$ expressed as in \eqref{decu}, define the norm 
\[
\vertiii{ v}=\max\bigl\{\sup_{t\ge 0} e^{\lambda t}|y^+(t)|, \ \sup_{t\ge 0} e^{\lambda t}|y^-(t)|,\ \sup_{t\ge 0} e^{\lambda t}\|v^c\|_{S^1([t,\infty))}\bigr\},
\]
and define the mapping: 
\[
\Phi (v) (t,x) = \tilde v (t, x)  = \tilde y^+  (t) V^+(x) + \tilde y^-(t) V^-(x) + \tilde v^c(t, x),
\]
where $\tilde y^\pm (t)$ and $\tilde v^c$ are given by the right sides of \eqref{inteqn}. 

% we and equip the space with the norm:
We shall prove $\Phi$ is a contraction on the closed ball:
%defined by 
\begin{align*}
B_{\delta,\lambda}=\bigl \{v = y^+ V^+ + y^- V^-+ v^c \bigl | y^{\pm} \in C^0([0, \infty))\;, v^c\in  S^1 ([0, \infty)),\; \vertiii{v}\le 2\delta.\bigr\}
\end{align*}
It is easy to see $B_{\delta, \lambda}$ increases in $\delta$ and decreases in $\lambda$. 

For $v_1, v_2 \in B_{\delta, \lambda}$ expressed as:  
\[v_j=y_j^+V^++y_j^-V^-+v_j^c\in B_{\delta, \lambda}, \; j=1,2. \]
we denote their images under $\Phi$ by $\tilde v_1$ and $\tilde v_2$. The contraction property follows from estimates of the nonlinear terms $R^{\pm}$ and $R^c$, which we now detail. 

\begin{claim}\label{JJ3} For $p \in [1, \frac 4{d-2})$, there exists $C>0$ such that, for any $v_1, v_2 \in B_{\delta, \lambda}$, the following estimates hold:
\begin{align*}
\begin{cases}
|R^{\pm}(v_1(s))|\le C\delta^2 e^{-2\lambda s},\\
|R^{\pm}(v_1(s))-R^{\pm}(v_2(s))|\le C \delta e^{-2\lambda s}\vertiii{v_1-v_2},\\
\|R^c(v_1)\|_{L_t^1 H^1([T_0, T_0+T_1])}\le C\langle T_1\rangle^{1-\frac 2{q^*}}\delta^2 e^{-2\lambda T_0}, \\
\|R^c(v_1)-R^c(v_2)\|_{L_t^1 H^1([T_0, T_0+T_1])}\le C\langle T_1\rangle ^{1-\frac 2{q^*}}\delta e^{-2\lambda T_0}
\vertiii{v_1-v_2}. 
\end{cases}
\end{align*}
\end{claim}
We focus on proving the difference estimates (the second and the fourth inequalities) as the rest follow by setting one of $v_i=0$. 

From the pointwise estimates \eqref{R2} and H\"older inequality, we compute: 
\begin{align*}
|R^{\pm}(v_1(s))-R^{\pm}(v_2(s))|&=|\langle LV^{\mp}, R(v_1)-R(v_2)\rangle |\\
&\lesssim  \|LV^{\mp}\|_{L_x^{\frac{4d}{2d-p(d-2)}}}\|v_1-v_2\|_{L_x^{\frac{4d}{2d-p(d-2)}}} \sum_{j=1,2}\bigl(\|Q\|_{L_x^{2^*}}^{p-1}\|v_j\|_{L_x^{2^*}}+\|v_j\|_{L_x^{2^*}}^p\bigr)
\\
&\lesssim \|v_1-v_2\|_{H^1}\sum_{j=1,2}( \|v_j\|_{H^1}+\|v_j\|_{H^1}^p)\\
&\le C \delta e^{-2\lambda s}\vertiii{v_1-v_2}. 
\end{align*}
Here we have used $\frac{4d}{2d-p(d-2)} \in (2, 2^*)$.

Similarly, for the gradient term, we have
\begin{align}
\|\nabla(R(v_1)-R(v_2))&\|_{L_t^1 L_x^2([T_0,T_0+T_1])}\le \|\nabla(v_1-v_2)\|_{L_t^{q^*}H_x^{1, r^*}}\label{525}\\
&\qquad\cdot\bigl[
T_1^{1-\frac 2{q^*}}\|Q\|_{L_x^{r}}^{p-1}\sum_{j=1,2}\|v_j\|_{L_t^{q^*}L_x^r}
+T_1^{1-\frac {p+1}{q^*}}\sum_{j=1,2}\|v_j\|_{L_t^{q^*}L_x^r}^p\bigr]\notag\\
&+\|v_1-v_2\|_{L_t^{q^*}L_x^{r}}\sum_{j=1,2} \big( \|\nabla v_j\|_{L_t^{q^*}L_x^{r^*}} +  \|v_j\|_{L_t^{q^*}L_x^{r^*}}  \big) \notag\\
&\qquad\cdot\bigl[T_1^{1-\frac 2{q^*}}\|Q\|_{L_x^{r}}^{p-1}+T_1^{1-\frac {p+1}{q^*}}
\sum_{j=1,2}\|v_j\|_{L_t^{q^*}L_x^r}^{p-1}\bigr]\notag\\
&\le C\langle T_1\rangle^{1-\frac {2}{q^*}}\delta e^{-2\lambda T_0}\vertiii{v_1-v_2}. \notag
\end{align}
The same bound holds for: 
\[
\|R(v_1)-R(v_2)\|_{L_t^1L_x^2([T_0, T_0+T_1])}\le C\langle T_1\rangle^{1-\frac {2}{q^*}}\delta e^{-2\lambda T_0}\vertiii{v_1-v_2}.
\]
These two estimates together with the boundedness of the projection onto center space yield the last estimate in Claim \ref{JJ3}.

We now verify the contraction property for $\Phi$. For any $v_1, v_2\in B_{\delta, \lambda}$, Claim \ref{JJ3} implies 
\begin{align*}
	e^{\lambda t}\bigl|\tilde y_1^-(t)-\tilde y_2^-(t)\bigr| &\le  e^{\lambda t}
	  \int_0^t e^{-e_0(t-s)}\bigl |R^-(v_1(s))-R^-(v_2(s))\bigr|ds \\
	&\le C\delta \int_0^t e^{(\lambda- e_0)(t-s)} e^{-\lambda s}ds \vertiii{v_1-v_2}\\
	&\le \frac 14\vertiii{v_1-v_2}, \; \forall t\ge0,
\end{align*}
for $\delta$ sufficiently small as in Theorem \ref{exiuni}. Similarly, 
\begin{align}
		e^{\lambda t}\bigl |\tilde y_1^+(t)-\tilde y_2^+(t)\bigr |&\le C\delta\int_t^{\infty}e^{(\lambda+e_0)(t-s)}e^{-\lambda s}ds\vertiii{v_1-v_2}\le \frac 14\vertiii{v_1-v_2}. \label{1009}
\end{align}
%for the same choice of $\delta$. 

For Strichartz estimate of $\tilde v_1^c-\tilde v_2^c$, we decompose
\begin{align*}
	\|\tilde v_1^c-&\tilde v_2^c\|_{ S^1([T,\infty))}\\
%	\biggl\|\int_t^{\infty} e^{(t-s)JL}R^c(v(s))ds\biggr \|_{\dot S^1([T,\infty))} \\
%	%&\le \sum_{N\ge 1, N\in\mathbb N
	%			}\biggl\|\int_t^{\infty} e^{(t-s)JL} R^c(v(s))ds \biggr\|_{\dot S^1([NT, (N+1)T])} \\
	&\le \sum_{
		N\ge 1, N\in\mathbb N
		} \biggl\|\int_t^{T +N}e^{(t-s)JL} \bigl(R^c(v_1(s))-R^c(v_2(s))\bigr)ds\biggr\|_{ S^1([T+ N-1, T+ N])}  \\
	&\quad +\sum_{
		N\ge 1, N\in\mathbb N
		} \biggl\|\int_{T+N}^{\infty} e^{(t-s)JL}\bigl(R^c(v_1(s))-R^c(v_2(s))\bigr)ds\biggr\|_{S^1([T+N-1, T+ N])} \\
	&:= I + II.
\end{align*}
Using \eqref{918} and Claim \ref{JJ3} yields: 
\begin{align*}
	I &\le \sum_{
		N\ge 1, N\in\mathbb N} 
		%\langle T\rangle^2 
		\|R^c(v_1(s))-R^c(v_2(s))\|_{L_t^1 H^1([T+N-1, T+N])} \\
	&\le C\delta\vertiii{v_1-v_2} \sum_{
		N\ge 1, N\in\mathbb N		} 
		%\langle T\rangle^3 
		 e^{-2\lambda (T+N-1)}\\
	& \le \frac 1{\lambda} C \delta  e^{-2\lambda T}\vertiii{v_1-v_2}\le \frac 18 \vertiii{v_1-v_2}e^{-2\lambda T}.
	 % \le \delta e^{-\lambda T}.
\end{align*}
For the term $II$, we further partition the integral and compute 
\begin{align*}
	II &\le \sum_{
		N\ge 1	}\sum_{
			 M\ge N+1
		 }\biggl \|\int_{T+M-1}^{T+M} e^{(t-s)JL} 
	     \bigl (R^c(v_1(s))-R^c(v_2(s))\bigr)ds\biggr\|_{ S^1([T+N-1, T+N])} \\
	   &\le \sum_{
			   N\ge 1,
			   M\ge N+1
               }\int_{T+M-1}^{T+M}\bigl \|e^{(t-s)JL}  \bigl (R^c(v_1(s))-R^c(v_2(s))\bigr)\bigr\|_{ S^1([T+N-1, T+N])}ds. 
\end{align*}
Note $0 \le s-t \le M+1$, applying Lemma \ref{Lemma3.1} and Claim \ref{JJ3} we obtain
\begin{align*}               
	   II&\le \sum_{
	N\ge 1, M\ge N+1        
	      }\int_{T+M-1}^{T+M} (M+1)^2\bigl \|R^c(v_1(s))-R^c(v_2(s))\bigr\|_{ H^1} ds\\
	   &\le \sum_{
			   M\ge2
                         }(M+1)^3 \|R^c(v_1(s))-R^c(v_2(s))\|_{L_t^1 H^1([T+M-1, T+M])} \\
	   &\le C\delta \vertiii{v_1-v_2} \sum_{
			   M\ge 2
                   	      }(M+1)^3  e^{-2\lambda (T+M-1)} \\
%	   &\lesssim (2\delta)^2 T^2 e^{-\frac 32 e_0 T} \\
	  & \le \frac C{\lambda^4} \delta e^{-2\lambda T}\vertiii{v_1-v_2}\le \frac 18\vertiii{v_1-v_2} e^{-2\lambda T}. 
	   \end{align*}
Combining these yields:
%and combining the case $T<1$, 
\begin{align}\label{1016}
\sup_{T\ge 0}e^{\lambda T}\|\tilde v^1_c-\tilde v^2_c\|_{ S^1([T,\infty))}\le\frac 14 \vertiii{v_1-v_2}.
\end{align}
This proves that 
\begin{align*}
\vertiii{\tilde v_1-\tilde v_2}\le \frac 12\vertiii{v_1-v_2}.
\end{align*}
Since $\vertiii {\Phi(0)} = |y_0^-| < \delta$,  the contraction estimate $Lip( \Phi) \le \frac 12$ also implies that  
%A similar argument shows 
$\Phi$ maps $B_{\delta,\lambda}$ into itself. The existence and uniqueness of solutions to \eqref{inteqn} in $B_{\delta,\lambda}$ are proved.

To prove the uniqueness in the class defined by \eqref{E:class}, we verify that 
exponential decay in $H^1(\R^d)$ implies the same decay as in the Strichartz space $S^1$. Specifically, given 
%solution  
%assume $v(t)$ is a solution to \eqref{1102} satisfying   
\[
\|v(t) \|_{H^1} \lsm e^{-\lambda t}, \quad t\ge 0,
\]
we shall prove
\begin{align}\label{a101}
\| v\|_{S^1 ([t, \infty))} \lsm e^{-\lambda t}, \quad t\ge 0.
\end{align}
Indeed, for any $\tau\ge 0$ and sufficiently small $\eta>0$ (independent of $\tau$), applying Strichartz estimate on $[\tau, \tau+\eta]$ and follow the similar procedure as in \eqref{JJ6}, \eqref{525}, we obtain
\begin{align*}
\|v\|_{ S^1{([\tau,\tau+\eta])}}\le &C\|v(\tau)\|_{H^1}+C\eta^{1-\frac 1{q^*}}\|v\|_{ S^1{([\tau,\tau+\eta])}}\\
&+C\eta^{1-\frac 2{q^*}}\|v\|_{ S^1{([\tau,\tau+\eta])}}^2+C\eta^{1-\frac{p+1}{q^*}}\|v\|_{ S^1{([\tau,\tau+\eta])}}^{p+1}.
\end{align*}
 Recall $\|v(\tau)\|_{H^1}\lsm e^{-\lambda \tau}$. Choosing sufficiently small $\eta$ and using a standard argument gives 
\begin{align*}
\|v\|_{S^1([\tau,\tau+\eta])}\le 2C\|v(\tau)\|_{H^1}\lsm e^{-\lambda \tau}. 
\end{align*}
The estimate of $\|v\|_{ S^1([t,\infty))}$ then follows by partitioning the interval $[t,\infty)$ into subintervals of length $\eta$ and summing all the local estimates.

Finally, the quadratic estimate \eqref{438} follows from the existence and uniqueness of solutions in the smaller ball $B_{C |y_0^-|, e_0}$ and inequality \eqref{1009} with $y_2^+ =0$. The time translation invariance follows from the uniqueness and decay of solutions. The proof of Theorem \ref{exiuni} is complete. 
\end{proof}

As an immediate consequence of Theorem \ref{exiuni}, we have

\begin{corollary}\label{cor:411}
Under the same conditions as in Theorem \ref{exiuni}, there exist exactly two solutions (up to time translation and phase rotation) $e^{it}Q_{\pm} (t)$ to NLS$_a$, satisfying 
\begin{align*}
\begin{cases}
\|Q_{\pm}-Q\|_{H^1}\le C e^{-e_0t}, \ \forall t\ge 0. \\
\|Q_+(0)\|_{H_a^1}>\|Q\|_{H^1_a},\ \ \|Q_-(0)\|_{H_a^1}<\|Q\|_{H^1_a}.
\end{cases}
\end{align*}
Moreover, any solution $u(t,x) \ne e^{it}Q$ with 
\begin{align*}
\|u(t)-e^{it}Q\|_{H^1}\le C e^{-\lambda t},\ \forall t\ge 0
\end{align*}
for some $C>0$ and $\lambda\in (0, e_0]$ uniquely coincides with either
\begin{align*}
\begin{cases}
u(t)=e^{i(t+T^+)}Q_+(t+T^+) & \textit{ if } \|u\|_{H_a^1}>\|Q\|_{H_a^1},\\
u(t)=e^{i(t+T^-)}Q_-(t+T^-) & \textit{ if } \|u\|_{H_a^1}<\|Q\|_{H_a^1},
\end{cases}
\end{align*}
for some $T^{\pm}\in \mathbb R$. 
\end{corollary}

We remark that the behavior of $Q_{\pm}$ for $t<0$ requires separate analysis which will be addressed in later sections.

\section{Rewrite (NLS$_a$) near the standing wave solution}\label{S:wave}

Let $Q$ be the ground state studied in the previous sections. We examine functions in the small neighborhood (in the invariant energy-mass level surface) of the ground state orbit $\{e^{i\theta}Q, \ \theta\in \mathbb S^1\}$ to obtain modulation decomposition for solutions in this regime. 

We start with the following preliminary result: 

\begin{lemma}[ $\nha$ Linear profile decomposition]\label{lm:lp}

Let $\{f_n\}$ be a bounded sequence in $\nha(\mathbb R^d)$.  There exists  a subsequence of $\{f_n\}$ which we still denote by $\{f_n\}$, $J^*\in \{0, 1, 2,\cdots\}\cup{\infty}$, $\{\phi^j\}_{j=1}^{J^*}\subset \nha(\R^d)$, $\{x_n^j\}_{j=1}^{J^*}\subset \Rd$ such that for every $0\le J\le J^*$, we have the decomposition 
\begin{align*}
f_n=\sum_{j=1}^{J} \phi^j(x-x_n^j)+r_n^J, \ r_n^J\in \nha(\Rd), 
\end{align*}
satisfying 
\begin{gather*}
x_n^j\equiv0, \; \textit{ or } \; x_n^j\to \infty,\ \ \lim_{n\to \infty}|x_n^j-x_n^k|=\infty, \ \textit{ for any } j\neq k, \\
\lim_{J\to J^*}\limsup_{n\to \infty}\|r_n^J\|_{{2^*}}=0,\\
\lim_{n\to \infty}\bigl(\|f_n\|_{L_x^2}^2-\sum_{j=1}^J \|\phi^j\|_{L_x^2}^2-\|r_n^J\|_{L_x^2}^2\bigr)=0,\\
\lim_{n\to \infty}\bigl(\|f_n\|_{L_x^{p+2}}^{p+2}-\sum_{j=1}^J \|\phi^j\|_{L_x^{p+2}}^{p+2}-\|r_n^J\|_{L_x^{p+2}}^{p+2}\bigr)=0,\\
\lim_{n\to \infty}\bigl(\|f_n\|_{\ha}^2-\sum_{j=1}^J \|\phi^j\|_{X_j}^2-\|r_n^J\|_{\ha}^2\bigr)=0. \\
\end{gather*}
Here, $X_j=\dot H^1$ if $x_n^j\to \infty$ and $X_j=\ha$ if $x_n^j\equiv0$. 
\end{lemma}

The proof of this lemma is a minor modification of Lemma 9.2 in \cite{YZZ22}, we skip the details. 

Define
\begin{equation}\label{df}
\mathbf{d}(f)=\bigl |\Vert f\Vert _{\dot{H}_{a}^{1}}^{2}-\Vert Q\Vert _{\dot{%
H}_{a}^{1}}^{2}\bigr|,
\end{equation}%
our first result shows $\mathbf d(f)$ measures the distance to the ground state orbit on the mass-energy surface of the ground state. 

\begin{lemma}\label{lm: vc} Let $f\in H_{a}^{1}(\mathbb{R}^{d})$ satisfy 
\begin{align*}
 M(f)=M(Q), \quad E(f)=E(Q). 
 \end{align*}
  Then for any $\varepsilon >0$, there exists $\delta >0$ such
that if $\mathbf{d}(f)<\delta $, 
\begin{equation*}
\ \inf_{\theta \in \mathbb{S}^{1}}\Vert f-e^{i\theta }Q\Vert
_{H_{a}^{1}}<\varepsilon .
\end{equation*}
\end{lemma}

\begin{proof}
We argue by contradiction. Suppose the statement fails; there must 
exist $\varepsilon _{0}>0$ and a sequence $\{f_{n}\}\subset H_{a}^{1}(%
\mathbb{R}^{d})$ satisfying 
\begin{equation}\label{311}
M(f_{n})=M(Q),\ \ E(f_{n})=E(Q),\textit{ and  }\mathbf{d}(f_{n})\rightarrow 0,
\end{equation}
yet for which
\begin{equation}\label{b531}
\inf_{\theta \in \mathbb{S}^{1}}\Vert f_n-e^{i\theta }Q\Vert
_{H_{a}^{1}}>\varepsilon _{0}
\end{equation}%
holds for all $n$. Clearly \eqref{311} implies the convergence 
\begin{align}\label{532}
\|f_n\|_{\ha}\to \|Q\|_{\ha}, \textit{ and } \|f_n\|_{L_x^{p+2}}^{p+2}\to \|Q\|_{L_x^{p+2}}^{p+2}. 
\end{align}
Applying Lemma \ref{lm:lp} to $\{f_n\}$, we obtain the decomposition 
\begin{align*}
f_n=\sum_{j=1}^J\phi^j(x-x_n^j)+r_n^J
\end{align*}
for each $0 \le J \le J^*$ with the stated properties. In particular, the norm decoupling together with \eqref{311}, \eqref{532} implies that 
\begin{gather*}
\|Q\|_{L_x^2}^2=\sum_{j=1}^J \|\phi^j\|_{L_x^2}^2+\lim_{n\to \infty}\|r_n^J\|_{L_x^2}^2, \\
\|Q\|_{L_x^{p+2}}^{p+2}=\sum_{j=1}^J\|\phi^j\|_{L_x^{p+2}}^{p+2}+\lim_{n\to \infty}\|r_n^J\|_{L_x^{p+2}}^{p+2},\\
\|Q\|_{\ha}^2=\sum_{j=1}^J\|\phi^j\|_{X_j}^2+\lim_{n\to \infty} \|r_n^J\|_{\ha}^2. 
\end{gather*}
Taking the limit as $J\to J^*$, we find  
\begin{align}\label{557}
\sum_{j=1}^{J^*} \|\phi^j\|_{L_x^2}^2\le \|Q\|_{L_x^2}^2, \quad \sum_{j=1}^{J^*}\|\phi^j\|_{X_j}^2\le \|Q\|_{\ha}^2, 
%\end{align}
%and
%\begin{align}\label{558}
\quad \|Q\|_{L_x^{p+2}}^{p+2}=\sum_{j=1}^{J^*}\|\phi^j\|_{L_x^{p+2}}^{p+2}. 
\end{align}
%\eqref{558} 
The last equality together with the sharp Gagliardo-Nirenberg inequality yields:
\begin{align}\label{906}
\|Q\|_{L_x^{p+2}}^{p+2}\le C_{GN}\sum_{j=1}^{J^*}\|\phi^j\|_{2}^{\frac{4-p(d-2)}2}\|\phi^j\|_{\ha}^{\frac{pd}2},
\end{align}
where $C_{GN}$ is the sharp constant: $C_{GN}=\tfrac{\|Q\|_{L_x^{p+2}}^{p+2}}{\|Q\|_{L_x^2}^{\frac{4-p(d-2)}2}\|Q\|_{\ha}^{\frac{pd}2}}$. Inserting this into \eqref{906}, we obtain 
\begin{align}\label{907}
\|Q\|_{L_x^2}^{\frac{4-p(d-2)}2}\|Q\|_{\ha}^{\frac{pd}2}\le \sum_{j=1}^{J^*}\|\phi^j\|_{L_x^2}^{\frac{4-d(p-2)}2}\|\phi^j\|_{\ha}^{\frac{pd}2},
\end{align}
which together with \eqref{557} gives
\begin{align}\label{1153}
\bigl(\sum_{j=1}^{J^*}\|\phi^j\|_{L_x^2}^2\bigr)^{\frac{4-p(d-2)}4}\bigl(\sum_{j=1}^{J^*}\|\phi^j\|_{X_j}^2\bigr)^{\frac{pd}4}
\le \sum_{j=1}^{J^*}\|\phi^j\|_{L_x^2}^{\frac{4-p(d-2)}2}\|\phi^j\|_{\ha}^{\frac{pd}2}. 
\end{align} 
As $p>\frac 4d$, $a<0$, $\|\phi^j\|_{\ha}\le \|\phi^j\|_{X_j}$, and $\|\phi^j\|_{\ha}<\|\phi^j\|_{\dot H^1}$, 
we conclude \eqref{557} holds if and only if $J^*=1$ and $x_n^1\equiv0$ hence (in a simplified notation):
 \begin{align*}
 f_n=\phi+r_n.
 \end{align*}
 Reapplying \eqref{907} and \eqref{557}, 
 %and \eqref{558}, 
 we obtain 
 \begin{align*}
 \|\phi\|_{L_x^2}=\|Q\|_{L_x^2}, \ \|\phi\|_{\ha}=\|Q\|_{\ha}, \ \|\phi\|_{L_x^{p+2}}=\|Q\|_{L_x^{p+2}},
 \end{align*}
 and $\lim_{n\to \infty}\|r_n\|_{\nha}=0$. The variational characterization of the ground state in Theorem \ref{thm: ground state} then implies there exists $\theta\in \mathbb S^1$ such that $\phi=e^{i\theta}Q$, and consequently 
 \begin{align*}
 \lim_{n\to \infty}\|f_n-e^{i\theta}Q\|_{\nha}=0.
 \end{align*}
We get a contradiction to \eqref{b531} and thus the lemma is proved. 
\end{proof}

\begin{lemma}\label{lm: dec} There exist $\varepsilon_0$ and $\delta _{0}$ such that for any $f\in H_a^1(\mathbb R^d)$ satisfying
\begin{align}\label{condf}
M(f)=M(Q), \ E(f)=E(Q), \ \textit{ and } \mathbf{d}(f)<\delta_0, 
\end{align}
there exists a unique $\theta\in \mathbb S^1$ depending on $f$ smoothly such that  
\begin{align}\label{ortho}
e^{-i\theta}f\perp_{\la+1}iQ,
%\textit{ and } \|f-e^{i\theta}Q\|_{\nha}<\eps_0. 
\end{align}
and 
\begin{align}\label{bound}
\|e^{-i\theta} f-Q\|_{\nha}\lesssim \eps_0. 
\end{align}
Moreover, the function admits the decomposition 
\begin{align}\label{decom}
e^{-i\theta}f=(1+\alpha)Q+\tilde u,
\end{align}
where $\alpha\in \mathbb R$ depending on $f$ smoothly and $\tilde u\perp_{\la+1}\{Q, iQ\}$ satisfy the bounds: 
\begin{align}\label{EQ0}
|\alpha|\sim \|\tilde u\|_{\nha}\sim \|e^{-i\theta}f-Q\|_{\nha}\sim \mathbf{d}(f).   
\end{align}
\end{lemma}

\begin{proof} We begin by considering functions near $Q$. 
Define functional $\mathcal F: \mathbb S^1\times \nha(\R^d)\to \R$ by  
\begin{align*}
\mathcal F(\theta, h)=\langle(\la+1)(e^{-i\theta}h), iQ\rangle.
\end{align*}
At $(\theta, h)=(0,Q)$, we have
\begin{align*}
\mathcal F(0, Q)=0, \ \frac{\partial \mathcal F}{\partial \theta}\biggr|_{(0, Q)}=-\|Q\|_{\nha}^2.
\end{align*}
By the Implicit Function Theorem, there exist $\eps_0, \eta_0>0$ and a $C^1$ map
\begin{align*}
\gamma: \nha(\R^d)\supset B_{{\eps_0}}(Q)\to (-\eta_0, \eta_0), \ ( C^1\mbox{ smoothness implies } \eta_0\le C \eps_0 )
\end{align*}
such that for any $f\in B_{{\eps_0}}(Q)$, there exists a unique phase $\theta_0=\gamma(f)\in (-\eta_0,\eta_0)$ satisfying 
$
\mathcal F(\theta_0, f)=0. 
%\ \| e^{-i\theta_0}f-Q\|_{\nha}\le \|f-Q\|_{\nha}+\|Q-e^{i\theta_0}Q\|_{\nha}\le \frac{\eps_0}2+C\eta_0<\eps_0.
$
This yields \eqref{ortho} with $\theta=\theta_0$. The decomposition \eqref{decom} and the orthogonality of $\tilde u$ hold with 
\begin{align}\label{339}
 \alpha=\frac{\langle(\la+1)(e^{-i\theta}f-Q),Q\rangle }{\|Q\|_{\nha}^2}.
\end{align}
The bound \eqref{bound} follows from:
\[
\ \| e^{-i\theta_0}f-Q\|_{\nha}\le \|f-Q\|_{\nha}+\|Q-e^{i\theta_0}Q\|_{\nha}\le {\eps_0}+C\eta_0\lesssim \eps_0.
\]

Given $\varepsilon_0$, from Lemma \ref{lm: vc}, there exists $\delta_0$ such that for general $f$ satisfying \eqref{condf}, there exists $\theta_1\in \mathbb S^1$ such that $e^{-i\theta_1}f\in B_{{\eps_0}}(Q)$. Repeating the same argument to $e^{-i\theta_1} f$ yields \eqref{ortho}- \eqref{decom} with $\theta=\theta_0+\theta_1$.

To prove the uniqueness of $\theta$ in $\mathbb S^1$, we suppose the uniqueness of $\theta$ fails for any choice of $\eps_0$ in \eqref{bound}. Then there is a sequence $\{f_n\}\subset \nha(\R^d)$ and two phase sequences $\{\theta_n\}, \{\tilde \theta_n\}\subset \mathbb S^1$ such that 
\begin{align*}
e^{-i\theta_n}f_n\perp_{\la+1}iQ,\ e^{-i\tilde\theta_n}f_n\perp_{\la+1} iQ
\end{align*}
and 
\begin{align*}
\|e^{-i\theta_n}f_n-Q\|_{\nha}\le \frac 1n, \ \ \|e^{-i\tilde \theta_n}f_n-Q\|_{\nha}\le \frac 1n. 
\end{align*}
From the triangle inequality, this implies 
\begin{align*}
\|e^{i(\theta_n-\tilde \theta_n)}Q-Q\|_{\nha}\le \frac 2n, 
\end{align*}
which implies $\theta_n-\tilde \theta_n\to 0$ as $n\to \infty$. For sufficiently large $n$, this clearly contradicts the uniqueness of the $\theta$-parameter for $e^{-i \theta_n} f_n$ close to $Q$.
% on the small interval $(-\eta_0, \eta_0)$. 

Finally, we prove the estimates in \eqref{EQ0}. Define the functional 
\begin{align*}
I(f)=M(f)+2E(f)=\langle(\la+1)f, f\rangle-\frac 2{p+2}\int_{\R^d} |f|^{p+2}dx.
\end{align*}
We compute the derivatives of $I$: 
\begin{gather}
I'(f) g=2\langle(\la+1)-|f|^p) f, g\rangle,\notag\\
\frac 12\langle I''(f)h, h\rangle =\langle(\la+1)h, h\rangle-p\int_{\R^d}|f|^{p-2}(f\cdot h)^2dx -\int_{\R^d}|f|^p|h|^2 dx. \label{1134}
\end{gather}
Here $f$ and $h$ are viewed as 2D real vectors in $f\cdot h$.
% is naturally defined when viewing complex functions. 
Clearly, for $f=Q$, we have 
\begin{align}\label{1133}
I'(Q)=0, \textit{ and } I''(Q)=2L =2\begin{pmatrix}L_1&0\\0&L_2\end{pmatrix}. 
\end{align}
Now taking $f$ satisfying \eqref{condf}, there is $\theta$ such that
\begin{align*}
e^{-i\theta}f=Q+\alpha Q+\tilde u:=Q+v,
\end{align*}
with the stated orthogonality. From Taylor's expansion, we have 
\begin{align}
I(f)&=I(e^{-i\theta} f)=I(Q+v)=I(Q)+\langle L v, v\rangle 
\label{451}\\&\qquad 
+\langle \int_0^1 (1-t)\bigl (I''(Q+tv)-I''(Q)\bigr) dt \, v, v\rangle. \label{452}
\end{align}
%In \eqref{451}, 
Due to the orthogonality and the fact that $(\la+ 1)Q = Q^{p+1}$, we have
\begin{align}\label{453}
\langle Lv, v\rangle=\alpha^2\langle LQ, Q\rangle +\langle L\tilde  u, \tilde u\rangle. 
\end{align}
 Assuming momentarily that \eqref{452} has the estimate 
 \begin{align}
 \eqref{452}\le C(\|v\|_{\nha}^{p+2}+\|v\|_{\nha}^3)\le C(\alpha^{p+2}+\|\tilde u\|_{\nha}^{p+2}+\alpha^3+\|\tilde u\|_{\nha}^3), \label{454}
 \end{align}
 we proceed to prove \eqref{EQ0} in the next few lines. 
 
 Indeed, inserting \eqref{452} and \eqref{453} into \eqref{451} and noting the fact $I(f)=I(Q)$, we obtain  
 \begin{align*}
 -\alpha^2 \langle LQ, Q\rangle =\langle L\tilde u, \tilde u\rangle+o(\alpha^2+\|\tilde u\|_{\nha}^2),
 \end{align*}
which together with the coercivity of $L$ on $\spa\{Q, iQ\}^{\perp}$ from Lemma \ref{lm: l2} gives 
\begin{align*}
\alpha\sim \|\tilde u\|_{\nha}.
\end{align*}
Moreover, from $M(f)=M(Q)$ and the definition of $\bd (f)$ in \eqref{df}, we have
\begin{align*}
\mathbf{d}(f)&=\mathbf{d}(e^{-i\theta} f)=\biggl |\bigl\|(1+\alpha)Q+\tilde u\bigr\|_{\nha}^2-\|Q\|_{\nha}^2\biggr|\\
&=\bigl |(1+\alpha)^2 \|Q\|_{\nha}^2 -\|Q\|_{\nha}^2+\|\tilde u\|_{\nha}^2\bigr|=\bigl|(2\alpha+\alpha^2)\|Q\|_{\nha}^2+
\|\tilde u\|_{\nha}^2\bigr|,
\end{align*}
which clearly establishes \eqref{EQ0}. 

To complete the proof, we must verify the estimate \eqref{454}. The second half of this estimate holds trivially. For the first half, it suffices to bound 
\begin{align}
& \langle (I''(Q+tv)-I''(Q))v, v\rangle\notag\\
=&-2p \int_{\R^d} |Q+tv|^{p-2}\bigl((Q+tv)\cdot v\bigr)^2-Q^{p-2}\bigl(Q\cdot v\bigr)^2 dx\label{343}\\
& -2\int_{\R^d}\bigl(|Q+tv|^p-Q^p)|v|^2 dx, \label{344}
\end{align}
uniformly for $t\in [0,1]$. The estimate of \eqref{344} follows from H\"older inequality, the embedding $\nha(\R^d)\subset L^{p+2}(\R^d)$ and the following point-wise estimate 
\begin{align}\label{345}
\bigl | |Q+tv|^p-Q^p\bigr|\le 
\begin{cases}
|v|^p,& 0<p\le 1;\\
C_p(|Q|^{p-1}+|v|^{p-1})|v|, & p>1. 
\end{cases}
\end{align}
Despite its complicated form, \eqref{343} can be estimated similarly. Lemma \ref{lm: dec} is finally proved. 
\end{proof}

Let $u(t)$ be a solution to (NLS$_a$) on a time interval $I$ satisfying 
\begin{equation}\label{1000}
M(u)=M(Q),\; E(u)=E(Q),\; \mathbf{d}(u(t))<\delta _{0},\ \forall t\in I.
\end{equation}
Applying Lemma \ref{lm: dec} to $e^{-it}u(t)$, we obtain continuous local coordinates 
\begin{align}
\theta(t)
%\in \mathbb S^1
, \, \alpha(t)\in \R, \ \tilde u(t)\in \nha (\R^d), \ \tilde u(t)\perp_{\la+1}\{iQ, Q\},\label{1001}
\end{align}
such that
\begin{align}\label{1002}
u(t)=e^{i(\theta(t)+t)}\bigl[(1+\alpha(t))Q+\tilde u(t)\bigr]:=e^{i(\theta(t)+t)}[Q+v(t)],
\end{align}
where
\begin{align}\label{vtu}
v(t)=\alpha(t) Q+\tilde u(t). 
\end{align}
Beyond the comparability guaranteed by Lemma \ref{lm: dec}, these local quantities also satisfy the following derivative estimates.

\begin{lemma}
\label{L: Modulation} Let $\frac 4d<p<\frac 4{d-2}$, $0>a>\frac{(d-2)^2}4\bigl(-1+\frac{p^2}{(p+1)^2}\bigr)$ and $u(t,x)$ be a solution of (NLS$_a$) on a time interval $I$ satisfying \eqref{1000}--\eqref{1002}. Then the local coordinates in \eqref{1001} satisfy 
\begin{equation}
|\alpha (t)|\sim \left\Vert v(t)\right\Vert _{H_{a}^{1}}\sim \left\Vert 
\tilde{u}(t)\right\Vert _{H_{a}^{1}}\sim \mathbf{d}(u(t)),
\label{modulation 1}
\end{equation}%
\begin{equation}
|\alpha ^{\prime }(t)|+|\theta ^{\prime }(t)|\lesssim \mathbf{d}(u(t)),
\label{modulation 2}
\end{equation}%
with implicit constants independent of $t$. 
\end{lemma}

\begin{proof}
Estimate \eqref{modulation 1} follows directly from \eqref{EQ0} in Lemma \ref{lm: dec}. To prove \eqref{modulation 2}, we first derive equations for $v=v_1+iv_2$ and $\tilde u=\tilde u_1+i\tilde u_2$. 

Inserting \eqref{1002} into (NLS$_a$) yields:
\begin{align}\label{1109} 
i\partial_t v-\la v+|Q+v|^p(Q+v)-(1+\theta'(t))(Q+v)-\la Q=0. 
\end{align}
Using $\la Q=-Q+Q^{p+1}$ and the fact that the linear terms (in $v$) in the expansion of $|Q+v|^p(Q+v)$ are given by 
\[
(p+1)Q^pv_1+iQ^p v_2,
\]
 \eqref{1109} becomes: 
\begin{align}\label{1162}
\partial_t v+i L_1 v_1-L_2 v_2+i\theta'(t)Q=R(v)-i\theta'(t)v, 
\end{align}
where $L_1, L_2$ and $R(v)$ are defined by \eqref{l1l2} and \eqref{r} in the previous section. 
The equation of $\tilde u$ follows from this and the expression \eqref{vtu}:
\begin{align}\label{1124}
\partial_t \tilde u+iL_1\tilde u_1-L_2\tilde u_2+\alpha'(t)Q - i  (p+1)\alpha(t)Q^{p+1}+i\theta'(t)Q=R(v)-i\theta'(t) v. 
\end{align}
Pairing \eqref{1124} against $(\la+1)Q$, $(\la+1)iQ$, using the orthogonality and the fact that $(\la+1)Q=Q^{p+1}$ gives 
\begin{gather}
\alpha'(t)\|Q\|_{\nha}^2=\langle L_2 \tilde u_2, Q^{p+1}\rangle +\langle R(v), Q^{p+1}\rangle-\langle i\theta'(t) v, Q^{p+1}\rangle,\label{1139}\\
\theta'(t)\|Q\|_{\nha}^2 =-\langle L_1\tilde u_1, Q^{p+1}\rangle +  (p+1) \alpha(t)\langle Q^{p+1},Q^{p+1}\rangle+\langle R(v)-i\theta'(t)v,  i Q^{p+1}\rangle. \label{1148}
\end{gather}
Using the following set of estimates:
\begin{align}\label{1159}
\begin{cases}
\bigl|\langle(\la+1)\tilde u, Q^{p+1}\rangle\bigr|+\bigl| \langle Q^p\tilde u, Q^{p+1}\rangle\bigr|+\bigl|\langle v, Q^{p+1}\rangle\bigr|\lesssim \mathbf{d}(u(t)), \quad  Q \in L^{2p+2}(\R^d), \\
%\langle Q^{p+1},Q^{p+1}\rangle \lesssim 1,\\
\bigl|\langle R(v), Q^{p+1}\rangle\bigr|\lesssim \mathbf{d}(u(t))^2+\mathbf{d}(u(t))^{p+1},
\end{cases}
\end{align}
we can control the right-hand sides of \eqref{1139} and \eqref{1148}, yielding
\begin{align*}
\bigl(|\alpha'(t)|+|\theta'(t)|\bigr)\|Q\|_{\nha}^2\lesssim \mathbf{d}(u(t))(1+|\theta'(t)|),
\end{align*}
from which \eqref{modulation 2} follows directly. 

It remains to prove \eqref{1159}. From H\"older inequality and \eqref{modulation 1}, the first row holds if  
%amounts to estimating 
\begin{align}\label{qnorm}
\qquad\|Q^{p+1}\|_{H^1_a}+\|Q^{2p+1}\|_{L_x^{\frac{2d}{d+2}}}+\|Q^{2p+2}\|_{L_x^1} <\infty.
\end{align}
These terms can be controlled using norms in \eqref{qspace} (recall  $r=\tfrac{2d(p+1)}{d-2},\ r^*=\tfrac{2d(p+1)}{2p+d} $ in \eqref{qqstar}):
\begin{align*}
\eqref{qnorm}\le \|\nabla Q\|_{L_x^{r^*}}\|Q\|_{L_x^r}^p+\|Q\|_{L_x^{\frac{2d(2p+1)}{d+2}}}^{2p+1}+\|Q\|_{L_x^{2p+2}}^{2p+2},
\end{align*}
which is bounded from \eqref{qspace} since $\frac{2d(2p+1)}{d+2}<r$, $2p+2<r$ under the constraint for $p$. 

For the last row of \eqref{1159}, we apply the pointwise estimate for $R(v)$ from Lemma \ref{lm:rv}
and H\"older inequality to obtain 
\begin{align*}
|\langle R(v), Q^{p+1}\rangle|\lesssim \|v\|_{L_x^{2^*}}^{p+1}\|Q\|^{p+1}_{\frac{2^*(p+1)}{2^*-p-1}}\lesssim \textbf{d}(u(t))^{p+1}, 
\end{align*}
for $0<p\le 1$, and:
\begin{align*}
|\langle R(v), Q^{p+1}\rangle|\lesssim \|v\|_{L_x^{2^*}}^{p+1}\|Q\|^{p+1}_{\frac{2^*(p+1)}{2^*-p-1}}+\|v\|_{L_x^{2^*}}^2\|Q\|_{L_x^{pd}}^{2p}\lesssim \textbf{d}(u(t))^{p+1} +\textbf{d}(u(t))^{2} 
\end{align*}
for $p>1$. 
Here the integrability of $Q$ follows again from \eqref{qspace} and the fact that $\frac{2^*(p+1)}{2^*-p-1}<r$ and $pd\le r$. This establishes \eqref{1159} and completes the proof of \eqref{modulation 2}. 
\end{proof}

\section{Dynamics of solutions on the mass-energy surface of $Q$}\label{S:Q}
In this section, we classify the dynamical behavior of solutions on the mass-energy level surface of the ground state. This classification is achieved by combining the local dynamic analysis from the previous section with a virial analysis.  

\subsection{Virial estimates}\label{virial}

Let $\phi (x)$ be a smooth radial function such that 
\begin{equation*}
\phi (x)=%
\begin{cases}
|x|^{2}, & |x|\leq 1; \\ 
0, & |x|>2,%
\end{cases}%
\;\mathit{and}\text{ }\phi''(r)\leq 2.
\end{equation*}%
For $R>0$, define the rescaled function $\phi _{R}(x)=R^{2}\phi (\frac{x}{R})$. Let $u(t,x)$ be a solution of (NLS$_{a}$) satisfying $M(u)=M(Q)$, $E(u)=E(Q)$. We define
the truncated virial %
\begin{equation*}
V_{R}(t)=\int_{\mathbb{R}^{d}}\phi _{R}(x)|u(t,x)|^{2}dx.
\end{equation*}%
A direct computation yields:
\begin{align*}
\partial _{t}V_{R}(t)& =2\mathbf{Im}\int_{\mathbb{R}^{d}}\bar u(t)%
\,\nabla u(t)\cdot \nabla \phi _{R}dx,
\end{align*}
and the second derivative is given by
\begin{align*}
\partial _{tt}V_{R}(t)& =4\mathbf{Re}\int_{\mathbb{R}^{d}}(\phi
_{R})_{jk}(x)u_{j}(t)\bar{u}_{k}(t)\,dx-\frac{2p}{p+2}\int_{\mathbb{R}%
^{d}}(\Delta \phi _{R})|u(t)|^{p+2}\,dx \\
& \quad -\int_{\mathbb{R}^{d}}(\Delta ^{2}\phi _{R})|u(t)|^{2}\,dx+\int_{%
\mathbb{R}^{d}}\frac{4ax\cdot \nabla \phi _{R}}{|x|^{4}}|u(t)|^{2}\,dx \\
& =8\Vert u(t)\Vert _{\dot{H}_{a}^{1}}^{2}-\frac{4pd}{p+2}\Vert u(t)\Vert
_{p+2}^{p+2}+A_{R}(u(t)) \\
& =(2pd-8)(\Vert Q\Vert _{\dot{H}_{a}^{1}}^{2}-\Vert u(t)\Vert _{\dot{H}%
_{a}^{1}}^{2})+A_{R}(u(t)). 
\end{align*}%
We express this in terms of distance function $\mathbf{d}(u(t))$ as
\begin{equation}
\partial _{tt}V_{R}(t)=%
\begin{cases}
(2pd-8)\mathbf{d}(u(t))+A_{R}(u(t)), & \mathit{if}\Vert u(t)\Vert _{\dot{H}%
_{a}^{1}}<\Vert Q\Vert _{\dot{H}_{a}^{1}},\\ 
(8-2pd)\mathbf{d}(u(t))+A_{R}(u(t)), & \mathit{if}\Vert u(t)\Vert _{\dot{H}%
_{a}^{1}}>\Vert Q\Vert _{\dot{H}_{a}^{1}},%
\end{cases}%
  \label{Vtt}
\end{equation}
where $A_R(u(t))$ is an error term given by 
\begin{eqnarray}
A_{R}(u(t)) &=&\int_{|x|>R}(\frac{4\partial _{r}\phi _{R}}{r}-8)|\nabla
u(t)|^{2}+\frac{2p}{p+2}(-\Delta \phi _{R}+2d)|u(t)|^{p+2}dx\label{458} \\
&&-\int_{|x|>R}\Delta ^{2}\phi _{R}|u(t)|^{2}+\frac{4a(x\cdot \nabla \phi
_{R}-2|x|^{2})}{|x|^{4}}|u(t)|^{2}dx\notag \\
&&+\int_{|x|>R}\frac{4(r\partial _{rr}\phi _{R}-\partial _{r}\phi _{R})}{%
r^{3}}|x\cdot \nabla u(t)|^{2}dx.\notag
\end{eqnarray}%
If  $u(t)$ is radial, the first and last integrals combine to yield
\begin{eqnarray}
A_{R}(u(t)) &=&\int_{|x|>R}(4\partial _{rr}\phi _{R}-8)|\nabla u(t)|^{2}+%
\frac{2p}{p+2}(-\Delta \phi _{R}+2d)|u(t)|^{p+2}dx \label{459}\\
&&-\int_{|x|>R}\Delta ^{2}\phi _{R}|u(t)|^{2}+\frac{4a(x\cdot \nabla \phi
_{R}-2|x|^{2})}{|x|^{4}}|u(t)|^{2}dx.\notag
\end{eqnarray}
Lastly we estimate $\partial_{t}V_{R} $ and $A_{R}$.

\begin{lemma}[Virial estimate]\label{LV} 
Let $\delta_0>0$ be the constant given in Lemma \ref{lm: dec}. Let $u(t,x)$ be a solution of (NLS$_{a}$) satisfying $M(u)=M(Q)$ and $%
E(u)=E(Q)$. Then the following  hold:
\begin{equation}
|\partial _{t}V_{R}(t)|\lesssim R^{2}\mathbf{d}(u(t)),  \label{VR}
\end{equation}%
\begin{equation}
|A_{R}(u(t))|\lesssim \left\{ 
\begin{array}{l}
\int_{|x|>R}\bigl(|\nabla u(t)|^{2}+|u(t)|^{p+2}+\frac{|u(t)|^{2}}{|x|^{2}}%
\bigr)dx, \\ 
e^{-R}\mathbf{d}(u(t))+\mathbf{d}(u(t))^{2},\ \ \mathit{\ }\text{ if } \; \mathbf{d}%
(u(t))<\delta _{0}\mathit{\ and}\text{ }|R|\gtrsim 1.
\end{array}%
\right.  \label{AR2}
\end{equation}%
%for some $c>0$ determined only by the ground state $Q$. 
If $u$ is radially symmetric, we also have 
\begin{equation}
A_{R}(u(t))\lesssim \int_{|x|>R}\bigr(|u(t)|^{p+2}+\frac{|u(t)|^{2}}{|x|^{2}}%
\bigr)dx.  \label{AR1}
\end{equation}%
\end{lemma}

\begin{proof}
Using H\"{o}lder inequality and Sobolev embedding, we estimate
\begin{equation*}
|\partial _{t}V_{R}(t)|\leq \Vert \nabla u\Vert _{L_x^2}\Vert u\Vert _{L_x^{2^{\ast}}}\Vert \nabla \phi _{R}\Vert _{L_x^d}\lesssim R^{2}\Vert \nabla \phi \Vert _{L_x^d}\Vert u\Vert _{\dot{H}_{a}^{1}}^{2}\lesssim R^{2}(\mathbf{d}(u(t))+\Vert
Q\Vert _{\dot{H}_{a}^{1}}^{2}).
\end{equation*}%
For $\mathbf{d}(u(t))\geq \delta _{0}$, this directly implies \eqref{VR}. 
For $\mathbf{d}(u(t))<\delta _{0}$, we use the decomposition $u(t)=e^{i(\theta(t)+t)}(Q+v(t))$ from \eqref{1002} and compute:
\begin{align*}
|\partial _{t}V_{R}(t)|& =\bigl|2\mathbf{Im}\int_{\mathbb{R}^{d}}\nabla \phi
_{R}\nabla (Q+v(t))(Q+\bar{v}(t))dx\bigr| \\
& =\bigl|2\mathbf{Im}\int_{\mathbb{R}^{d}}\nabla \phi _{R}(\nabla Q\bar{v}%
(t)+\nabla v(t)Q+\nabla v(t)\bar{v}(t))dx\bigr| \\
& \leq \Vert \nabla \phi _{R}\Vert _{L_x^d}(\Vert \nabla Q\Vert _{L_x^2}\Vert \nabla
v\Vert _{L_x^2}+\Vert \nabla v\Vert _{L_x^2}^{2}) \\
& \lesssim R^{2}\mathbf{d}(u(t)).
\end{align*}
This completes the proof of \eqref{VR}. 

We turn to estimating $A_{R}(u(t))$. The first bound in \eqref{AR2} follows directly from the bound of $\phi$ and its derivatives. For the refined estimate when $\mathbf{d}(u(t))<\delta_0$ and $R\gtrsim 1$, we use the decomposition \eqref{1002} and the fact that $A_R(Q)=A_R(e^{it}Q)=0$. 
This yields
\begin{align*}
|A_{R}(u(t))|& =|A_{R}(e^{i(\theta (t)+t)}(Q+v(t)))|=|A_{R}(Q+v(t))-A_{R}(Q)|
\\
& \lesssim \Vert \nabla Q\Vert _{L_x^{2}(|x|\geq R)}\Vert \nabla v(t)\Vert _{L_x^2}+\Vert \nabla v(t)\Vert _{L_x^2}^{2}+\Vert \frac{Q}{|x|}\Vert
_{L_x^{2}(|x|\geq R)}\Vert \nabla v(t)\Vert _{L_x^2} \\
& +\Vert v\Vert _{L_x^{2^{\ast }}}(\Vert Q\Vert _{L_x^{\frac{2d(p+1)}{d+2}}(|x|\geq
R)}^{p+1}+\Vert v(t)\Vert _{H_{a}^{1}}^{p+1}) \\
& \lesssim e^{{ -R}}\mathbf{d}(u(t))+\mathbf{d}(u(t))^{2},
\end{align*}%
where the exponential decay $e^{-R}$ arises from the asymptotic behavior of $Q$ (Lemma \ref{lm: infinity}).

In the radial case, the term involving $\nabla u$ in \eqref{459} can be dropped due to $\phi''(r)\le 2$, simplifying the estimate to \eqref{AR1}. Lemma \ref{LV} is proved. 
\end{proof}

\subsection{Classification of non-scattering solutions in the sub-critical case\label{decay_sub}}

In this subsection, we prove that the orbits of non-scattering solutions must be on the stable/unstable manifold of some ground state $e^{i\theta} Q$ when they lie on the mass-energy level surface of $Q$ with subcritical kinetic energy. Our main result is: 

\begin{theorem}
\label{thm: subcase} Let $e^{it}Q_-$ be the exponentially convergent solution constructed in Corollary \ref{cor:411}. Let $u$ be a solution of NLS$_{a}$ satisfying 
\begin{equation}
M(u)=M(Q),E(u)=E(Q),\Vert u_{0}\Vert _{\dot{H}_{a}^{1}}<\Vert Q\Vert _{\dot{H%
}_{a}^{1}},\Vert u\Vert _{S^1([0,\infty ))}=\infty .  \label{1122}
\end{equation}%
Then there exists a unique pair of parameters $(\theta, T)\in \mathbb S^1\times \R$ such
that:
\begin{equation}
u(t,x)=e^{i(\theta +t)}Q_{-}(t+T,x).  \label{357}
\end{equation}%
Moreover, in the negative time direction, $u$ exists globally with $\Vert u\Vert
_{S^1((-\infty ,0])}<\infty $.

Conversely, if $\|u\|_{S^1((-\infty, 0])}=\infty$, then $u$ must  scatter forward in time and takes the form:  
\begin{align*}
u(t,x)=e^{i(\theta+t)}\bar Q_-(-t+T, x),
\end{align*}
for some $\theta$ and $T$. 
\end{theorem}

The remainder of this subsection is devoted to proving Theorem \ref{thm: subcase}. We begin by establishing key properties of solutions satisfying \eqref{1122}.

\subsubsection{Compactness and its implications.}

Building on previous work in
%Killip-Murphy-Visan-Zheng and Lu-Miao-Murphy in 
\cite{KMVZ17} and \cite{LMM18}, we recall that $E(Q)$ represents the minimal energy threshold for non-scattering solutions on the mass level surface of $Q$. Consequently, solutions satisfying \eqref{1122} exhibit precompactness in $\nha(\R^d)$:
\begin{equation}
\{u(t)\}_{t\in \lbrack 0,\infty )}\mathit{\ is}\text{ }\textit{precompact in } H_{a}^{1}(\mathbb{R}^{d}). \label{precomp}
\end{equation}%
This compactness property implies the following spatial localization: for any $\eps>0$, 
there exists $C(\eps)>0$ such that 
\begin{equation}
\int_{|x|>C(\varepsilon )}|\nabla
u(t)|^{2}+|u(t)|^{p+2}+|u(t)|^{2}dx<\varepsilon .  \label{AA char}
\end{equation}
Combining \eqref{AA char} with virial estimates yields: 

\begin{lemma}
\label{L: sub tn} For any solution $u(t)$ of NLS$_{a}$ satisfying %
\eqref{1122}, the distance function $\mathbf{d}(u(t))$ has the following properties:

a) There exists a sequence $t_{n}\rightarrow +\infty $
with $\mathbf{d}(u(t_{n}))\rightarrow 0$.

b) For any interval $[a,b]\subset
\lbrack 0,\infty )$, we have:
\begin{equation}
\int_{a}^{b}\mathbf{d}(u(t))dt\leq C[\mathbf{d}(u(a))+\mathbf{d}(u(b))].
\label{1019}
\end{equation}
Consequently, there exist constants $C, c>0$ such that
\begin{equation}
\int_{t}^{\infty}\mathbf{d}(u(s))ds\leq Ce^{-ct}, \ t\ge 0.  \label{1113}
\end{equation}%

\end{lemma}

\begin{proof}
We first prove statement a). Let $\varepsilon >0$ and define $T_{0}(\varepsilon )=({C(\varepsilon ))}^{2}/{%
\varepsilon }$. Clearly,
\begin{equation*}
(\varepsilon t)^{\frac{1}{2}}\geq C(\varepsilon ),\ \forall t\ge T_{0}.
\end{equation*}%
For $T\ge T_0$, set $R=\left( \varepsilon T\right) ^{\frac{1}{2}}$ and consider the virial quantity $V_R(t)$. From Lemma \ref{LV}, we obtain
\begin{align*}
|\partial _{t}V_{R}(t)|& \le C R^{2}\mathbf{d}(u(t)) =C\varepsilon T, \\
|A_{R}(u(t))|& \le C \int_{|x|>R}\bigl(|\nabla u(t)|^{2}+|u(t)|^{p+2}+%
\frac{|u(t)|^{2}}{|x|^{2}}\bigr)dx \\
& \le C\int_{|x|>C(\varepsilon )}\bigl(|\nabla u(t)|^{2}+|u(t)|^{p+2}+%
\frac{|u(t)|^{2}}{|x|^{2}}\bigr)dx \\
& \leq C \varepsilon.
\end{align*}
Substituting into \eqref{Vtt} yields:
\begin{equation*}
\partial _{tt}V_{R}(t)\geq (2pd-8)\mathbf{d}(u(t))-C\varepsilon.
\end{equation*}%
Integrating over $[T_{0},T]$ gives:
\begin{equation*}
\frac{1}{T}\int_{T_{0}}^{T}\mathbf{d}(u(t))dt\le C \frac{\varepsilon
(T-T_{0})+\varepsilon T}{T}\lesssim \varepsilon.
\end{equation*}%
Taking the limit as $T\to \infty$ then $\eps\to 0$, we obtain
\begin{equation*}
\lim_{T\rightarrow \infty }\frac{1}{T}\int_{0}^{T}\mathbf{d}(u(t))dt=0, 
\end{equation*}%
establishing the existence of $t_n\to \infty$ with $\mathbf{d}(u(t_n))\to 0$.

We now establish the integral estimate b). From the virial estimate
 for $\partial_t V_R$ in \eqref{VR}, the desired inequality \eqref{1019} follows from applying the Fundamental Theorem of Calculus to the lower bound estimate: 
\begin{equation}
 \partial _{tt}V_{R_0}(t)\geq (pd-4)\mathbf{d}(u(t)),
\label{gap 3}
\end{equation}
for some $R_0>1$. 

To verify this lower bound, we observe from \eqref{Vtt} that it suffices to show:
\begin{align}\label{1010}
|A_{R_0}(u(t))|\leq
(pd-4)\mathbf{d}(u(t)).
\end{align}
This inequality follows immediately from the error estimate \eqref{AR2} in both cases combined with the compactness of $u(t)$, thus ending the proof of \eqref{1019}. 

Finally let $t_n$ be the sequence from part a), we apply \eqref{1019} to the interval $[t, t_n]$ and obtain
\begin{equation}
\int_{t}^{t_{n}}\mathbf{d}(u(s))ds\leq C[\mathbf{d}(u(t))+\mathbf{d}%
(u(t_{n}))],  \label{decay 0}
\end{equation}%
yielding
\begin{equation*}
\int_{t}^{\infty }\mathbf{d}(u(s))ds\leq C\mathbf{d}(u(t)),\ \forall t\geq 0
\end{equation*}%
by taking $t_n\to \infty$. The exponential decay estimate \eqref{1113} then follows by applying  Gr\"{o}nwall's inequality to this inequality. 
\end{proof}

We are ready to prove Theorem \ref{thm: subcase}.

\subsubsection{Proof of Theorem \protect\ref{thm: subcase}}

\begin{proof} Let $u$ be a solution satisfying \eqref{1122}. By Corollary \ref{cor:411}, it suffices to establish
\begin{align}\label{448}
\|u(t)-e^{i(\theta+t)}Q\|_{\nha}\le Ce^{-ct}, \, \forall t\ge 0,
\end{align}
for some $\theta\in \mathbb S^1$ and contants $c, C>0$. 

A crucial step is proving 
\begin{align}\label{141}
\lim_{t\to \infty}\mathbf{d}(u(t))=0.
\end{align} 
We proceed by a contradiction argument. Let $\{t_n\}$ be the sequence from Lemma \ref{L: sub tn} with $\mathbf{d}(u(t_n))\to 0$. If \eqref{141} fails, there exists a subsequence (still denoted $\{t_n\}$), 
$\delta _{1}\in (0,\delta _{0})$ and times $\tau_n\in(t_n, t_{n+1})$ such that:
\begin{equation*}
\mathbf{d}(u(\tau _{n}))=\delta _{1}.
%\mathit{\ and}\text{ }\mathbf{d}%
%(u(t))\leq \delta _{1},\ \forall t\in \lbrack t_{n},\tau _{n}].  \label{840}
\end{equation*}%
From the modulation estimates \eqref{modulation 1}, we have:
\[
\alpha(\tau _{n})\sim \delta_1>0,\ \text{ while } \alpha(t _{n})\to 0.
\]
However, integrating $\alpha'(t)$ over $[t_n,\tau_n]$ and applying \eqref{1113} yields:
\begin{equation}
|\alpha (t_{n})-\alpha (\tau _{n})|\leq \int_{t_{n}}^{\tau _{n}}|\alpha
^{\prime }(s)|ds\leq C\int_{t_{n}}^{\tau _{n}}\mathbf{d}(u(s))ds\leq
Ce^{-ct_{n}}\to 0, \label{sub 1}
\end{equation}%
leading to a contradiction. This establishes \eqref{141}. 

For sufficiently large $T_0$ and $t\ge T_0$, we define modulation parameters $\alpha(t), \theta(t)$ via the orthogonal decomposition \eqref{1001}. The convergence \eqref{141} implies $\alpha(t)\to 0$ by \eqref{modulation 1}. 

For any $[t,\tau]\subset[T_0, \infty)$, estimate \eqref{sub 1} generalizes to: 
\begin{equation*}
|\alpha (t)-\alpha (\tau )|\leq Ce^{-ct},\ \tau \geq t\geq T_{0}.
\end{equation*}%
Taking $\tau\to \infty$ gives the exponential decay: 
\begin{equation}
|\alpha (t)|\leq Ce^{-ct},\ \forall t\geq T_{0}.\label{251}
\end{equation}%
 Consequently, from \eqref{modulation 1} we obtain 
\begin{equation}
\Vert v(t)\Vert _{H_{a}^{1}}+\mathbf{d}(u(t))\leq Ce^{-ct},\forall t\geq
T_{0}.  \label{sub 2}
\end{equation}%

With \eqref{sub 2} we integrate the $\theta'(t)$ estimate from \eqref{modulation 2} over $[t_n, t_m]\subset[T_0,\infty)$ to obtain
\begin{equation}
|\theta (t_{n})-\theta (t_{m})|\leq \int_{t_{n}}^{t_{m}}|\theta ^{\prime
}(t)|dt\leq C\int_{t_{n}}^{t_{m}}\mathbf{d}(u(t))dt\leq Ce^{-ct_{n}}.
\label{sub 3}
\end{equation}%
Thus, there exists $\theta \in \mathbb{R}$ such that:
\begin{equation}
|\theta (t)-\theta |\leq Ce^{-ct},\ \forall t\geq T_{0}.  \label{sub 4}
\end{equation}
Combining \eqref{1002} with the decay estimates \eqref{251}-\eqref{sub 4} establishes \eqref{448}, from which \eqref{357} follows via Corollary \ref{cor:411}. 

We are left to show the scattering in the opposite time direction by contradiction. Assume 
\[
\Vert Q_{-}\Vert _{S^1((-\infty ,0])}=\infty, 
\]
the compactness of $\{Q_-(t), t\in \R\}$ extends the distance estimates to all times:
\begin{equation*}
\int_{a}^{b}\mathbf{d}(Q_{-}(t))dt\leq C\big(\mathbf{d}(Q_{-}(a))+\mathbf{d}%
(Q_{-}(b))\big)\quad\textit{ and }\lim_{|t|\to \infty}\mathbf{d}(Q_-(t))=0. 
\end{equation*}%
This implies $\int_{-\infty
}^{\infty }\mathbf{d}(Q_{-}(t))dt=0$, forcing $Q_{-}\equiv e^{i\theta
}Q$, which contradicts the subcritical kinetic energy assumption. 

The final statement follows simply from the time reversibility of (NLS$_a$). The proof of Theorem \ref{thm: subcase} is then completed. 
\end{proof}

\subsection{Exponential convergence in the super-critical case}

This subsection is devoted to characterizing solutions of NLS$_{a}$ on the mass-energy level surface of $Q$ if the kinetic energy exceeds that of $Q$. Unlike the subcritical case, such solutions do not necessarily exhibit the compactness in $\nha(\R^d)$. We therefore restrict our analysis to solutions with additional spatial decay.

\begin{theorem}
\label{thm: super case} Let $e^{it}Q_+$ be the exponentially convergent solution constructed in Corollary \ref{cor:411}. Let $u$ be a solution to (NLS$_{a}$) on $%
[0,+\infty )$ satisfying 
\begin{equation}
M(u)=M(Q),E(u)=E(Q),\Vert u\Vert _{\dot{H}_{a}^{1}}>\Vert Q\Vert _{\dot{H}%
_{a}^{1}}.  \label{549}
\end{equation}%
Assume further that either
\begin{equation*}
xu_{0}\in L^{2}(\mathbb{R}^{d})\text{ or }u_{0}\in H_{rad}^{1}(\mathbb{R}%
^{d}).
\end{equation*}%
Then there exist $\theta, T\in \mathbb R$ such that:
\[
u(t,x)=e^{i(\theta +t)}Q_{+}(t+T,x).
\] 
For solutions defined on $(-\infty, 0]$ with the same other conditions, we have
\[
u(t,x)=e^{i(\theta+t)}\bar Q_+(-t+T, x). 
\]
\end{theorem}

As in the previous section, the key step is establishing the exponential decay estimate:
\begin{align}\label{531}
\int_t^\infty \mathbf{d}(u(s))ds \le C e^{-ct}. 
\end{align}
We focus on proving this estimate for both finite variance and radial cases.

\subsubsection{The finite variance case}  For solutions with finite variance, we consider the virial quantity: 
\[
V(t)=\int_{\mathbb{R}^{d}}|x|^{2}|u(t,x)|^2dx,
\]
whose derivatives satisfy:
\begin{eqnarray*}
V^{\prime }(t) =4\;\mathbf{Im}\int_{\mathbb{R}^{d}}x\cdot \nabla u(t)%
\overline{u(t)}dx, \qquad
V^{\prime \prime }(t) =(8-2pd)\mathbf{d}(u(t)).
\end{eqnarray*}
Clearly $V(t)>0, V''(t)<0$, and due to the concavity,  
\[
V'(t)>0, \ \forall \ t\ge 0. 
\]

We first show that \eqref{531} can be derived from the inequality 
\begin{equation}
(V^{\prime }(t))^{2}\leq C\mathbf{d}^{2}(u(t))V(t)=CV^{\prime \prime
}(t)^{2}V(t). \label{V}
\end{equation}%
 Indeed, rewriting \eqref{V} as:
\begin{equation}\label{1020}
\frac{V^{\prime }(t)}{\sqrt{V(t)}}\leq -CV^{\prime \prime }(t),
\end{equation}%
we integrate over $[0,t]$ to get%
\begin{equation*}
\sqrt{V(t)}-\sqrt{V(0)}\leq -C(V^{\prime }(t)-V^{\prime }(0))\leq CV^{\prime
}(0),
\end{equation*}%
showing $V(t)$ is uniformly bounded. Thus \eqref{1020} gives
\begin{equation*}
V^{\prime }(t)\leq -CV^{\prime \prime }(t), 
\end{equation*}%
which by Gr\"onwall yields
\[
V^{\prime }(t)\leq Ce^{-ct}. 
\]
Finally, integrating $V''(t)$ proves \eqref{531}:
\begin{equation*}
 Ce^{-ct}\ge V^{\prime }(t)=-\int_{t}^{\infty }V^{\prime \prime
}(s)ds=(2pd-8)\int_{t}^{\infty }\mathbf{d}(u(s))ds.
\end{equation*}%

We establish \eqref{V} by discussing two cases.

If $\mathbf{d}(u(t))\geq \delta _{0}$, \eqref{531} follows directly from H\"{o}lder inequality:
\begin{eqnarray*}
|V^{\prime }(t)|^{2} &\leq &\Vert \nabla u(t)\Vert _{L_x^2}^{2}\Vert xu(t)\Vert
_{L_x^2}^{2}\lesssim (\mathbf{d}(u(t))+\Vert Q\Vert _{\dot{H}_{a}^{1}}^{2})V(t)
\\
&\lesssim &\mathbf{d}(u(t))V(t)\lesssim \mathbf{d}(u(t))^{2}V(t),
\end{eqnarray*}%
where the implicit constants 
%are allowed to 
depend on $\delta _{0}$. 
%which is acceptable. 
When $\mathbf{d}%
(u(t))<\delta _{0}$, we use the decomposition \eqref{1002} to estimate:
\begin{align*}
|V^{\prime }(t)|& =\bigl|4\mathbf{Im}\int_{\mathbb{R}^{d}}x\cdot \nabla
(Q+v(t))(Q+\bar{v}(t))dx\bigr| \\
& \leq 4\left\vert \int_{\mathbb{R}^{d}}x\cdot \nabla Q\bar{v}%
(t)dx\right\vert +4\int_{\mathbb{R}^{d}}\left\vert \frac{x\cdot \nabla v(t)}{%
|x|}|x|\bar{u}(t)\right\vert dx \\
& \leq 4\Vert x\cdot \nabla Q\Vert _{L_x^2}\Vert v\Vert _{L_x^2}+4\Vert \nabla
v(t)\Vert _{L_x^2}\Vert xu(t)\Vert _{L_x^2}\\
& \lesssim \Vert v(t)\Vert _{H_{a}^{1}}\mathbf{(}1+\sqrt{V(t)}\mathbf{)} \\
& \lesssim \mathbf{d}(u(t))\sqrt{V(t)},
\end{align*}%
where we used the fact that $V(t)\geq V(0)\gtrsim 1$ due to the positivity of $V'$. Squaring this estimate yields \eqref{V}.

\subsubsection{The radial case}

Let $V_R(t)$ be the truncated virial as defined in Section \ref{virial}  and let $A_R(t)$ be the error term in \eqref{Vtt}. In the radial case, we use the decay properties of radial functions to establish the following estimate for $A_R(t)$: there exists a constant $R_0(\delta_0, Q)>0$ such that 
%for all $R\ge R_0$, 
\begin{equation}
A_{R}(u(t))\leq (pd-4)\mathbf{d}(u(t)), \quad \forall R\ge R_0.  \label{441}
\end{equation}
 To prove \eqref{441}, we consider two cases. If
$\mathbf{d}(u(t))<\delta _{0}$, then \eqref{441} follows directly from the second row in (\ref{AR2}). If $\mathbf{d}(u(t))\geq \delta _{0}$, we use the radial Sobolev embedding in the form 
\begin{equation*}
|x|^{d-1}|u(t)|^{2}\lesssim \Vert u(t)\Vert _{L_x^2}\Vert \nabla u(t)\Vert _{L_x^2}
\end{equation*}%
to estimate 
\begin{eqnarray*}
\int_{|x|>R}|u(t)|^{p+2}dx &\leq &\int_{|x|>R}|u(t)|^{2}\left( \frac{\Vert
u(t)\Vert _{L_x^2}\Vert \nabla u(t)\Vert _{L_x^2}}{|x|^{d-1}}\right) ^{\frac{p}{2}%
}dx\leq CR^{-\frac{(d-1)p}2}\Vert u(t)\Vert _{\dot{H}_{a}^{1}}^{\frac{p}{2}} \\
&\leq &CR^{-\frac{(d-1)p}2}(\mathbf{d}(u(t))+\Vert Q\Vert _{\dot{H}_{a}^{1}}^{2})^{\frac{p%
}{4}}.
\end{eqnarray*}%
Combining this with (\ref{AR1}), we obtain
\begin{align*}
A_{R}(u(t))& \lesssim \int_{|x|>R}|u(t)|^{p+2}dx+\int_{|x|>R}\frac{|u(t)|^{2}%
}{|x|^{2}}dx \\
& \lesssim R^{-\frac{(d-1)p}2}(\mathbf{d}(u(t))+\Vert Q\Vert _{\dot{H}_{a}^{1}}^{2})^{%
\frac{p}{4}}+R^{-2} \\
& \lesssim \bigl(R^{-\frac{(d-1)p}2}+R^{-2}\big)\mathbf{d}(u(t)),
\end{align*}%
where in the last step we used the fact that $\frac{p}{4}<%
\frac{1}{d-2}\leq 1$. Combining both cases yields \eqref{441}.

With \eqref{441} and \eqref{Vtt}, we deduce the concavity of $V_{R_0}(t)$:
\begin{equation}
\partial _{tt}V_{R_0}(t)\leq (4-pd)\mathbf{%
d}(u(t))<0, \label{Vsup}
\end{equation}%
which further implies 
\begin{equation}
\partial _{t}V_{R_0}(t)>0,\ \forall t\geq 0.  \label{Vr mono}
\end{equation}%
Additionally, from \eqref{VR}, we have $\partial_t V_{R_0}(t)\le CR_0^2\mathbf{d}(u(t))$. Integrating \eqref{Vsup} over $[t,T]$ gives
\begin{equation*}
\int_{t}^{T}\mathbf{d}(u(s))ds\lesssim \int_{t}^{T}-\partial
_{tt}V_{R_0}(s)ds\leq \partial _{t}V_{R_0}(t)-\partial _{t}V_{R_0}(T)\leq \partial
_{t}V_{R_0}(t)\lesssim \mathbf{d}(u(t)).
\end{equation*}%
 Finally, \eqref{531} is obtained by taking $T\rightarrow \infty $ and applying Gr\"{o}nwall's inequality.
 
\section{Appendix: Spectral Analysis on the linearization at the ground state}\label{S:App}

In the notations of Section \ref{S:groundS}, 
%and \ref{S:uniqueness}, for $b_0\in S^0$, i.e. $q(b_0,r)$ is a 
let $Q(r)$ denote ground state solution and 
% for \eqref{12} thus satisfies the asymptotics in \eqref{18} and \eqref{64}. For simplicity of the discussion, in this part, we denote 
\begin{align*}
%Q(r)=q(b_0,r),\; \tilde q(r)=q_b(b_0,r), \;
Q_1(r)
%= \frac \p{\p \lambda} \big( \lambda^{\frac 2p} Q(\lambda  r)\big)|_{\lambda=1} 
=\tfrac 2pQ(r)+rQ_r(r).
\end{align*}

Let $L_1$ and $L_2$ be the operators defined in \eqref{l1l2}, and we recall the following equations:
\begin{equation}\label{79.1}
L_1 \tq:=0,\; 
L_1 Q_1=-2Q, \:
L_2 Q=0.
\end{equation}
The goal of this section is to prove the following Lemma:
\begin{lemma}\label{lm: l2} 
 Let $p\in (\frac 4d, \frac 4{d-2})
%\cap[1,\infty)
$.  Then the following properties hold. 
%for $L_1$ and $L_2$: 

(i) In $H^1(\R^d)$, $L_1$ and $L_2$ satisfy
%has exactly one negative direction and
\begin{align}\label{524}
\ker(L_1)=\{0\}, \quad \ker(L_2)=\spa \{Q\}.
\end{align}

(ii) 
%$\ker(L_2)=\spa \{Q\}$,  and 
$L_1, L_2$ satisfy uniform ellipticity: there exist $0<c<C<\infty$ such that for all $v\perp_{\la+1}Q$ (i.e., $\langle v, (\la+1)Q\rangle=0)$,
\begin{align*}
c\|v\|_{H^1}^2\le \langle L_i v, v\rangle \le C\|v\|_{H^1}^2, \quad i=1,2.
\end{align*}

(iii) For $L=\begin{pmatrix} L_1&0\\0&L_2\end{pmatrix}$ and $J=\begin{pmatrix}0&1\\-1&0\end{pmatrix}$, 
 \begin{gather*}
 \ker{(JL)}=\spa\bigl\{\begin{pmatrix}0\\Q\end{pmatrix}\bigr\}, 
 %\textit{ for }k=1, 
 \; \text{ and }
 \ker{(JL)^k}=\spa\bigl\{\begin{pmatrix}Q_1\\0\end{pmatrix}, \begin{pmatrix}0\\Q\end{pmatrix}\bigr\}, \textit{ for }k\ge 2. 
 \end{gather*} 
\end{lemma} 
\begin{proof}

 To prove \eqref{524} for $L_1$, we decompose $u\in H^1(\R^d)$ using spherical harmonics. Let $\{Y_j(\theta)\}_{j=0}^\infty$ be the spherical harmonics satisfying 
\begin{equation*}
- \Delta_{\mathbb S^{d-1}} Y_j(\theta)=\mu_j Y_j(\theta), \ \textit{ with } \|Y_j\|_{L^2 (\mathbb S^{d-1})} =1, 
\end{equation*}
where,
\[
0=\mu_0<\mu_1\le \mu_2\le \cdots\to \infty, \; \mu_1=d-1, \; Y_0= (C_{d-1})^{-\frac 12} \triangleq \big(\area (\mathbb S^{d-1})\big)^{-\frac 12}.
\]
For $u=\sum_{j=0}^\infty f_j(r) Y_j(\theta)$, 
\[L_1 u=\sum_{j=0}^\infty \Big( \bigl(\mathcal L_{a+\mu_j}+1-(p+1)Q^p\bigr)f_j(r) \Big) Y_j(\theta).\]
By orthogonality between $Y_j(\theta)$, \eqref{524} for $L_1$ reduces to showing that for each $j\ge 0$,  \begin{align}\label{1023}
(\mathcal L_{a+\mu_j}+1-(p+1)Q^p)G=0
\end{align}
has no non-trivial solution $G\in H_{rad}^1(\R^{d})$. We proceed case by case. 

{\it Case 1.} $j=0$. Assume for contradiction that $G\in H_{rad}^1(\R^d)$ solves
\begin{align}
\label{a531}
L_1 G=(\la+1-(p+1)Q^p)G=0.
\end{align}
Since $L_1$ has only one negative direction (by Lemma \ref{lm: existence}),
%the operator $L_2$ has only one negative direction, equation \eqref{531} implies that 
$0$ is the second eigenvalue of $L_1$, consequently, $G(r)$ must change sign exactly once at some $r_0>0$. 

Define $v(r)=Q(r)-\tfrac p2Q(r_0)^p Q_1(r)$, a direct computation yields: 
\begin{align*}
L_1 v(r)=pQ(r)(Q(r_0)^p-Q(r)^p):=f(r).
\end{align*}
Since $Q$ is decreasing, $f(r)$ changes sign exactly once at $r_0$. Thus,
\[
0= \langle v, L_1 G\rangle = \langle f, G\rangle= C_{d-1}\int_0^\infty f(r)G(r)r^{d-1}dr\neq 0,
\]
which is a contradiction. 
% the orthogonality between the range and kernel for self-adjoint operators. 
Hence $G=0$. \\

{\it Case 2.} $j=1, 2, \ldots, d$, and $\mu_j=d-1$. 
Assume $G\in H_{rad}^1$ solves
\[
\bigl(\mathcal L_{a+d-1}+1-(p+1)Q^p\bigr)G=0. 
\]
We first show $G$ does not change sign. For any $v=\sum_{j=0}^{\infty}v_j(r)Y_j(\theta) \ne 0,$
we evaluate 
\begin{align*}
&\langle (\mathcal L_{a+d-1}+1-(p+1)Q^p)v, v\rangle\\
&=\sum_{j=0}^\infty \langle (\mathcal L_{a+d-1}+1-(p+1)Q^p)v_j(r), v_j(r)\rangle+\sum_{j=1}^\infty \mu_j\int_{\R^d}\frac{ |v_j(x)|^2}{|x|^2} dx \\
&= \sum_{j=0}^\infty\langle L_1 \bigl (v_j(r)Y_1(\theta)\bigr), v_j(r)Y_1(\theta)\rangle
 +\sum_{j=1}^\infty \mu_j\int_{\R^d} \frac{|v_j(x)|^2 }{|x|^2}dx \ge 0,
\end{align*}
where the non-negativity of the first term follows from Lemma \ref{lm: existence} and the orthogonality between $v_j(r)Y_1(\theta)$ and $Q(r)$. Therefore $0$ is the first eigenvalue of $\mathcal L_{a+d-1}+1-(p+1)Q^p$ and
$G>0$.

Next, using the radial equations of $G$ and $Q'$:
\begin{gather}
-G''-\tfrac {d-1}r G'+G+\tfrac{a+d-1}{r^2}G-(p+1)Q^p G=0,\label{814}\\
-Q'''-\tfrac {d-1}r Q''+Q'+\tfrac{a+d-1}{r^2}Q'-\tfrac{2a}{r^3}Q-(p+1)Q^pQ'=0.\label{815}
\end{gather}
and computing $[\eqref{814}\cdot r^{d-1}Q'-\eqref{815}\cdot r^{d-1}G]$ yields: 
\[r^{d-1}Q'''G+(d-1)r^{d-2}Q''G-r^{d-1}Q'G''-(d-1)r^{d-2}Q'G'+{2a}r^{d-4}QG=0.\]
This equation can be rewritten as:
\begin{align}\label{825}
\tfrac d{dr}[r^{d-1}(Q''G-Q'G')]+{2a}r^{d-4} QG=0. 
\end{align}
On the one hand, from the fact $Q>0, G>0$ and $a<0$, \eqref{825} implies that $r^{d-1}(Q''G-Q'G')(r)$ is monotone increasing in $r$. On the other hand, using the exponential decay from Lemma \ref{lm: infinity}, we have
\[\lim_{r\to \infty} r^{d-1}(Q''G-Q'G')=0.\]
Therefore 
\[(Q''G-Q'G)(r)<0, \; \forall r>0.\]
Dividing both sides of this inequality by $G^2$ gives
\[\bigl (\tfrac{Q'}G\bigr)'<0,\; \forall r>0.\]
However, from Lemma \ref{lm: infinity}, Lemma \ref{lm: q} and Corollary \ref{cor: g}, we have
\begin{align*}
\tfrac {Q'(r)}{G(r)}\to -\infty,\textit{ as } r\to 0^+,\;
\tfrac{Q'(r)}{G(r)}\to C<0,\textit{ as } r\to \infty. 
\end{align*}
leading to a contradiction. Thus, $G=0$. 

{\it Case 3.} $j>d$, $\mu_j>d-1$. 
For any nontrivial function $G(x)\in H_{rad}^1(\R^d)$,
\begin{align*}
\langle (\mathcal L_{a+\mu_j}+1-(p+1)Q^p)G(r), G(r)\rangle&=\langle(\mathcal L_{a+d-1}+1-(p+1)Q^p)G(r),G(r)\rangle\\
&+(\mu_j-d+1)\int_{\R^d}\frac{|G(x)|^2}{|x|^2} dx>0. 
\end{align*}
So equation \eqref{1023} has only trivial radial solutions in this case. Combining all cases yields \eqref{524} for $L_1$.

We turn to proving the uniform ellipticity (ii) and focus only on the lower bound estimates, as the upper bound follows
directly from H\"older inequality and Sobolev embedding.
Define the symmetric bounded linear operator $K: L^2(\R^d) \to L^2(\R^d)$ by
\[K=(\la+1)^{-\frac 12}Q^p(\la+1)^{-\frac 12},\]
which yields the factorization 
\[I-(p+1)K=(\la+1)^{-\frac 12}L_1(\la+1)^{-\frac 12}.\]
Our goal reduces to proving  
\[\langle (I-(p+1)K)u, u\rangle \ge c\|u\|_{L_x^2}, \; \text{for all } \langle u, (\la+1)^{\frac 12}Q\rangle =0.\]
From 
Lemma \ref{lm: l2}(i)  
 and Lemma \ref{lm: existence}(3), we know $\ker (I-(p+1)K)=\{0\} $ and $I-(p+1)K$ has only one negative direction $(\la+1)^{\frac 12}Q$. 
Assuming $K$ is compact (to be verified), the eigenvalues $\{\mu_j\}_{j=1}^\infty$ of $I-(p+1)K$ satisfy:
\begin{align*}
\mu_1<0<\mu_2\le \mu_3\le\cdots\to 1.
\end{align*}
The uniform ellipticity follows with the constant $c=\mu_2>0$. 
onential decay of $Q$ yields

We now prove that $K: L^2(\R^d)\hookrightarrow L^2(\R^d)$ is compact.  
Let $\{f_n\} \subset L^2(\R^d)$ be a bounded sequence. Then  
$\{(\la+1)^{-\frac12} f_n\}$ is bounded in $H^1(\R^d)$.  

Let $\varepsilon > 0$ be a small number to be determined later. By the compact embedding 
$H^1 \hookrightarrow L^{\frac{2d}{d-2+\varepsilon}}$ on bounded domains and a diagonal argument, 
there exists a subsequence $\{f_{n_j}\}$ such that 
$(\la+1)^{-\frac12} f_{n_j}$ converges in 
$L^{\frac{2d}{d-2+\varepsilon}}(B_{\R^d}(0,R))$ for every $R > 0$.  

Since $\phi \to Q\phi$ is a bounded operator from $L^{\frac{2d}{d-2+\varepsilon}}$ to $L^{\frac{2d}{d+2}}$ on any domain in $\R^d$ due to the H\"older’s inequality for sufficiently small $\eps$:
\begin{align*}
\|Q^p \phi \|_{L^{\frac{2d}{d+2}}}
\le \|Q\|_{L^{\frac{pd}{2-\varepsilon/2}}}^p 
   \|\phi\|_{L^{\frac{2d}{d-2+\varepsilon}}} 
\lesssim \|\phi\|_{L^{\frac{2d}{d-2+\varepsilon}}},
\end{align*}
we have that $Q^p (\la+1)^{-\frac12} f_{n_j}$ converges in 
$L^{\frac{2d}{d+2}}(B_{\R^d}(0,R))$ for any $R > 0$.  
Here the norm of $Q$ is independent of $R$ due to $\beta > 0$ or
$p < \tfrac{4}{d-2}$ together with the asymptotics of $Q$ in Lemma~\ref{lm: q} and Lemma~\ref{lm: infinity}. The convergence of 
$Q^p (\la+1)^{-\frac12} f_{n_j}$ in 
$L^{\frac{2d}{d+2}}(\R^d)$ then follows quickly from the decay of $Q$ given in Lemma~\ref{lm: infinity}:
\begin{align*}
\|Q^p (\la+1)^{-\frac12} f_{n_j}\|_{L^{\frac{2d}{d+2}}(B_{\R^d}(0,R))^c}
   \le C e^{-cR}. 
\end{align*}

Finally, since $(\la+1)^{-\frac12}$ is a bounded operator from 
$L^{\frac{2d}{d+2}}(\R^d)$ to $L^2(\R^d)$, we conclude that 
$(\la+1)^{-\frac12} Q^p (\la+1)^{-\frac12} f_{n_j}$ converges in 
$L^2(\R^d)$.

For the operator $L_2$, we consider 
 \[(\la+1)^{-\frac 12}L_2(\la+1)^{-\frac 12}=I-K,\]
and denote its eigenvalues by $\lambda_1\le \lambda_2\le \cdots \to 1$. Since for any $u\in L^2_x(\R^d)$, 
\[\langle (I-K)u, u\rangle >\langle (I-(p+1)K)u, u\rangle,\]
we know the eigenvalues satisfy $\lambda_j>\mu_j$. This together with the fact that $(\la+1)^{\frac 12}Q\in \ker(I-K)$ and $\mu_2>0$ implies that $\lambda_1=0$ and $\lambda_2>0$. Thus both $\ker(I-K)$ and $\ker (L_2)$ are one dimensional 
\[
\ker(I-K) = \text{span}\{(\la+1)^{\tfrac{1}{2}}Q\}, \ \ker(L_2) = \text{span}\{Q\}.
\]
The uniform ellipticity of $L_2$ then follows with the constant $c=\lambda_2>0$.

Finally, we prove the structure of the generalized kernel of $JL$. The case $k=1$ follows immediately from the kernel properties of $L_1, L_2$. Now for any $v=v_1+iv_2\in H^1(\R^d)$ satisfying
\[(JL)^2 v=0, \]
we obtain the system:
\[L_1v_1=cQ,\; L_2v_2=0, \]
for some constant $c\neq 0$. The solutions are $v_1=c_1Q_1$ and $v_2=c_2Q$ for some constants $c_1, c_2$.  Thus $\ker\{(JL)^2\}=\text{span}\{Q_1, iQ\}$.

We now examine the case $k=3$. Suppose there exists $v\in \ker\{(JL)^3 \setminus \ker\{(JL)^2$. 
%H^1(\R^d)$ with $v\notin \ker\{(JL)^2\}$ such that $(JL)^3v=0$. 
Since $JLv\in \ker\{(JL)^2\} \setminus \ker\{(JL)$, there must exist constants $c_1\neq 0$ and $c_2\in \R$ such that 
\begin{align}\label{1116}
JLv=c_1Q_1+c_2 iQ. 
\end{align}
Without loss of generality, we normalize to $c_1=1$. 

We now compute the quadratic term in two different ways. 
First, using \eqref{1116} we obtain: 
\begin{align*}
\langle LQ_1, Q_1\rangle&=\langle L(JLv-c_2iQ), JLv-c_2iQ\rangle =\langle LJLv, JL v\rangle-\langle LJL v, c_2 iQ\rangle\\ 
&=\langle LJL v, JLv\rangle =-\langle (JL)^2 v, Lv\rangle =-\langle (JL)^3v, Jv\rangle=0. 
\end{align*}
However, from \eqref{79.1} we directly compute:
\begin{align}\label{503}
\langle LQ_1, Q_1\rangle =\langle -2Q, Q_1\rangle =(d-\tfrac 4p)\int_{\R^d} Q^2 (x) dx> 0,
\end{align}
since $p>\frac 4d$. This contradiction shows that $\ker\{(JL)^3\}=\ker\{(JL)^2\}$. It also implies  $\ker\{(JL)^{k+1}\}= \ker\{(JL)^k\}$ for general $k \ge 2$ following from a simple linear algebra argument.  

This completes the proof of (iii) and Lemma \ref{lm: l2}. 
\end{proof}


\begin{thebibliography}{99}

%\bibitem{R: Aubin} T. Aubin, \emph{Probl\'{e}mes isop\'{e}rim\'{e}triques et
%espaces de Sobolev.} J. Diff. Geom. \textbf{11} (1976), 573--598.
%%\msn{0448404}

%\bibitem{R: Gerard critical wave} H. Bahouri and P. G\'{e}rard, \emph{High
%frequency approximation of solutions to critical nonlinear wave equations.}
%Amer. J. Math. \textbf{121} (1999), no. 1, 131--175. %MR1705001

%\bibitem{R: Berger laplace} M. Berger, P. Gauduchon, and E. Mazet, \emph{Le spectre d'une vari\'{e}t\'{e} riemannienne.}
%2nd ed., Lecture Note in Mathematics, \textbf{194}, Springer-Verlag, New York/Berlin, 1971.



%\bibitem{R: Smith} M.D. Blair, H.F. Smith, and C.D. Sogge, \emph{Strichartz estimates and nonlinear Schr\"{o}dinger equation on manifolds with boundary.}
% Math. Ann. \textbf{354} (2012), 1397-1430.

%\bibitem{R: Bliss} G. A. Bliss, \emph{An integral inequality.} J. London
%Math. Soc. \textbf{5} (1930), 40--46.

%\bibitem{Bour98} J. Bourgain, \emph{Scattering in the energy space and below for 3D NLS.} Journal D'Analyse Mathematique. \text{75} (1998), 267-297.

%\bibitem{R: Bourgain finite app}  J. Bourgain, \emph{Approximation of solutions of the cubic nonlinear Schr\"{o}dinger equations by finite-dimensional equations and nonsqueezing properties.}
% Internat. Math. Res. Notices, 1994, no. 2, 79--88.

\bibitem{Bourgain} J. Bourgain, \emph{Global well-posedness of
defocusing 3D critical NLS in the radial case}. J. Amer. Math. Soc. \textbf{12} (1999), 145--171.

%\bibitem{R: Bourgain 1} J. Bourgain,\emph{\ New global well-posedness
%results for nonlinear Schr\"{o}dinger equations. }AMS Colloquium
%Publications, \textbf{46}, 1999.

%\bibitem{R: BrezisLieb} H.~Br\'{e}zis and E.~Lieb, \emph{A relation between
%pointwise convergence of functions and convergence of functionals.} Proc.
%Amer. Math. Soc. \textbf{88} (1983), 486--490.

\bibitem{BPSTZ03} N. Burq, F. Planchon, J. Stalker, and A. S.
Tahvildar-Zadeh, \emph{Strichartz estimates for the wave and Schr\"{o}dinger
equations with the inverse-square potential.} J. Funct. Anal. \textbf{203}
(2003), 519--549.

%\bibitem{R: Strichartz inverse 1} N. Burq, F. Planchon, J. Stalker, and A.
%S. Tahvildar-Zadeh, \emph{Strichartz estimates for the wave and Schr\"{o}%
%dinger equations with potentials of critical decay.} Indiana Univ. Math. J.
%\textbf{53} (2004), 1665--1680.

\bibitem{CFR22} L. Campos, L. G. Farah, and S. Roudenko, \emph{Threshold solutions for the nonlinear Schr\"odinger equation.} Rev. Mat. Iberoam, \textbf{38} (2022), no. 5, 1637--1708.

%\bibitem{R: Caze} T. Cazenave, \emph{Semilinear Schr\"{o}dinger equations.}
%Courant Lecture Notes in Mathematics, Vol. 10. New York: New York University
%Courant Institute of Mathematical Sciences, 2003. ISBN: 0-8218-3399-5.


\bibitem{C72}
{ C. V. Coffman,} \emph{Uniqueness of the ground state solution
for $\Delta u-u+u\sp{3}=0$ and a variational characterization
of other solutions.}
{  Arch. Rational Mech. Anal.} {\bf 46} (1972), 81--95.

\bibitem{CJ93}
{\sc C. B. Clemons and C. Jones,}
A geometric proof of the Kwong-McLeod uniqueness result.
{\it SIAM J. Math. Anal.} {\bf 24} (1993),  436--443.

%\bibitem{Cav} T. Cazenave, \emph{Semilinear Schr\"odinger equations.} Courant Lecture Notes in Mathematics,
%Vol. 10. New York: New York University Courant Institute of Mathematical Sciences, 2003. ISBN: 0-8218-3399-5.


%\bibitem{R: CW Cauchy NLS} T. Cazenave and F. Weissler, \emph{The Cauchy
%problem for the critical nonlinear Schr\"{o}dinger equation in }$H^{s}$\emph{%
%.} Nonlinear Anal. \textbf{14} (1990), 807--836.

%\bibitem{R: CK lemma} M. Christ and A. Kiselev, \emph{Maximal functions
%associated to filtrations.} J. Funct. Anal. \textbf{179} (2001), 409--425.

%\bibitem{R: fractional chain rule} M.~Christ and M.~Weinstein, \emph{%
%Dispersion of small amplitude solutions of the generalized Korteweg--de
%Vries equation.} J. Funct. Anal. \textbf{100} (1991), 87--109.

%\bibitem{R: CKSTT annal} J. Colliander, M. Keel, G. Staffilani, H. Takaoka,
%and T. Tao, \emph{Global well-posedness and scattering for the
%energy-critical nonlinear Schr\"{o}dinger equation in $\mathbb{R}^{3}$.}
%Annals of Math. \textbf{167} (2008), 767--865.

%
%\bibitem{R: EB Davis} E. B. Davies, \emph{Spectral theory and differential operators.}
%Cambridge Studies in Advanced Mathematics, \textbf{42}, Cambridge University Press, Cambridge, 1995.




\bibitem{DHR08}T. Duyckaerts, J. Holmer, and S. Roudenko, \emph{Scattering for the non-radial 3D cubic nonlinear Schr\"{o}dinger equation.}
Math. Res. Lett., \textbf{15} (2008), no. 6, 1233-1250.

\bibitem{DM08} T. Duyckaerts and F. Merle, \emph{Dynamic of threshold solutions for energy-critical NLS.}
Geometric And Functional Analysis. \textbf{08} (2008), 1787-1840.

\bibitem{DM08a} T. Duyckaerts and F. Merle, \emph{Dynamics of threshold solutions for energy-critical Wave equation.}
International Mathematics Research Papers 2008, rpn002.

\bibitem{DR10} T. Duyckaerts and S. Rodenko, \emph{Threshold solutions for the focusing 3D cubic Schr\"odinger equation.} Rev. Mat. Iberoamericana, \textbf{26} (2010), no. 1, 1--56.

%\bibitem{R: Fang general} D. Fang, J. Xie and T. Cazenave, \emph{Scattering for the focusing energy-subcritical nonlinear Schr\"{o}dinger equation.}
%Sci. China Math., \textbf{54} (2011), 2037--2062.

%\bibitem{DFVV}
%P. D'Ancona, L. Fanelli, L. Vega, and N. Visciglia, \emph{Endpoint Strichartz estimates for the magnetic Schr\"odinger equation.}
%J. Funct. Anal. \textbf{258} (2010), 3227--3240.

%\bibitem{FFFP} L. Fanelli, V. Felli, M. A. Fontelos, and A. Primo,
%\emph{Time decay of scaling critical electromagnetic Schr\"{o}dinger flows.} Commun. Math. Phys. \textbf{324} (2013), 1033--1067.

%\bibitem{R: Fanelli} L. Fanelli, V. Felli, M. A. Fontelos, and A. Primo,
%\emph{Time decay of scaling critical electromagnetic Schr\"{o}dinger flows.}
%Comm. Math. Phys. \textbf{324} (2013), no. 3, 1033--1067.

%\bibitem{GV2010}  J. Ginibre and G. Velo, \emph{Quadratic Morawetz inequalities and asymptotic completeness in the energy space for nonlinear
%Schr\"{o}dinger and Hartree equations}, Quart. Appl. Math. \textbf{68} (2010), 113--134.

%\bibitem{GVV}
%M. Goldberg, L. Vega, and N. Visciglia,  \emph{Counterexamples of Strichartz inequalities for Schr\"{o}dinger equations with repulsive potentials.}
%Int. Math Res. Not., 2006, 13927 (2006).
%
%\bibitem{R: Grill} G. Grillakis, \emph{On nonlinear Schr\"{o}dinger
%equations.} Comm. PDE \textbf{25} (2000), 1827--1844.

%\bibitem{R: Gromov} M. Gromov, \emph{Pseudoholomorphic curves in symplectic manifolds.} Invent. Math. \textbf{82} (1985), no. 2, 307--347.

%\bibitem{R: Holmer} J. Holmer and S. Roudenko, \emph{A sharp  condition for scattering of the radial 3D cubic nonlinear Schr\"{o}dinger equation.}
%Communications in Mathematical Physics, \textbf{282} (2008), no. 2, 435-467.

%\bibitem{R: Iv} O. Ivanovici, \emph{On the Schr\"{o}dinger equation outside strictly convex obstacles.}
% Anal. PDE, \textbf{3} (2010), no. 3, 261--293.

%\bibitem{R: IonPaus1} A. D. Ionescu and B. Pausader, \emph{Global
%well-posedness of the energy-critical defocusing NLS on $\mathbb{R}\times
%\mathbb{T}^{3}$.} Comm. Math. Phys. \textbf{312} (2012), no. 3, 781--831.
%%\msn{2925134}
%
%\bibitem{R: IonPaus} A. D. Ionescu and B. Pausader, \emph{The
%energy-critical defocusing NLS on $\mathbb{T}^{3}$.} Duke Math. J. \textbf{%
%161} (2012), no. 8, 1581--1612. %\msn{2931275}.
%
%\bibitem{R: KSWW} H. Kalf, U. W. Schmincke, J. Walter, and R. W\"{u}st,
%\emph{\ On the spectral theory of Schr\"{o}dinger and Dirac operators with
%strongly singular potentials. In Spectral theory and differential equations.}
%182--226. Lect. Notes in Math. \textbf{448} (1975) Springer, Berlin.

\bibitem{KT98} M. Keel and T. Tao, \emph{Endpoint
Strichartz estimates.} Amer. J. Math. \textbf{120} (1998), 955--980.

%\bibitem{R: Kenig half} C. Kenig and F. Merle, \emph{Scattering for $\dot{H}^{\frac{1}{2}}$ bounded solutions to the cubic, defocusing NLS in 3 dimensions.}
%Trans. Amer. Math. Soc., \textbf{362} (2010), 1937-1962.

\bibitem{KM06} C. Kenig and F. Merle, \emph{Global
well-posedness, scattering, and blow-up for the energy-critical focusing
nonlinear Schr\"{o}dinger equation in the radial case.} Invent. Math.
\textbf{166} (2006), 645--675.

%\bibitem{R: Keraani} S. Keraani, \emph{On the blow up phenomenon of the
%critical nonlinear Schr\"{o}dinger equation.} J. Funct. Anal. \textbf{235}
%(2006), no. 1, 171--192. %MR2216444

%\bibitem{R: KKSV:gKdV} R. Killip, S. Kwon, S. Shao, and M. Visan, \emph{On
%the mass-critical generalized KdV equation.} DCDS-A \textbf{32} (2012),
%191--221. %\msn{2837059}

\bibitem{KMVZZ18} R. Killip, C. Miao, M. Visan, J. Zhang,
and J. Zheng, \emph{Sobolev spaces adapted to the Schr\"{o}dinger operator
with inverse-square potential.} Math. Z. {\bf 288} (2018), no. 3-4, 1273-1298.

\bibitem{KMVZZ17} R. Killip, C. Miao, M. Visan, J. Zhang,
and J. Zheng, \emph{The energy-critical NLS with inverse-square potential.}
Discrete Contin. Dyn. Syst. {\bf 37} (2017), no. 7, 3831-3866.

\bibitem{KV10} R. Killip and M. Visan, \emph{The
focusing energy-critical nonlinear Schr\"{o}dinger equation in dimensions
five and higher.} Amer. J. Math. \textbf{132} (2010), no. 2, 361--424.
MSN{2654778}

%\bibitem{R: Clay note} R. Killip and M. Visan, \emph{Nonlinear Schr\"{o}%
%dinger equations at critical regularity.} In \textquotedblleft Evolution
%equations\textquotedblright , 325--437, Clay Math. Proc., \textbf{17}. Amer.
%Math. Soc., Providence, RI, 2013.

%\bibitem{R: KV:quintic} R. Killip and M. Visan, \emph{Global well-posedness
%and scattering for the defocusing quintic NLS in three dimensions.} Anal.
%PDE, \textbf{5} (2012), no. 4, 855--885.
%\bibitem{R: KVZ12 obstacle Harmonic} R. Killip, M. Visan, and X. Zhang, \emph{Riesz transforms outside a convex obstacle.}
%Int. Math. Res. Not., 2016, no. 19, 5875-5921.
%
%\bibitem{R: KVZ12 obstacle} R. Killip, M. Visan, and X. Zhang, \emph{Quintic
%NLS in the exterior of a strictly convex obstacle.} Amer. J. Math., \textbf{%
%138} (2016), no. 5, 1193-1346.
%
%\bibitem{R: KVZ12 obstacle sub} R. Killip, M. Visan, and X. Zhang,
%\emph{The focusing cubic NLS on exterior domains in three dimensions.}
% Applied Mathematics Research eXpress, \textbf{1} (2016), 146-180.
%
%\bibitem{R: KVZ non-sqz} R. Killip, M. Visan, and X. Zhang,
%\emph{Finite-dimensional approximation and non-squeezing for the cubic nonlinear
%Schr\"{o}dinger equation on $\mathbb{R}^{2}$. } Preprint \texttt{arXiv:1606.07738.}
%%%%%%%%%%%%%%%%%%%%%%%%%%%%%%%%%%%%%%%%%%%%%%%%%%%%%%%%%%%%%%%%%%%%%%%%%%%%%%%%%%%%%%%




\bibitem{KMVZ17} R. Killip, J. Murphy, M. Visan, and J. Zheng, \emph{The focusing cubic NLS with inverse-square potential in three space dimensions.} Differential Integral Equations, \textbf{30} (2017), no. 3-4, 161--206.

\bibitem{Kwong}
{\sc M. K. Kwong,}
Uniqueness of positive solutions of $\Delta u-u+u^p=0$ in $\mathbb{R}^n$.
{\it Arch. Rational Mech. Anal.} {\bf 105} (1989), 243--266.

\bibitem{KL92}
{\sc M. K. Kwong and Y. Li,} Uniqueness of radial solutions of semilinear elliptic equations.
{\it Trans. Amer. Math. Soc.} {\bf 333} (1992), 339--363.

\bibitem{KZ91}
{\sc M. K. Kwong and L. Zhang,}
Uniqueness of the positive solution of $\Delta u+f(u)=0$ in
an annulus. {\it Diff. Int. Eqs.} {\bf 4} (1991), 583--596.

%\bibitem{R: KVZ non-sqz R1} R. Killip, M. Visan, X. Zhang, \emph{Symplectic non-squeezing for the cubic NLS on the line.}
%Preprint \texttt{arXiv:1606.09467}.


\bibitem{LZ09} D. Li and X. Zhang, \emph{Dynamics for the energy critical nonlinear Schr\"{o}dinger equation in high dimensions.}
J. Funct. Anal., \textbf{256} (2009), no. 6, 1928-1961.

\bibitem{LZ11} D. Li and X. Zhang, \emph{Dynamics for the energy critical nonlinear Wave equation in high dimensions.}
Trans. AMS., \textbf{363} (2011), 1137–1160.

\bibitem{LZ22} Z. Lin and C. Zeng, \emph{Instability, index theorem, and exponential trichotomy for linear Hamiltonian PDEs.} Mem. Am. Math. Soc., \textbf{275} (2022), no. 136.



\bibitem{LMM18}J. Lu, C. Miao, and J. Murphy, \emph{Scattering in $H^1$ for the intercritical NLS with an inverse-square potential.} J. Differential Equations,  \textbf{264} (2018), no. 5, 3174–3211. 


\bibitem{MMMZ25}Z. Ma, C. Miao, J. Murphy and J. Zheng, \emph{Dynamics of subcritical threshold solutions for the 4d energy-critical NLS.} Preprint, arXiv: 2508.02608.


\bibitem{McLeod}
{  K. McLeod,}
\emph{Uniqueness of positive radial solutions of $\Delta u+f(u)=0$ in $\mathbb{R}^n$, II.}
{\it Tran. Amer. Math. Soc.} {\bf 339}  (1993), 495--505.

\bibitem{MS81}
{  K. McLeod and J. Serrin,} \emph{Uniqueness of solutions of semilinear Poisson equations.}
{\it Proc. Nat. Acad. Sci. USA}, {\bf 78} (1981), 6592--6595.


\bibitem{MS87}
{ K. McLeod and J. Serrin,}
\emph{Uniqueness of positive radial solutions of $\Delta u+f(u)=0$ in $\mathbb{R}^n$.}
{  Arch. Rational Mech. Anal.}
{\bf 99} (1987), 115--145.

\bibitem{MMZ23} C. Miao, J. Murphy, and J. Zheng, \emph{Threshold scattering for the focusing NLS with a repulsive potential.} Indiana Univ. Math. J. \textbf{72} (2023), no. 2, 409–453. 
 
\bibitem{MNN21} D. Mukherjee, P. T. Nam and P.-T. Nguyen, \emph{Uniqueness of ground state and minimal-mass blow-up solutions for focusing NLS with Hardy potential.} J. Funt. Anal., \textbf{281} (2021), no. 5, 109092.

%\bibitem{R: LiebLoss} E. Lieb and M. Loss, \emph{Analysis.} Second edition.
%Graduate Studies in Mathematics, \textbf{14}. American Mathematical Society,
%Providence, RI, 2001. %\msn{1817225}

%\bibitem{MZZ} C. Miao, J. Zhang and J. Zheng, \emph{Strichartz estimates for wave equation with
%inverse-square potential.} Communications in Contemporary Mathematics, \text{15} (2013), DOI: 10.1142/S0219199713500260.

%\bibitem{R: partial Banach fixed point} S.G. Matthews, \emph{Partial metric
%topology.} Proc. 8th Summer Conference on General Topology and Applications,
%Ann. New York Acad. Sci., \textbf{728} (1994), 183--197.


%\bibitem{R: Miao Hartree} C. Miao, Y. Wu and G. Xu, \emph{Dynamics for the focusing, energy-critical nonlinear Hartree equation.}
%Forum Mathematicum, \textbf{27} (2015), no. 1, 373-447.


%\bibitem{R: PausTzW} B. Pausader, N. Tzvetkov, and X. Wang, \emph{Global
%regularity for the energy-critical NLS on $\mathbb{S}^{3}$.} Ann. Inst. H.
%Poincar\'{e} Anal. Non Lin\'{e}aire \textbf{31} (2014), no. 2, 315--338.
%%MR3181672

%\bibitem{PSS} F. Planchon, J. Stalker and A. S. Tahvildar-Zadeh, \emph{ $L^p$ estimates for the wave equation with the inverse-square potential.}
%Discrete Contin. Dynam. Systems, \text{9} (2003), 427-442.

%\bibitem{R: PSS1 dispersive of inverse} F. Planchon, J. Stalker, and A. S.
%Tahvildar-Zadeh, \emph{Dispersive estimates for wave equation with the
%inverse-square potential.} Discrete Contin. Dynam. Systems, \textbf{9}
%(2003), 1387--1400.

%\bibitem{R: Rey} O. Rey, \emph{The role of the Green’s function in a nonlinear elliptic equation involving the critical
%Sobolev exponent.} J. Funct. Anal., \textbf{89} (1990), no. 1, 1-52.

%\bibitem{RS4}  M. Reed and B. Simon, \emph{Methods of modern mathematical physics. IV. Analysis of operators.} Academic Press, New York-London, 1978.
%MR0493421

%\bibitem{RS} I. Rodnianski and W. Schlag,
%\emph{Time decay for solutions of Schr\"odinger equations with rough and time-dependent potentials.} Invent. Math., \text{155} (2004), 451-513.

%\bibitem{R: 1+4 Monica} E. Ryckman and M. Visan, \emph{Global well-posedness
%and scattering for the defocusing energy-critical nonlinear Schr\"{o}dinger
%equation in $\mathbb{R}^{1+4}$.} Amer. J. Math. \textbf{129} (2007), 1--60.

%\bibitem{Schlag} W. Schlag, \emph{ Dispersive estimates for Schr\"odinger operators: a survey.} Ann. of Math., \text{163} (2007), 255-285.

%\bibitem{SSS1} W. Schlag, A. Soffer and W. Staubach, \emph{Decay for the wave and Schr\"odinger evolutions on manifolds with conical ends, I.}
%Trans. Amer. Math. Soc., \text{362} (2010), 19-52.

%\bibitem{SSS2} W. Schlag, A. Soffer and W. Staubach, \emph{Decay for the wave and Schr\"odinger evolutions on manifolds with conical ends, II.}
%Trans. Amer. Math. Soc., \text{362} (2010), 289-318.

\bibitem{P65}
 S. I. Pohozaev,
\emph{Eigenfunctions of the equation $\Delta u + \lambda f(u) = 0$.}   Soviet Math.,
{\bf 5} (1965), 1408--1411.


\bibitem{SW13} N. Shioji and K. Watanabe, \emph{A generalized Pohozaev identity and uniqueness of positive radial solutions of $\Delta u + g(r)u + h(r)u^p = 0$.} J. Differ. Equ., \textbf{255} (2013), 4448--4475.

\bibitem{SZ23} Q. Su and Z. Zhao, \emph{Dynamics of subcritical threshold solutions for energy-critical NLS.} Dyn. Partial Differ. Equ., \textbf{20} (2023), no. 1, 37--72.

\bibitem{Tang03} {\sc M. Tang,} Uniqueness of positive radial solutions for $\Delta u -u + u^p
= 0$ on an annulus. {\it J. Diff. Eqs.} {\bf 189} (2003), 148--160.

%\bibitem{R: Strauss} W. A. Strauss, \emph{Existence of solitary waves in higher dimensions.}
%   Commun. Math. Phys., \textbf{55} (1977), no. 2, 149–-162.
%
%\bibitem{R: Talenti} G. Talenti, \emph{Best constant in Sobolev inequality.}
%Ann. Mat. Pura. Appl. \textbf{110} (1976), 353--372. %\msn{0463908}

%\bibitem{R: TaoRadial energy} T. Tao, \emph{Global well-posedness and
%scattering for higher-dimensional energy-critical non-linear Schr\"{o}dinger
%equation for radial data.} New York J. of Math. \textbf{11} (2005), 57--80.

%\bibitem{R: stability high d} T. Tao and M.~Visan, \emph{Stability of
%energy-critical nonlinear Schr\"{o}dinger equations in high dimensions.}
%Electron. J. Diff. Eqns. \textbf{118} (2005), 1--28.

%\bibitem{Tit} E. C. Titchmarsh,
%\emph{Eigenfuction expansions associated with second-order differential equations.} University press, Oxford, 1946.

%\bibitem{R: VZ  hardy inverse} J. L. Vazquez and E. Zuazua, \emph{The Hardy
%inequality and the asymptotic behaviour of the heat equation with an
%inverse-square potential.} J. Funct. Anal. \textbf{173} (2000), 103--153.

%\bibitem{R: Monica defocusing high} M. Visan, \emph{The defocusing
%energy-critical nonlinear Schr\"{o}dinger equation in higher dimensions.}
%Duke Math. J. \textbf{138} (2007) 281--374.
%
\bibitem{Y91} E. Yanagida, \emph{Uniqueness of positive radial solutions of $\Delta u + g(r)u + h(r)u^p = 0$ in $\mathbb{R}^n$.} Arch. Ration. Mech. Anal., \textbf{115} (1991), 257--274.

\bibitem{Y20} K. Yang,  \emph{Scattering of the energy-critical NLS with inverse square potential.} J. Math. Anal. Appl., 487(2020), 124006.

\bibitem{Y21} K. Yang,  \emph{Scattering of the focusing energy-critical NLS with inverse square potential in the radial case.}   Commun. Pure Appl. Anal., {\bf 20} (2021), no.1, 77-99.

\bibitem{YZZ22} K. Yang, C. Zeng and X. Zhang, \emph{Dynamics of threshold solutions for energy critical NLS with inverse square potential}. SIAM J. Math. Anal., \textbf{54} (2022), no. 1, 173-219.

\bibitem{YZ22} K. Yang and X. Zhang, \emph{Dynamics of threshold solutions for energy critical NLW with inverse square potential.} Math. Z., \textbf{302} (2022), 353--389.

\bibitem{YZ23} K. Yang and X. Zhang, \emph{Scattering of
the focusing energy-critical NLS with inverse square potential.} Discrete Contin. Dyn. Syst., \textbf{43}(2023), no. 7, 2608-2636.

\bibitem{ZZ20} J. Zhang and J. Zheng, \emph{Strichartz estimates and wave equation in a conic singular space.} Math. Ann., \textbf{376} (2020), 525--581.



 







 
\end{thebibliography}
\end{document}